\documentclass{amsart}
\usepackage{amsmath,amsthm,amssymb}
\usepackage{color}

\usepackage{bm}

\makeatletter
    
    \@addtoreset{equation}{section}
  \makeatother

\newtheorem{definition}{Definition}[section]
\newtheorem{proposition}[definition]{Proposition}
\newtheorem{theorem}[definition]{Theorem}
\newtheorem{lemma}[definition]{Lemma}

\newtheorem{remark}{Remark}[section]

\pagebreak

\title[Nonlinear damped wave equation]{
$L^p$-$L^q$ estimates for the damped wave equation
and the critical exponent for the nonlinear problem with slowly decaying data}
\author[M. Ikeda]{Masahiro Ikeda}
\address[M. Ikeda]{Department of Mathematics,
Faculty of Science and Technology, Keio University,
3-14-1, Hiyoshi, Kohoku-ku, Yokohama, 223-8522, Japan/
Center for Advanced Intelligence Project, RIKEN, Japan
}
\email{masahiro.ikeda@keio.jp/masahiro.ikeda@riken.jp}

\author[T. Inui]{Takahisa Inui}
\address[T. Inui]{
Department of Mathematics, Graduate School of Science,
Osaka University,
Toyonaka, Osaka 560-0043, Japan
}
\email{inui@math.sci.osaka-u.ac.jp}

\author[M. Okamoto]{Mamoru Okamoto}
\address[M. Okamoto]{
Division of Mathematics and Physics, Faculty of Engineering,
Shinshu University,
4-17-1 Wakasato, Nagano City 380-8553, Japan
}
\email{m\_okamoto@shinshu-u.ac.jp}

\author[Y. Wakasugi]{Yuta Wakasugi}
\address[Y. Wakasugi]{
Department of Engineering for Production and Environment,
Graduate School of Science and Engineering,
Ehime University,
3 Bunkyo-cho, Matsuyama, Ehime, 790-8577, Japan
}
\email{wakasugi.yuta.vi@ehime-u.ac.jp}

\date{\today}

\begin{document}

\maketitle
\begin{abstract}
We study the Cauchy problem of the damped wave equation
\begin{align*}
	\partial_{t}^2 u - \Delta u + \partial_t u = 0
\end{align*}
and give sharp $L^p$-$L^q$ estimates of the solution
for $1\le q \le p < \infty\ (p\neq 1)$ with derivative loss.
This is an improvement of the so-called Matsumura estimates.
Moreover, as its application, we consider the nonlinear problem
with initial data in
$(H^s\cap H_r^{\beta}) \times (H^{s-1} \cap L^r)$ with
$r \in (1,2]$, $s\ge 0$, and $\beta = (n-1)|\frac{1}{2}-\frac{1}{r}|$,
and prove the local and global existence of solutions.
In particular, we prove the existence of the global solution with small initial data
for the critical nonlinearity with the power $1+\frac{2r}{n}$,
while it is known that the critical power $1+\frac{2}{n}$ belongs to the blow-up region when $r=1$.
We also discuss the asymptotic behavior of the global solution
in supercritical cases.
Moreover, we present blow-up results in subcritical cases.
We give estimates of lifespan by an ODE argument.
\end{abstract}

\tableofcontents

\section{Introduction}
The damped wave equation
\begin{align*}
	\partial_{t}^2 u - \Delta u + \partial_t u = 0
\end{align*}
is known as a model describing the wave propagation with friction,
and studied for long years.
In particular, for the Cauchy problem
\begin{align}%
\label{ldw}
	\left\{ \begin{array}{ll}
		\partial_{t}^2 u - \Delta u + \partial_t u = 0,
			&(t,x) \in (0,\infty) \times \mathbb{R}^n,\\
		u(0,x) = u_0(x),\ \partial_t u(0,x) = u_1(x),
			&x \in \mathbb{R}^n,
	\end{array}\right.
\end{align}%
where $(u_0, u_1)$ is a given function,
the asymptotic behavior of the solution has been investigated
by many mathematicians after the pioneering work by Matsumura \cite{Ma76}.
Matsumura \cite{Ma76} applied the Fourier transform to \eqref{ldw} and obtained
the formula
\begin{align*}
	u(t,x) = \mathcal{D}(t) (u_0+u_1) + \partial_t \mathcal{D}(t)u_0,
\end{align*}
where $\mathcal{D}(t)$ is defined by
\begin{align*}%
	\mathcal{D}(t) := e^{-\frac{t}{2}}\mathcal{F}^{-1} L(t,\xi) \mathcal{F}
\end{align*}%
with
\begin{align}%
\label{L}
	L(t,\xi) &:= \begin{cases}
	\displaystyle \frac{\sinh ( t \sqrt{1/4-|\xi|^2} )}{\sqrt{1/4 - |\xi|^2}}
		&(|\xi| < 1/2),\\[12pt]
	\displaystyle \frac{\sin ( t \sqrt{|\xi|^2-1/4})}{\sqrt{|\xi|^2 - 1/4}}
		&(|\xi| > 1/2).
	\end{cases}
\end{align}%
Using the above formula, he proved
the so-called Matsumura estimates ($L^p$-$L^q$ estimates)
\begin{align}
\label{mat_est}
	\| u(t) \|_{L^p} &\lesssim \langle t \rangle^{-\frac{n}{2}\left( \frac{1}{q} - \frac{1}{p} \right)}
		\left( \| u_0 \|_{L^q} + \| u_1 \|_{L^q} \right) \\
\nonumber
		&\quad + e^{-t/4} \left( \| u_0 \|_{H^{[\frac{n}{2}]+1}}
				+ \| u_1 \|_{H^{[\frac{n}{2}]}} \right),
\end{align}
where
$1 \le q \le 2 \le p \le \infty, \langle t \rangle:=(1+|t|^2)^{1/2}$,
$[n/2]$ denotes the integer part of $n/2$,
and the notation $f \lesssim g$ stands for
$f \le C g$ with some constant $C>0$.
The first term and the second term in the right-hand side are
corresponding to the low and high frequency part of the solution, respectively.
The estimates \eqref{mat_est} indicate that
the low frequency part of the solution behaves like that of the heat equation
\begin{align}
\label{h}
	\partial_t v - \Delta v = 0.
\end{align}
Here, we recall the well-known $L^p$-$L^q$ estimates for the heat equation
\begin{align}
\label{h_lplq}
	\| \mathcal{G}(t) g \|_{L^p}
	\lesssim t^{-\frac{n}{2}\left( \frac{1}{q} - \frac{1}{p} \right)} \| g \|_{L^q},
\end{align}
where $1 \le q \le p \le \infty$, $g \in L^q(\mathbb{R}^n)$,
and
\[
	\mathcal{G}(t) := \mathcal{F}^{-1} e^{-t|\xi|^2} \mathcal{F}.
\]
Namely, $\mathcal{G}(t)g$ is the solution of \eqref{h} with $v(0) = g$
(see \cite{GiGiSa}).
Also, we see from \eqref{mat_est} that the high frequency part causes
derivative losses like the wave equation
\begin{align}
\label{w}
	\partial_t^2 w - \Delta w = 0.
\end{align}
We also recall the estimates for the wave equation:
\begin{align}
\label{w_lp}
	\| \mathcal{W}(t) g  \|_{L^p} \lesssim \langle t \rangle^{\delta'_p} \| g \|_{H^{\beta-1}_p}
\end{align}
with some constant $\delta'_p > 0$,
where $1 < p < \infty$, $g\in H^{\beta-1}_p(\mathbb{R}^n)$,
$\beta = (n-1)| \frac{1}{2}- \frac{1}{p}|$,
and
\[
	\mathcal{W}(t) := \mathcal{F}^{-1} \frac{\sin (t |\xi|)}{|\xi|} \mathcal{F},
\]
namely, $\mathcal{W}(t)g$ is the solution of \eqref{w} with
$(w, \partial_t w)(0) = (0, g)$
(see \cite{Sj70, Mi80}). 
Here, we set 
$H^s_p(\mathbb{R}^n) := \{ f \in \mathcal{S}' (\mathbb{R}^n) ;
\| f \|_{H^s_p} = \| \langle \nabla \rangle^s f \|_{L^p} < \infty \}$.
However, in contrast to the estimates \eqref{h_lplq} and \eqref{w_lp},
the Matsumura-type estimate \eqref{mat_est}
requires the restriction $q\le 2 \le p$,
and the derivative losses of the high frequency part seems not sharp.
Therefore, we expect that the estimate \eqref{mat_est} can be improved.

Indeed, in the following
we give an improvement of the Matsumura-type estimate \eqref{mat_est}.
Let $\chi_{\le 1}(\nabla)$ and $\chi_{>1}(\nabla)$
be the cut-off Fourier multipliers defined by \eqref{chi} for low and high frequency, respectively. 
Our first result reads as follows. 

\begin{theorem}[$L^p$-$L^q$ estimates]\label{thm_lplq}
Let
$1\le q \le p < \infty$ with $p\neq 1$,
$s_1 \ge s_2$,
and
$\beta = (n-1) | \frac{1}{2} - \frac{1}{p} |$.
Then,
there exists $\delta_p > 0$ such that
\begin{align}%
\label{lplq}
	&\| |\nabla|^{s_1} \mathcal{D}(t) g \|_{L^p}\\
\nonumber
	&\lesssim \langle t \rangle^{-\frac{n}{2}\left( \frac{1}{q} - \frac{1}{p} \right)-\frac{s_1-s_2}{2}}
			\| |\nabla|^{s_2}  \chi_{\le 1}(\nabla) g \|_{L^q}
	+ e^{-\frac{t}{2}} \langle t \rangle^{\delta_p}
		\| |\nabla|^{s_1} \chi_{>1}(\nabla) g \|_{H_p^{\beta -1}},\\
\label{lplqdt}
	&\| |\nabla|^{s_1} \partial_t \mathcal{D}(t) g \|_{L^p}\\
\nonumber
	&\lesssim \langle t \rangle^{%
				-\frac{n}{2}\left( \frac{1}{q} - \frac{1}{p} \right)-\frac{s_1-s_2}{2}-1}
			\| |\nabla|^{s_2}  \chi_{\le 1}(\nabla) g \|_{L^q}
	+ e^{-\frac{t}{2}} \langle t \rangle^{\delta_p}
		\| |\nabla|^{s_1} \chi_{>1}(\nabla) g \|_{H_p^{\beta}},
\end{align}%
for $t> 0$,
provided that the right-hand side is finite.
\end{theorem}

We will prove this theorem in the next section.
The main ideas of the proof are the following.
To remove the restriction of the exponent
$1 \le q \le 2 \le p \le \infty$,
we derive a pointwise estimate for the convolution kernel for low frequency part.
Also, to make the derivative losses sharp, we apply the estimate \eqref{w_lp}
to the high frequency part.

\begin{remark}
Chen, Fan and Zhang \cite{ChFaZh14, ChFaZh15}
stated similar $L^p$-$L^q$ estimates for the damped fractional wave equation.
However, unfortunately, the proof seems incomplete.
For the damped wave equation, we give a complete proof.
Moreover, our argument remains valid for the damped fractional wave equation, through minor modifications.
\end{remark}

In the proof of Theorem \ref{thm_lplq},
we explicitly write a leading part of the 
convolution kernel of $\mathcal{D}(t)g$ in the low frequency part,
which has the same coefficient as that of the heat kernel
(see Lemmas \ref{lem_derivk} and \ref{lem_derivkg} below). 
Accordingly, the difference in the low frequency part satisfies a better 
time-decay estimate.

\begin{theorem}\label{thm_lplqdiff}
Let
$1\le q \le p < \infty$ with $p\neq 1$,
$s_1 \ge s_2$,
and
$\beta = (n-1) | \frac{1}{2} - \frac{1}{p} |$.
Then, there exists $\delta_p > 0$ such that
\begin{align}%
\label{lplqdiff}
	\| |\nabla|^{s_1} (\mathcal{D}(t) - \mathcal{G}(t)) g \|_{L^p}
	&\lesssim \langle t \rangle^{%
			-\frac{n}{2}\left( \frac{1}{q} - \frac{1}{p} \right)-\frac{s_1-s_2}{2}-1}
			\| |\nabla|^{s_2} \chi_{\le 1}(\nabla) g \|_{L^q} \\
\nonumber
	&\quad + e^{-\frac{t}{2}} \langle t \rangle^{\delta_p}
		\| |\nabla|^{s_1} \chi_{>1}(\nabla) g \|_{H_p^{\beta -1}},
\end{align}%
for $t\ge 1$,
provided that the right-hand side is finite.
\end{theorem}

\begin{remark}
Nishihara \cite{Ni03MathZ} gives an improvement of the estimate of \eqref{mat_est}
in the $3$-dimensional case
of the form
\begin{align}
\label{ni_est}
	\| \mathcal{D}(t) f - \mathcal{G}(t) f - e^{-t/2} \mathcal{W}(t) f \|_{L^p}
	\lesssim  t^{-\frac{3}{2}\left( \frac{1}{q} - \frac{1}{p} \right)-1} \| f\|_{L^q},
\end{align}
where $1\le q \le p \le \infty$
(for other space dimensions, see \cite{MaNi03, HoOg04, Na04, SaWa17}).
In other words, 
$\mathcal{D}(t) f$ is asymptotically expressed as
\begin{align*}
	\mathcal{D}(t)f \sim \mathcal{G}(t) f + e^{-t/2} \mathcal{W}(t)f\quad (t\to \infty),
\end{align*}
and it implies the high frequency part causing the derivative loss
is explicitly given by $\mathcal{W}(t)f$ when $n=3$.
Therefore, combining \eqref{ni_est} and \eqref{w_lp},
we can obtain similar estimates as \eqref{lplq}, \eqref{lplqdiff}.
However, our approach is direct and has broad utility.
\end{remark}

Our next purpose is the application of Theorems \ref{thm_lplq} and \ref{thm_lplqdiff}
to the Cauchy problem of the nonlinear damped wave equation
\begin{align}%
\label{dw}
	\left\{ \begin{array}{ll}
		\partial_{t}^2 u - \Delta u + \partial_t u = \mathcal{N}(u),
			&(t,x) \in (0,\infty) \times \mathbb{R}^n,\\
		u(0,x) = \varepsilon u_0(x),\ \partial_t u(0,x) = \varepsilon u_1(x),
			&x \in \mathbb{R}^n,
	\end{array}\right.
\end{align}%
where
$\mathcal{N}(u)$ denotes the nonlinearity,
$(u_0, u_1)$ is a given function, which denotes the shape of the initial data,
and $\varepsilon$ is a positive parameter, which denotes the size of the initial data.

Our concern is to prove the local and global existence of the solution,
asymptotic behavior, and blow-up of solutions
when the initial data do not belong to $L^1(\mathbb{R}^n)$ in general.
More precisely, we show the existence of the global solution with small data
even for the critical nonlinearity.

Based on the linear estimates \eqref{mat_est},
many mathematicians studied the global existence and blow-up of solutions
(see \cite{
HaKaNa04DIE, HaKaNa06, HaKaNa07JMAA, HaNa17JMAA, HoOg04, 
IkMiNa04, IkNiZh06, IkTa05, Kar00, KaNaOn95, KiQa02,
LiZh95, MaNi03, Ma76, Na04, Ni03MathZ, 
ToYo01, YaMi00, Zh01} and the references therein).
In particular, from these studies, the critical exponent was determined
as $p_c(n) = 1+ \frac{2}{n}$,
provided that the initial data decay sufficiently fast at the spatial infinity.
Here, the critical exponent means the threshold of the
global existence and the blow-up of solutions for small initial data.
More precisely, 
if $p$ is larger than the critical exponent $p_c$, 
then for any shape $(u_0,u_1)$, 
there exists $\varepsilon_0>0$ such that the solution exists globally in time for any $\varepsilon \in (0,\varepsilon_0)$,
and 
if $p$ is smaller than the critical exponent $p_c$, 
then there exists a shape $(u_0,u_1)$ and $\varepsilon_0>0$ such that the solution blows up in finite time for any $\varepsilon \in (0,\varepsilon_0)$.

However, there are only few results when the initial data slowly decay,
namely, do not belong to $L^1(\mathbb{R}^n)$,
at the spatial infinity.
Nakao and Ono \cite{NakOn93} studied the case
$(u_0, u_1) \in H^1(\mathbb{R}^n) \times L^2(\mathbb{R}^n)$
and they proved the global well-posedness with small data when $p \ge 1+\frac{4}{n}$.
We also refer the reader to \cite{Na11} for
the global existence of solutions with slowly decaying initial data in modulation spaces,
but the nonlinearity should be a polynomial of $u$.
Ikehata and Ohta \cite{IkOh02} proved the global existence of solutions for
small data
$(u_0, u_1) \in (H^1(\mathbb{R}^n) \cap L^r(\mathbb{R}^n))
\times (L^2(\mathbb{R}^n) \cap L^r(\mathbb{R}^n))$
when $p > 1+ \frac{2r}{n}$.
Here
\begin{align*}%
	&r \in [1,2] &(n=1,2),\\
	&r\in \left[ \frac{\sqrt{n^2+16n}-n}{4}, \min\left\{ 2, \frac{n}{n-2} \right\} \right]
	&(3\le n \le 6).
\end{align*}%
They also proved for any $n \ge 1$ and for the nonlinearity
$\mathcal{N}(u) = |u|^{p-1}u$ with $1<p<1+\frac{2r}{n}$,
there exists $(u_0, u_1) \in (H^1(\mathbb{R}^n) \cap L^r(\mathbb{R}^n))
\times (L^2(\mathbb{R}^n) \cap L^r(\mathbb{R}^n))$
such that there is no global solution even if
the size of the initial data $\varepsilon$ is arbitrary small.
Here we shall give a remark on their results.
In the supercritical case
$p> 1+\frac{2r}{n}$,
their solution belongs only to
$C([0,\infty); H^1(\mathbb{R}^n)) \cap C^1([0,\infty) ; L^2(\mathbb{R}^n))$
and we do not know whether $u(t) \in L^r(\mathbb{R}^n)$.
It is a natural question whether the solution $u$
has the same integrability near the spatial infinity as the initial data.
Narazaki and Nishihara \cite{NaNi08} further considered the asymptotic profile
of the solution under the assumption
$(u_0, u_1) \sim  \langle x \rangle^{-kn}$ with $k \in (0,1]$.
They proved that when $n \le 3$ and $p > 1+\frac{2}{kn}$
(which corresponds to the condition $p > 1+\frac{2r}{n}$
in terms of the Lebesgue space $L^r(\mathbb{R}^n)$),
the small data global existence holds and the solution is approximated by
$\varepsilon \mathcal{G}(t)(u_0+u_1)$.
Moreover, in \cite{IkeInWa17},
we extended the above results to higher dimensional cases
in terms of the weighted Sobolev spaces
\[
	H^{s,\alpha}(\mathbb{R}^n)
	= \langle x \rangle^{-\alpha} H^s(\mathbb{R}^n)
	:=\{f\in \mathcal{S}'(\mathbb{R}^n) ;\ \|\langle x\rangle ^{\alpha}\langle\nabla\rangle^sf\|_{L^2}<\infty\},
\]
where the symbol $\langle\nabla\rangle^s$ stands for the Fourier multiplier
$\mathcal{F}^{-1} \left[ \langle \xi \rangle^s \hat{f}(\xi) \right] (x)$.
We showed that if
$p > 1+\frac{2r}{n}$
and
if the initial data satisfy
$(u_0, u_1) \in
(H^{s,0}\cap H^{0,\alpha})(\mathbb{R}^n) \times (H^{s-1,0}\cap H^{0,\alpha})(\mathbb{R}^n)$
with
$\alpha > n(\frac{1}{r}-\frac12)$
and sufficiently small,
then the global solution uniquely exists.
However, in this setting, we cannot treat the critical case $p = 1+\frac{2r}{n}$.

In the present paper, based on the improved $L^p$-$L^q$ estimates
given in Theorem \ref{thm_lplq}, we further generalize the results of \cite{IkeInWa17}
when the initial data belong to
$L^r(\mathbb{R}^n)$
with
$r \in (1,2]$.
In particular, we prove the small data global existence
in the critical case $p = 1+\frac{2r}{n}$.
This result is completely new when $r \in (1,2)$.
We recall that when $r=1$, $p=1+\frac{2}{n}$, and $\mathcal{N}(u) = |u|^p$,
the local solution blows up in a finite time
even if the size of the initial data $\varepsilon$ is arbitrary small,
provided that the shape of the initial data
$(u_0, u_1) \in (H^1(\mathbb{R}^n)\cap L^1(\mathbb{R}^n))
\times (L^2(\mathbb{R}^n) \cap L^1(\mathbb{R}^n))$
has positive integral average
(see \cite{Zh01}).
Namely, when $r=1$, the critical power $p=1+\frac{2}{n}$
belongs to the blow-up case.
On the other hand, when $r=2$,
as we explained before, Nakao and Ono \cite{NakOn93} showed
that the critical power $p=1+\frac{4}{n}$ belongs to the global-existence case.
Our main result for the nonlinear problem (Theorem \ref{thm_gwp})
shows that for $r \in (1,2)$, the critical exponent $p=1+\frac{2r}{n}$
belongs to the global-existence case,
although some restriction on the range of $r$ is imposed.
Also, we refer the reader to \cite{We81} in which
the global existence of solutions to the critical semilinear heat equation
$v_t - \Delta v = v^{1+\frac{2r}{n}}$
was proved,
when the initial datum belongs to $L^r(\mathbb{R}^n)$
and is sufficiently small.

To state our results, we first define a solution.
We say that a function
$u \in L^{\infty}(0,T ; L^2(\mathbb{R}^n))$
is a mild solution of \eqref{dw} if
$u$
satisfies the integral equation
\begin{align*}%
	u(t) = \varepsilon \mathcal{D}(t) (u_0+u_1)
		+ \varepsilon \partial_t\mathcal{D}(t) u_0
		+ \int_0^t \mathcal{D}(t-\tau) \mathcal{N}(u(\tau)) d\tau
\end{align*}%
in
$L^{\infty}(0,T ; L^2(\mathbb{R}^n))$.
Moreover, we shall call $u$
an $H^s$-mild solution (resp. an $H^s\cap L^r$-mild solution)
if $u$ is a mild solution of \eqref{dw} satisfying
$u \in C([0,T) ; H^s(\mathbb{R}^n))$
(resp. $u \in C([0,T) ; H^s(\mathbb{R}^n) \cap L^r(\mathbb{R}^n))$).

We assume that
there exists
$p > 1$
such that
$\mathcal{N} \in C^{p_0}(\mathbb{R})$
with some integer
$p_0 \in [0,p]$
and
\begin{align}%
\label{N}
	\begin{cases}
	\mathcal{N}^{(l)}(0) = 0,\\
	\left| \mathcal{N}^{(l)}(u) - \mathcal{N}^{(l)}(v) \right|
	\le C |u-v| ( |u|+|v|)^{p-l-1}
	\end{cases}
	\quad (l=0,1, \ldots, p_0).
\end{align}%

Before going to the global existence results,
we first prepare the local existence of a unique $H^s\cap L^r$-mild solution
based on the linear estimates in Theorem \ref{thm_lplq}.
We note that
an $H^s \cap L^r$-mild solution is also an $H^s$-mild solution for any $r \in (1,2]$.
After introducing the existence of $H^s \cap L^r$-mild solution,
we also discuss the blow-up alternative for $H^s$-mild solution

\begin{theorem}[Local existence]\label{thm_lwp}
Let
$n\ge 1$ and $p > 1$,
and assume \eqref{N}.
Let
$s \ge 0$ and $r \in (1,2]$
satisfy
$[s] \le p_0$,
$r \ge \frac{2(n-1)}{n+1}$
and
\begin{align*}%
	\begin{array}{ll}
	\displaystyle
		1< p
		< \infty,
			&\mbox{if}\quad 2s\ge n,\\
	\displaystyle
		1 < p
		\le 1 + \frac{ \min \left\{n, 2 \right\} }{n-2s}
			&\mbox{if}\quad 2s<n.
	\end{array}
\end{align*}%
Let $\beta = (n-1)\left( \frac{1}{r} - \frac12 \right)$ and let the initial data satisfy
\begin{align*}%
	(u_0, u_1) \in
		(H^s(\mathbb{R}^n) \cap H_r^{\beta}(\mathbb{R}^n))
		\times (H^{s-1}(\mathbb{R}^n) \cap L^r(\mathbb{R}^n) ).
\end{align*}%
Then, for any
$\varepsilon >0$,
there exists
$T >0$
such that
the problem \eqref{dw} admits a unique $H^s\cap L^r$-mild solution
\begin{align*}%
	u \in C([0,T) ; H^s(\mathbb{R}^n) \cap L^r(\mathbb{R}^n) ),\quad
	\partial_t u \in C([0,T); H^{s-1}(\mathbb{R}^n)).
\end{align*}%
Moreover,
for the lifespan of the $H^s$-mild solution defined by
\begin{align}
\label{t2}
	T_2(\varepsilon) &:= \sup\left\{ T \in (0,\infty] ; 
		\ \mbox{there exists a unique $H^s$-mild solution of \eqref{dw}} \right.\\
\nonumber
		&\qquad \quad \left.
		\mbox{with}\ 
		u \in C([0,T);H^s(\mathbb{R}^n)), \partial_t u \in C([0,T);H^{s-1}(\mathbb{R}^n))
		\right\},
\end{align}
we have the blow-up alternative:
if $T_2(\varepsilon) < \infty$, then the solution satisfies
\begin{align}%
\label{bu}
	\liminf_{t\to T_2(\varepsilon)} \| (u, \partial_t u)(t) \|_{H^s \times H^{s-1}} = \infty.
\end{align}%
\end{theorem}

\begin{remark}\label{rem_r}
{\rm (i)} We remark that the assumption
$r \ge \frac{2(n-1)}{n+1}$
implies
$\beta =  (n-1) \left(\frac{1}{r} - \frac{1}{2} \right) \le 1$,
namely, the derivative loss in the linear estimate
does not exceed
$1$.

\noindent
{\rm (ii)} In the previous result \cite{IkeInWa17},
the local existence requires
$p \ge \max\{ 1+\frac{r}{n}, 1+\frac{r}{2} \}$,
which comes from estimates involving weighted Sobolev norms.
Theorem \ref{thm_lwp} removes this condition and we do not need
any restriction from below on $p$.

\noindent
{\rm (iii)} In Theorem \ref{thm_lwp}, 
we show the blow-up criterion only for the $H^s$-mild solution. 
It is difficult to obtain the blow-up criterion for the $H^s \cap L^r$-mild solution for $r\in[1,2)$
because the derivative loss prevents us from extending the local solution.
Namely, the solution does not have the persistence property,
which means that
the solution $u(t)$ belongs to the same space as the initial data
with continuous dependence on the time variable.
\end{remark}

We prove Theorem \ref{thm_lwp} in Section 3.
Our proof is based on the $L^p$-$L^q$ estimates given in Theorem \ref{thm_lplq}
and the contraction mapping principle.
To control the nonlinear term, we introduce an appropriate norm
for the nonlinearity (see \eqref{Y})
, which is inspired by
Hayashi, Kaikina and Naumkin \cite{HaKaNa04DIE}.
Then, to estimate the derivative of the nonlinearity,
we apply the fractional chain rule.

Moreover, in the critical or supercritical case
$p \ge 1+\frac{2r}{n}$,
we have the global existence of the $H^s \cap L^r$-mild solution
for the small initial data.

\begin{theorem}[Global existence of $H^s\cap L^r$-mild solution for small data]\label{thm_gwp}
In addition to the assumption of Theorem \ref{thm_lwp},
we suppose
\begin{align*}%
	1+ \frac{2r}{n} \le p.
\end{align*}%
Then, there exists
$\varepsilon_0 =\varepsilon_0(n,p,r,s,
\|u_0\|_{H^s\cap H^{\beta}_r}, \|u_1\|_{H^{s-1}\cap L^r})
>0$
such that for any
$\varepsilon \in (0,\varepsilon_0]$,
the problem \eqref{dw} admits a unique global $H^s\cap L^r$-mild solution satisfying
\begin{align*}%
	u \in C([0,\infty) ; H^s(\mathbb{R}^n) \cap L^r(\mathbb{R}^n) ),\quad
	\partial_t u \in C([0,\infty) ; H^{s-1}(\mathbb{R}^n)).
\end{align*}%
\end{theorem}

The reason why the global solution exists even in the critical case $p=1+\frac{2r}{n}$
is that
the nonlinearity $\mathcal{N}(u)$ decays faster than the linear part
at the spatial infinity.
More precisely,
we see that $\mathcal{N}(u) \in L^{\sigma_1}(\mathbb{R}^n)$
with $\sigma_1=\max\{ 1, \frac{r}{p} \} < r$ (see Section 3),
while the linear part of the solution satisfies
$\varepsilon \mathcal{D}(t) (u_0+u_1)+\varepsilon \partial_t\mathcal{D}(t) u_0
\in L^r(\mathbb{R}^n)$.
This enables us to control the nonlinearity even in the critical case
$p = 1+\frac{2r}{n}$.

On the other hand, if we consider the global existence of the
$H^s$-mild solution and do not require that
$u \in C([0,\infty) ; L^r(\mathbb{R}^n))$,
then we do not need to impose $r \geq \frac{2(n-1)}{n+1}$.
\begin{theorem}[Global existence of $H^s$-mild solution for small data]\label{thm_gwp2}
Let $n \ge 1$ and $p>1$, and assume \eqref{N}.
Let $s \ge 0$ and $r \in (1,2]$ satisfy $[s] \le p_0$,
$r \in \left( \frac{\sqrt{n^2+16n}-n}{4}, 2 \right]$,
and
\begin{align*}%
	\begin{array}{ll}
	\displaystyle 1+\frac{2r}{n} \le p < \infty &\mbox{if} \quad 2s \ge n,\\[5pt]
	\displaystyle 1+\frac{2r}{n} \le p \le 1+ \frac{\min\{n,2\}}{n-2s}
		&\mbox{if} \quad 2s < n.
	\end{array}
\end{align*}%
Let the initial data satisfy
\begin{align*}%
	(u_0, u_1) \in
		(H^s(\mathbb{R}^n) \cap L^r(\mathbb{R}^n))
		\times (H^{s-1}(\mathbb{R}^n) \cap L^r(\mathbb{R}^n) ).
\end{align*}%
Then, there exists $\varepsilon_0 =\varepsilon_0(n,p,r,s,
\|u_0\|_{H^s\cap L^r}, \|u_1\|_{H^{s-1}\cap L^r}) > 0$
such that for any $\varepsilon \in (0,\varepsilon_0]$,
the problem \eqref{dw} admits a unique global $H^s$-mild solution satisfying
\begin{align*}%
	u \in C([0,\infty) ; H^s(\mathbb{R}^n) ),\quad
	\partial_t u \in C([0,\infty) ; H^{s-1}(\mathbb{R}^n)).
\end{align*}%
\end{theorem}

\begin{remark}
Theorem \ref{thm_gwp2} states that
we can relax the condition
$r \in [\frac{2(n-1)}{n+1}, 2]$ in the case $n \ge 5$ if we do not require that
$u \in C([0,\infty) ; L^r(\mathbb{R}^n))$.
Indeed, concerning the assumption on the range of $r$,
we see that
\begin{align*}%
	(1,2] \cap \left[ \frac{2(n-1)}{n+1}, 2\right] \subset
	(1,2] \cap \left( \frac{\sqrt{n^2+16n}-n}{4}, 2\right]
	\quad \mbox{if and only if} \quad n\ge 5.
\end{align*}%
Therefore, for $n \ge 5$, the assumption of $r$ in Theorem \ref{thm_gwp2}
is weaker than that of Theorem \ref{thm_gwp}.
\end{remark}

Furthermore, in the supercritical case $p>1+\frac{2r}{n}$,
we prove that
the solution is approximated by that of the linear heat equation \eqref{h}
with the initial data $\varepsilon (u_0 + u_1)$.
This extends the results by \cite{NaNi08} to all space dimensions.
\begin{theorem}[Asymptotic behavior of global solutions]\label{thm_ab}
Let $u$ be the global $H^s\cap L^r$-mild solution constructed in Theorem \ref{thm_gwp}.
We further assume
$p > 1+\frac{2r}{n}$.
Then, for any $\delta>0$ and $t \ge 1$,
\begin{align*}
	\| |\nabla|^s (u(t) - \varepsilon \mathcal{G}(t) (u_0+u_1)) \|_{L^2}
	&\lesssim  \langle t \rangle^{%
		- \frac{n}{2}\left(\frac{1}{r}-\frac12\right) - \frac{s}{2}
		- \min\{1, \frac{n}{2r}(p-1)-1,
			\frac{n}{2}\left(\frac{1}{\sigma_1}-\frac{1}{r}\right)\}
			+\delta
		}\\ %
	&\quad \times
		\varepsilon ( \| u_0 \|_{H^s \cap H^{\beta}_r} + \| u_1 \|_{H^{s-1} \cap L^r}),\\
		\| u(t) - \varepsilon \mathcal{G}(t) (u_0+u_1) \|_{L^2}
	&\lesssim \langle t \rangle^{%
		- \frac{n}{2}\left(\frac{1}{r}-\frac12\right)
		- \min\{1, \frac{n}{2r}(p-1)-1,
			\frac{n}{2}\left(\frac{1}{\sigma_1}-\frac{1}{r}\right)\}
			+\delta
		}\\ %
	&\quad \times
		\varepsilon ( \| u_0 \|_{H^s \cap H^{\beta}_r} + \| u_1 \|_{H^{s-1} \cap L^r}),\\
		\| u(t) - \varepsilon \mathcal{G}(t)(u_0+u_1) \|_{L^r}
	&\lesssim \langle t \rangle^{%
			- \min\{
				1, \frac{n}{2r}(p-1)-1,
				\frac{n}{2}\left(\frac{1}{\sigma_1}-\frac{1}{r}\right),
				\frac{n}{2}\left(\frac{p}{r}-\frac{1}{q} \right)
			\}
			+\delta
		}\\ %
	&\quad \times
		\varepsilon ( \| u_0 \|_{H^s \cap H^{\beta}_r} + \| u_1 \|_{H^{s-1} \cap L^r}),
\end{align*}
where
$q= r$ if $2s \ge n$ and $q = \min\{ r, \frac{2n}{p(n-2s)} \}$ if $2s < n$,
and the implicit constant depends on $\delta$.
\end{theorem}

We next handle the subcritical case $p < 1+\frac{2r}{n}$.
We have the sharp lower bound
and an almost sharp upper bound of the lifespan.
To state the result, we define the
lifespan of $H^s\cap L^r$-mild solution by
\begin{align}
\label{tr}
	T_r(\varepsilon) &:= \sup\left\{ T \in (0,\infty] ; 
		\ \mbox{there exists a unique $H^s\cap L^r$-mild solution of \eqref{dw}}
		\right\}.
\end{align}

\begin{theorem}[Lower bound of the lifespan for $H^s\cap L^r$-mild solution]\label{thm_lb}
In addition to the assumption of Theorem \ref{thm_lwp},
we suppose that
\begin{align*}%
		p < 1+ \frac{2r}{n}.
\end{align*}%
Then, there exists
$\varepsilon_1 = \varepsilon_1(n,p,r,s,
\|u_0\|_{H^s\cap H^{\beta}_r}, \|u_1\|_{H^{s-1}\cap L^r}) > 0$
such that for any
$\varepsilon \in (0,\varepsilon_1]$,
the lifespan of $H^s\cap L^r$-mild solution $T_r(\varepsilon)$ is estimated as
\begin{align}%
\label{lb}
	T_r(\varepsilon) \gtrsim \varepsilon^{-1/\omega},
\end{align}%
where
$\omega = \frac{1}{p-1}-\frac{n}{2r}$
and
the implicit constant is independent of
$\varepsilon$.
\end{theorem}
We prove Theorem \ref{thm_lb} in Section 6.
The proof is a slight modification of Theorem \ref{thm_lwp}.

The rate $-1/\omega$ of $\varepsilon$ in the estimate \eqref{lb} is optimal
in the sense that we cannot obtain the estimate
$T_r(\varepsilon) \gtrsim \varepsilon^{-1/\omega -\delta}$
for any $\delta > 0$ in general.
More precisely, we give the following upper estimate of $T_2(\varepsilon)$
(see also Remark \ref{rem_ls}).

\begin{theorem}[Upper bound of the lifespan for $H^s$-mild solution]\label{thm_ub}
In addition to the assumption of Theorem \ref{thm_lwp},
we assume that
$\mathcal{N}(u) = \pm |u|^p$
with
$p< 1+\frac{2r}{n}$.
Then, for any $\delta > 0$,
there exist initial data
$(u_0, u_1) \in (H^s(\mathbb{R}^n) \cap H_r^{\beta}(\mathbb{R}^n))
\times (H^{s-1}(\mathbb{R}^n) \cap L^r(\mathbb{R}^n))$
and a constant
$\varepsilon_2 = \varepsilon_2(n,p,r,s,\delta)>0$
such that for any
$\varepsilon \in (0,\varepsilon_2]$,
the lifespan of the $H^s$-mild solution defined by \eqref{t2} is estimated as
\begin{align*}%
	T_2(\varepsilon) \lesssim \varepsilon^{-1/\omega - \delta},
\end{align*}%
where
the implicit constant is
dependent on $\delta$
but independent of
$\varepsilon$.
\end{theorem}
We will prove more general blow-up results in Section 7.
The proof is based on the argument of \cite{FuIkeWa_pre},
in which the blow-up of solutions to the semilinear wave equation
with time-dependent damping was studied via
an analysis of ordinary differential inequality.

\begin{remark}\label{rem_ls}
By the definitions of lifespans of $H^s$-mild solution \eqref{t2}
and $H^s\cap L^r$-mild solution \eqref{tr},
we immediately have
$T_r(\varepsilon) \le T_2(\varepsilon)$.
From this and Theorems \ref{thm_lb} and \ref{thm_ub},
we see that
\begin{align*}%
	\varepsilon^{-1/\omega}
	\lesssim T_r(\varepsilon) \le T_2(\varepsilon)
	\lesssim \varepsilon^{-1/\omega - \delta},
\end{align*}%
which gives an almost optimal estimate for both $T_r(\varepsilon)$ and $T_2(\varepsilon)$.
\end{remark}

The rest of the paper is organized as follows.
In Section 2, we prove our $L^p$-$L^q$ estimates.
Theorems \ref{thm_lwp} and \ref{thm_gwp}, namely,
the local and global existence of $H^s\cap L^r$-mild solution are proved in Section 3.
Then, Section 4 is devoted to the proof of Theorem \ref{thm_gwp2},
that is, the global existence of
$H^s$-mild solution.
In Section 5, we give a proof of Theorem \ref{thm_ab}.
Finally, in Sections 6 and 7, we give proofs of Theorems \ref{thm_lb} and \ref{thm_ub},
respectively.

We introduce notations used throughout this paper.
For the variable $x =(x_1, \dots, x_n) \in \mathbb{R}^n$,
we use the notation of derivatives $\partial_j = \frac{\partial}{\partial x_j}\ (j=1,\ldots,n)$.
Let $\bm{1}_I$ be the characteristic function of $I \subset \mathbb{R}$.
The notation $X \sim Y$ stands for $X \lesssim Y$ and $Y \lesssim X$.

Let $\chi \in C_0^{\infty} (\mathbb{R})$
be a cut-off function satisfying $\chi (r)=1$ for $|r| \le 1$ and $\chi (r) =0$ for $|r| \ge 2$.
We write
\begin{align}
\label{chi}
	\chi_{<a} (r) := \chi \left( \frac{r}{a} \right),\ 
	\chi_{\ge a} (r) := 1 - \chi_{<a} (r),\ 
	\chi_{a \le \cdot <b} (r) := \chi_{<b} (r) - \chi_{<a} (r)
\end{align}
for $0<a<b$.

For a function $f : \mathbb{R}^n \to \mathbb{C}$,
we define the Fourier transform and the inverse Fourier transform by
\begin{align*}
	\mathcal{F} [f] (\xi) = \hat{f} (\xi) = (2\pi)^{-n/2}
		\int_{\mathbb{R}^n} e^{-i x \xi} f(x) \,dx,\quad
	\mathcal{F}^{-1} [f] (x) = (2\pi)^{-n/2}
		\int_{\mathbb{R}^n} e^{i x \xi} f(\xi) \,d\xi.
\end{align*}
Moreover, for a measurable function $m = m(\xi)$, we denote
the Fourier multiplier $m(\nabla)$ by
\begin{align*}
	m(\nabla) f (x) = \mathcal{F}^{-1} \left[ m(\xi) \hat{f}(\xi) \right] (x).
\end{align*}

For $s \in \mathbb{R}$ and $p \in (1,\infty)$,
we denote the usual Sobolev space by $H^s_p(\mathbb{R}^n)$
and its homogeneous version by $\dot{H}^s_p(\mathbb{R}^n)$.

\section{Proof of the $L^p$-$L^q$ estimates}

We divide $\mathcal{D}(t)$ into low and high frequency parts $\mathcal{D}(t) = \mathcal{D}_1(t) + \mathcal{D}_2(t)$, where
\[
\mathcal{D}_1(t) = \mathcal{F}^{-1} \left[ \chi_{<1} (|\xi|) e^{-\frac{t}{2}} L(t,\xi) \mathcal{F} \right], \quad
\mathcal{D}_2(t) = \mathcal{F}^{-1} \left[ \chi_{\ge 1} (|\xi|) e^{-\frac{t}{2}} L(t,\xi) \mathcal{F} \right].
\]
Let $\mathfrak{d}$ be the multiplier of $\mathcal{D}_1$, namely
\[
\mathfrak{d} (t,x) := \mathcal{F}^{-1} \left[ \chi_{<1} (|\xi|) e^{-\frac{t}{2}} L(t,\xi) \right] (x).
\]

\subsection{$L^p$-$L^q$ estimates for linear damped wave equation} \label{sec_lplq}
First, we focus on the low frequency part.
The $L^p$-$L^q$ estimates of the low frequency part is similar to that of the heat propagator. 
The first step is to get the pointwise estimate for the kernel $\mathfrak{d}$, 
which gives the value of the $L^r$-norm of the kernel $\mathfrak{d}$. 
The second step is to get the $L^p$-$L^q$ estimates 
whose proof is based on Young's inequality and the value of the $L^r$-norm of the kernel $\mathfrak{d}$. 

We have the following pointwise estimate of the kernel $\mathfrak{d}$. 
\begin{proposition} \label{prop_pointwisem1}
For $s \ge 0$, we have
\begin{equation}\label{eq2.1}
| |\nabla|^s \mathfrak{d} (t,x)|
\lesssim \min \left( |x|^{-1}, \langle t \rangle^{-\frac{1}{2}} \right)^{s+n}.
\end{equation}
Moreover, for any $j \in \mathbb{N}$,
\begin{equation}\label{eq2.2}
|\mathfrak{d} (t,x)|
\lesssim \langle t \rangle^{-\frac{n}{2}} \min \left( \langle t \rangle^{\frac{1}{2}} |x|^{-1}, 1 \right)^j.
\end{equation}
\end{proposition}
For the proof of Proposition \ref{prop_pointwisem1}, we observe the following lemmas.
\begin{lemma} \label{lem_lowexpint}
For $t \ge 0$, $a \in \mathbb{R}$, and $\sigma \in (0,\frac{1}{2})$, we have
\begin{equation*}
\int_{\sigma \le |\xi| \le \frac{1}{2}} |\xi|^a e^{-t |\xi|^2} d\xi
\lesssim
\begin{cases}
\langle t \rangle^{-\frac{a+n}{2}}, & \text{if } a>-n, \\
\langle \log (\langle t \rangle^{\frac{1}{2}} \sigma) \rangle, & \text{if } a=-n, \\
\sigma^{a+n}, & \text{if } a<-n.
\end{cases}
\end{equation*}
Moreover, $\sigma =0$ is allowed if $a>-n$.
\end{lemma}
\begin{proof}
By changing variable $\eta = \langle t \rangle^{\frac{1}{2}} \xi$, we have
\begin{align*}
\int_{\sigma \le |\xi| \le \frac{1}{2}} |\xi|^a e^{-t |\xi|^2} d\xi
= \langle t \rangle^{-\frac{a+n}{2}} \int_{\langle t \rangle^{\frac{1}{2}} \sigma \le |\eta| \le \frac{1}{2} \langle t \rangle^{\frac{1}{2}}} |\eta|^a e^{-\frac{t}{\langle t \rangle} |\eta|^2} d\eta.
\end{align*}
When $\langle t \rangle^{\frac{1}{2}} \sigma \ge \frac{1}{2}$, the integral on the right hand side is bounded by
\[
\int_{\frac{1}{2} \le |\eta| \le \frac{1}{2} \langle t \rangle^{\frac{1}{2}}} |\eta|^a e^{-\frac{t}{\langle t \rangle} |\eta|^2} d\eta
\le 
\begin{cases}
\displaystyle
\int_{\frac{1}{2} \le |\eta| \le 2} |\eta|^a d\eta
\lesssim 1, & \text{if } 0 \le t \le 1, \\[15pt]
\displaystyle
\int_{|\eta| \ge \frac{1}{2}} |\eta|^a e^{-\frac{1}{2} |\eta|^2} d\eta
\lesssim 1, & \text{if } t \ge 1.
\end{cases}
\]
For $\langle t \rangle^{\frac{1}{2}} \sigma \le \frac{1}{2}$, we have
\begin{align*}
	\int_{\langle t \rangle^{\frac{1}{2}}
		\sigma \le |\eta| \le \frac{1}{2} \langle t \rangle^{\frac{1}{2}}}
		|\eta|^a e^{-\frac{t}{\langle t \rangle} |\eta|^2} d\eta
	& \le \int_{\langle t \rangle^{\frac{1}{2}} \sigma \le |\eta| \le \frac{1}{2}}
		|\eta|^a d\eta +
		\int_{\frac{1}{2} \le |\eta| \le \frac{1}{2} \langle t \rangle^{\frac{1}{2}}}
		|\eta|^a e^{-\frac{t}{\langle t \rangle} |\eta|^2} d\eta \\
& \lesssim
\begin{cases}
1, & \text{if } a>-n, \\
\langle \log (\langle t \rangle^{\frac{1}{2}} \sigma) \rangle, & \text{if } a=-n, \\
\langle t \rangle^{\frac{a+n}{2}} \sigma^{a+n}, & \text{if } a<-n,
\end{cases}
\end{align*}
which concludes the proof.
\end{proof}

\begin{lemma} \label{lem_derivk}
For $k \in \mathbb{Z}_{\ge 0}$,
there exist some constants $C_{l,m}^{(k)}$ $(k-[\frac{k}{2}] \le l \le k$, $0 \le m \le l)$
satisfying
\begin{align}
\partial_1^{k} \left( \frac{e^{t \sqrt{\frac{1}{4}-|\xi|^2}}}{\sqrt{\frac{1}{4}-|\xi|^2}} \right)
= e^{t \sqrt{\frac{1}{4}-|\xi|^2}} \sum_{l=k-[\frac{k}{2}]}^k \sum_{m=0}^l C_{l,m}^{(k)} t^m \xi_1^{2l-k} (\tfrac{1}{4}-|\xi|^2)^{-l+\frac{m-1}{2}}
\label{derivk}
\end{align}
for $t \in \mathbb{R}$, where $\partial_1=\partial/\partial \xi_{1}$.
\end{lemma}

\begin{proof}
For $k=0$, we have $C^{(0)}_{0,0}=1$.
We assume that \eqref{derivk} holds for some $k \in \mathbb{Z}_{\ge 0}$.
For simplicity, we define $C^{(k)}_{l,m} =0$ for $(l,m) \notin \{ (l,m) \in \mathbb{Z}_{\ge 0}^2 \colon  k-[\frac{k}{2}] \le l \le k, 0 \le m \le l \}$.
Then, a direct calculation yields
\begin{align*}
& e^{-t \sqrt{\frac{1}{4}-|\xi|^2}} \partial_1^{k+1} \left( \frac{e^{t \sqrt{\frac{1}{4}-|\xi|^2}}}{\sqrt{\frac{1}{4}-|\xi|^2}} \right) \\
&= e^{-t \sqrt{\frac{1}{4}-|\xi|^2}} \partial_1 \left\{ e^{t \sqrt{\frac{1}{4}-|\xi|^2}} \sum_{l=k-[\frac{k}{2}]}^k \sum_{m=0}^l C^{(k)}_{l,m} t^m \xi_1^{2l-k} (\tfrac{1}{4}-|\xi|^2)^{-l+\frac{m-1}{2}} \right\} \\
& = \sum_{l=k-[\frac{k}{2}]}^k \sum_{m=0}^l C^{(k)}_{l,m} \left\{ -t^{m+1} \xi_1^{2l-k+1} (\tfrac{1}{4}-|\xi|^2)^{-l+\frac{m-2}{2}} \right.\\
& \quad \left. + (2l-k) t^m \xi_1^{2l-k-1} (\tfrac{1}{4}-|\xi|^2)^{-l+\frac{m-1}{2}} + (2l-m+1) t^m \xi_1^{2l-k+1} (\tfrac{1}{4}-|\xi|^2)^{-l+\frac{m-3}{2}} \right\} \\
&= \sum_{l=k+1-[\frac{k}{2}]}^{k+1} \sum_{m=1}^l (-1) C^{(k)}_{l-1,m-1} t^{m} \xi_1^{2l-(k+1)} (\tfrac{1}{4}-|\xi|^2)^{-l+\frac{m-1}{2}} \\
& \quad + \sum_{l=k+1-[\frac{k+1}{2}]}^{k} \sum_{m=0}^l (2l-k) C^{(k)}_{l,m} t^{m} \xi_1^{2l-(k+1)} (\tfrac{1}{4}-|\xi|^2)^{-l+\frac{m-1}{2}} \\
& \quad + \sum_{l=k+1-[\frac{k}{2}]}^{k+1} \sum_{m=0}^{l-1} (2l-m-1) C^{(k)}_{l-1,m} t^{m} \xi_1^{2l-(k+1)} (\tfrac{1}{4}-|\xi|^2)^{-l+\frac{m-1}{2}} \\
&= \sum_{l=k+1-[\frac{k+1}{2}]}^{k+1} \sum_{m=0}^l \left\{ -C^{(k)}_{l-1,m-1} + (2l-k) C^{(k)}_{l,m} + (2l-m-1) C^{(k)}_{l-1,m} \right\} \\
& \hspace*{90pt} \times t^{m} \xi_1^{2l-(k+1)} (\tfrac{1}{4}-|\xi|^2)^{-l+\frac{m-1}{2}}.
\end{align*}
Hence, the constants $C^{(k+1)}_{l,m}$ are defined by
\begin{equation} \label{recurrenceC}
C^{(k+1)}_{l,m} := -C^{(k)}_{l-1,m-1} + (2l-k) C^{(k)}_{l,m} + (2l-m-1) C^{(k)}_{l-1,m},
\end{equation}
which shows \eqref{derivk}.
\end{proof}

\begin{proof}[Proof of Proposition \ref{prop_pointwisem1}]
We prove the inequality with respect to the right side in the minimum in \eqref{eq2.1} and \eqref{eq2.2}, \textit{i.e.}
\begin{align*}
	| |\nabla|^s \mathfrak{d} (t,x)| \lesssim\langle t \rangle^{-\frac{s+n}{2}}
\end{align*}
for any $s \geq 0$.
Since
\begin{equation} \label{rationalize}
-\frac{1}{2}+\sqrt{\frac{1}{4}-|\xi|^2} = -\frac{|\xi|^2}{\frac{1}{2}+\sqrt{\frac{1}{4}-|\xi|^2}} \le -|\xi|^2
\end{equation}
for $|\xi| \le \frac{1}{2}$, Lemma \ref{lem_lowexpint} gives
\begin{align}
\label{d(t)est0}
	||\nabla|^s \mathfrak{d} (t,x)|
	&= (2\pi)^{-\frac{n}{2}} \left| \int_{\mathbb{R}^n} e^{ix\xi}
		\chi_{<1}(|\xi|) |\xi|^s e^{-\frac{t}{2}} L(t,\xi) d\xi \right| \\
\nonumber
	&\lesssim \int_{|\xi|<\frac{1}{2}-\delta} |\xi|^s e^{-t|\xi|^2} d\xi 
		+ \int_{\frac{1}{2}-\delta \le |\xi| \le 2} t 	e^{-\frac{t}{4}} d\xi\\
\nonumber
	&\lesssim \langle t \rangle^{-\frac{s+n}{2}},
\end{align}
where $\delta>0$ is a sufficiently small number.

It remains to prove the inequality with respect to the left side in the minimum in \eqref{eq2.1}
and \eqref{eq2.2}.
To obtain the decay with respect to $|x|$,
we divide $\mathfrak{d}$ into two parts
$\mathfrak{d} = \mathfrak{d}_1 + \mathfrak{d}_2$:
\begin{align*}
	& \mathfrak{d}_1 (t,x)
		:= \mathcal{F}^{-1} \left[ \chi_{<\frac{1}{8}} (|\xi|) e^{-\frac{t}{2}} L(t,\xi) \right] (x), \\
	& \mathfrak{d}_2 (t,x)
		:= \mathcal{F}^{-1} \left[ \chi_{\frac{1}{8} \le \cdot < 1} (|\xi|) e^{-\frac{t}{2}} L(t,\xi) \right] (x).
\end{align*}
Without loss of generality, we may assume that $|x| \sim |x_1|$.
First, we prove the estimate for $\mathfrak{d}_{1}$
with respect to the left side in the minimum in \eqref{eq2.2}.
Namely, we show
\begin{align*}
	|\mathfrak{d}_1 (t,x)|
	\lesssim \langle t \rangle^{-\frac{n}{2}+\frac{j}{2}} |x|^{-j},
\end{align*}
 for any $j \in \mathbb{N}$.

By \eqref{rationalize} and Lemma \ref{lem_derivk},
for $k \in \mathbb{Z}_{\ge 0}$ and $|\xi| \le \frac{1}{4}$, we have
\begin{equation} \label{lowderivL}
\begin{aligned}
\left| e^{-\frac{t}{2}} \partial_1^k L(t,\xi) \right|
	& \le \frac{1}{2} e^{t \left( -\frac{1}{2}+\sqrt{\frac{1}{4}-|\xi|^2} \right)}
		\sum_{l=k-[\frac{k}{2}]}^k
		\sum_{m=0}^l |C_{l,m}^{(k)}| t^m |\xi|^{2l-k} (\tfrac{1}{4}-|\xi|^2)^{-l+\frac{m-1}{2}} \\
	& \quad + \frac{1}{2} e^{t \left( -\frac{1}{2} -\sqrt{\frac{1}{4}-|\xi|^2} \right)}
		\sum_{l=k-[\frac{k}{2}]}^k
		\sum_{m=0}^l |C_{l,m}^{(k)}| t^m |\xi|^{2l-k} (\tfrac{1}{4}-|\xi|^2)^{-l+\frac{m-1}{2}} \\
	& \lesssim \sum_{l=k-[\frac{k}{2}]}^k \langle t \rangle^l |\xi|^{2l-k} e^{-t|\xi|^2}.
\end{aligned}
\end{equation}
Since
$\left| \partial_1^k \left( \chi_{<\frac{1}{8}} (|\xi|) \right) \right| \lesssim 1$
for
$k \in \mathbb{Z}_{\ge 0}$,
integration by parts, \eqref{lowderivL}, and Lemma \ref{lem_lowexpint} yield
\begin{equation} \label{d(t)est1}
\begin{aligned}
	& | \mathfrak{d}_1 (t,x)| \\
	&= (2\pi)^{-\frac{n}{2}} \left| \int_{\mathbb{R}^n}
		e^{ix\xi} \chi_{<\frac{1}{8}} (|\xi|) e^{-\frac{t}{2}} L(t,\xi) d\xi \right| \\
	&= (2\pi)^{-\frac{n}{2}} \frac{1}{|x_1|^j} \left| \sum_{k=0}^j
		\binom{j}{k} \int_{\mathbb{R}^n} e^{ix\xi}
		\partial_1^{j-k} \left( \chi_{<\frac{1}{8}} (|\xi|) \right) e^{-\frac{t}{2}}
			\partial_1^k L(t,\xi) d\xi \right| \\
	& \lesssim |x|^{-j} \sum_{k=0}^j
		\sum_{l=k-[\frac{k}{2}]}^k \langle t \rangle^l
			\int_{|\xi| \le \frac{1}{4}} |\xi|^{2l-k} e^{-t |\xi|^2} d\xi \\
	& \lesssim |x|^{-j} \langle t \rangle^{\frac{j}{2}-\frac{n}{2}}
\end{aligned}
\end{equation}
for any $j \in \mathbb{N}$.

Secondly, we show the inequality for $\mathfrak{d}_{1}$
with respect to the left side in the minimum in \eqref{eq2.1} \textit{i.e.}
$| |\nabla|^s \mathfrak{d}_{1} (t,x)| \lesssim |x|^{-(s+n)}$.
For
$s \ge 0$, 
we assume
$|x| \ge \langle t \rangle^{\frac{1}{2}}$,
otherwise the desired bound follows from \eqref{d(t)est0}.
Then, we further divide the multiplier
$\mathfrak{d}_1$
into two parts
$\mathfrak{d}_1 = \mathfrak{d}_1' + \mathfrak{d}_1''$:
\begin{align*}
	& \mathfrak{d}_1' (t,x)
		:= \mathcal{F}^{-1} \left[ \chi_{< \frac{1}{8|x|}} (|\xi|) e^{-\frac{t}{2}} L(t,\xi) \right] (x), \\
	& \mathfrak{d}_1'' (t,x)
		:= \mathcal{F}^{-1} \left[ \chi_{\frac{1}{8|x|} \le \cdot <\frac{1}{8}} (|\xi|)
							e^{-\frac{t}{2}} L(t,\xi) \right] (x).
\end{align*}
By \eqref{lowderivL}, we have
\begin{equation} \label{d(t)est1'}
	\begin{aligned}
	| |\nabla|^s \mathfrak{d}_1' (t,x)|
	& = (2\pi)^{-\frac{n}{2}}
		\left| \int_{\mathbb{R}^n} e^{ix\xi}
			\chi_{<\frac{1}{8|x|}} (|\xi|) |\xi|^s e^{-\frac{t}{2}} L(t,\xi) d\xi \right| \\
	& \lesssim \int_{|\xi| \le \frac{1}{4|x|}} |\xi|^s d\xi
	\lesssim |x|^{-s-n}.
\end{aligned}
\end{equation}
We set $j := [s]+n+1$ for simplicity.
Integration by parts $j$-times yields
\begin{align*}
	& | |\nabla|^s \mathfrak{d}_1'' (t,x)| \\
	&= (2\pi)^{-\frac{n}{2}} \left|
		\int_{\mathbb{R}^n} e^{ix\xi} \chi_{\frac{1}{8|x|}\le \cdot <\frac{1}{8}} (|\xi|) |\xi|^s
		e^{-\frac{t}{2}} L(t,\xi) d\xi
		\right| \\
	&= (2\pi)^{-\frac{n}{2}} \frac{1}{|x_1|^j} \left|
		\sum_{k=0}^j \binom{j}{k}
		\int_{\mathbb{R}^n} e^{ix\xi}
		\partial_1^{j-k} \left( \chi_{\frac{1}{8|x|} \le \cdot <\frac{1}{8}} (|\xi|) |\xi|^s \right)
		e^{-\frac{t}{2}} \partial_1^k L(t,\xi) d\xi
		\right|
\end{align*}
Since
$| \partial_1^k ( \chi_{\frac{1}{8|x|} \le \cdot <\frac{1}{8}} (|\xi|) |\xi|^s ) |
\lesssim |\xi|^{s-k}$
for
$k \in \mathbb{Z}_{\ge 0}$,
from \eqref{lowderivL} we further obtain 
\begin{align*}
	| |\nabla|^s \mathfrak{d}_1'' (t,x)|
	\lesssim |x|^{-j} \sum_{k=0}^j \sum_{l=k-[\frac{k}{2}]}^k
	\langle t \rangle^l
	\int_{\frac{1}{8|x|} \le |\xi| \le \frac{1}{4}} |\xi|^{s-j+2l} e^{-t |\xi|^2} d\xi.
\end{align*}
Finally, Lemma \ref{lem_lowexpint} concludes
\begin{align}
\nonumber
	&| |\nabla|^s \mathfrak{d}_1'' (t,x)|\\
\nonumber
	& \lesssim |x|^{-j}
		\sum_{l=0}^j \langle t \rangle^l
			\Big\{ \langle t \rangle^{-l+\frac{j-s-n}{2}} \bm{1}_{(\frac{j-s-n}{2},j]} (l)
			+ |\log (\langle t \rangle^{\frac{1}{2}} |x|^{-1})| \bm{1}_{\{ l=\frac{j-s-n}{2} \}} (l) \\
\label{d(t)est1''}
	& \hspace*{150pt} + |x|^{-2l+j-s-n} \bm{1}_{[0,\frac{j-s-n}{2})}(l) \Big\} \\
\nonumber
	& \lesssim |x|^{-j}
		\langle t \rangle^{\frac{j-s-n}{2}}
		+ |x|^{-j} \langle t \rangle^{\frac{j-s-n}{2}}
		|\log (\langle t \rangle^{\frac{1}{2}} |x|^{-1})|
		+ |x|^{-s-n} \\
\nonumber
	& \lesssim |x|^{-s-n},
\end{align}
provided that $|x| \ge \langle t \rangle^{\frac{1}{2}}$.

At last, we go on to the estimate for $\mathfrak{d}_2$ with respect to the left side in the minimum in \eqref{eq2.1} and \eqref{eq2.2}. 
More precisely, we prove the better estimate 
$| |\nabla|^s \mathfrak{d}_{2} (t,x)| \lesssim  |x|^{-j}$ 
for any $j\in \mathbb{N}$. 
Since $L(t,\xi)$ is smooth with respect to $\xi$,
\[
	|\partial_1^k L(t,\xi)|
	\lesssim \langle t \rangle ^{k+1}
	\left\{ e^{t\sqrt{\frac{1}{4}-|\xi|^2}}
		\bm{1}_{[0,\frac{1}{2}]} (|\xi|) + \bm{1}_{[\frac{1}{2},2]} (|\xi|) \right\}
\]
for $|\xi| \le 2$ and $k \in \mathbb{Z}_{\ge 0}$.
Since
\[
	-\frac{1}{2} + \sqrt{\frac{1}{4}-|\xi|^2} \le \frac{-4+\sqrt{15}}{8} <0
\]
for $\frac{1}{8} \le |\xi| \le \frac{1}{2}$, we have
\[
	e^{-\frac{t}{2}} \int_{\frac{1}{8} \le |\xi| \le \frac{1}{2}} \left| \partial_1^k L(t,\xi) \right| d\xi
	\lesssim \langle t \rangle^{k+1} e^{ \frac{-4+\sqrt{15}}{8} t}.
\]
Hence, integration by parts yields
\begin{equation} \label{d(t)est2}
	\begin{aligned}
	& | |\nabla|^s \mathfrak{d}_2 (t,x)| \\
	&= (2\pi)^{-\frac{n}{2}}
	\left| \int_{\mathbb{R}^n} e^{ix\xi}
		\chi_{\frac{1}{8} \le \cdot < 1} (|\xi|) |\xi|^s e^{-\frac{t}{2}} L(t,\xi) d\xi \right| \\
	&= (2\pi)^{-\frac{n}{2}} \frac{1}{|x_1|^j}
		\left| \sum_{k=0}^j \binom{j}{k}
		\int_{\mathbb{R}^n} e^{ix\xi} \partial_1^{j-k}
		\left( \chi_{\frac{1}{8} \le \cdot < 1} (|\xi|) |\xi|^s \right) 
			e^{-\frac{t}{2}} \partial_1^k L(t,\xi) d\xi \right| \\
	& \lesssim |x|^{-j} \langle t \rangle^{j+1} e^{\frac{-4+\sqrt{15}}{8} t}
	\end{aligned}
\end{equation}
for any $j \in \mathbb{N}$.
Combining \eqref{d(t)est0} and \eqref{d(t)est1}--\eqref{d(t)est2}, we obtain the desired bound.
\end{proof}

\begin{remark} \label{remarkd112}
We divide $\mathfrak{d}_1$ into two parts $\mathfrak{d}_1 = \mathfrak{d}_1^{(1)} - \mathfrak{d}_1^{(2)}$:
\begin{align*}
& \mathfrak{d}_1^{(1)} (t,x) := \mathcal{F}^{-1} \bigg[ \chi_{<1} (|\xi|) \frac{e^{t \left( -\frac{1}{2}+\sqrt{\frac{1}{4}-|\xi|^2} \right)}}{2 \sqrt{\frac{1}{4}-|\xi|^2}} \bigg] (x), \\
& \mathfrak{d}_1^{(2)} (t,x) := \mathcal{F}^{-1} \bigg[ \chi_{<1} (|\xi|) \frac{e^{t \left( -\frac{1}{2}-\sqrt{\frac{1}{4}-|\xi|^2} \right)}}{2 \sqrt{\frac{1}{4}-|\xi|^2}} \bigg] (x).
\end{align*}
In \eqref{lowderivL}, we neglect the decay factor $e^{-\frac{t}{2}}$ of $\mathfrak{d}_1^{(2)}$.
By keeping this factor, the proof of Proposition \ref{prop_pointwisem1} yields that for any $s \ge 0$ and $j \in \mathbb{Z}_{\ge 0}$, we have
\begin{align*}
& | |\nabla|^s \mathfrak{d}_1^{(2)} (t,x)|
\lesssim \min \left( |x|^{-1}, \langle t \rangle^{-\frac{1}{2}} \right)^{s+n} e^{-\frac{t}{2}}, \\
& | \mathfrak{d}_1^{(2)} (t,x)|
\lesssim \langle t \rangle^{-\frac{n}{2}} \min \left( \langle t \rangle^{\frac{1}{2}} |x|^{-1}, 1 \right)^j  e^{-\frac{t}{2}}.
\end{align*}
\end{remark}

Proposition \ref{prop_pointwisem1}
and Young's inequality lead to the following linear estimate for the low frequency.

\begin{proposition} \label{prop_lpboundlow}
Let $1\le q \le p \le \infty$ and $s_1 \ge s_2 \ge 0$.
Then, 
\[
\| |\nabla|^{s_1} \mathcal{D}_1(t) g \|_{L^p}
\lesssim \langle t \rangle^{-\frac{n}{2}\left( \frac{1}{q} - \frac{1}{p} \right) -\frac{s_1-s_2}{2}} \| |\nabla|^{s_2} g \|_{L^q}
\]
for $|\nabla|^{s_2} g \in L^q(\mathbb{R}^n)$.
\end{proposition}

\begin{proof}
It reduces to show the bound
\begin{equation} \label{d(t)lpbound}
\| |\nabla|^s \mathfrak{d} (t) \|_{L^p}
\lesssim \langle t \rangle^{-\frac{n}{2} \left( 1-\frac{1}{p} \right)-\frac{s}{2}}
\end{equation}
for $s \ge 0$ and $1 \le p \le \infty$.
Indeed, \eqref{d(t)lpbound} and Young's inequality show
\begin{align*}
\| |\nabla|^{s_1} \mathcal{D}_1(t) g \|_{L^p}
& = \| (|\nabla|^{s_1-s_2} \mathfrak{d}(t)) \ast |\nabla|^{s_2} g \|_{L^p}
\le \| |\nabla|^{s_1-s_2} \mathfrak{d}(t) \|_{L^r} \| |\nabla|^{s_2} g \|_{L^q} \\
& \lesssim \langle t \rangle^{-\frac{n}{2} (1-\frac{1}{r}) -\frac{s_1-s_2}{2}} \| |\nabla|^{s_2} g \|_{L^q}
\sim \langle t \rangle^{-\frac{n}{2} (\frac{1}{q}-\frac{1}{p}) -\frac{s_1-s_2}{2}} \| |\nabla|^{s_2} g \|_{L^q}
\end{align*}
with $\frac{1}{r} = \frac{1}{p}-\frac{1}{q}+1$.

It remains to show \eqref{d(t)lpbound}.
The case $p=\infty$ is a direct consequence of Proposition \ref{prop_pointwisem1}.
For $s>0$ and $1 \le p <\infty$, Proposition \ref{prop_pointwisem1} implies
\begin{align*}
\| |\nabla|^s \mathfrak{d} (t) \|_{L^p}
& \lesssim \langle t \rangle^{-\frac{s+n}{2}} \left( \int_{|x| < \langle t \rangle^{\frac{1}{2}}} dx \right)^{\frac{1}{p}} + \left( \int_{|x| \ge \langle t \rangle^{\frac{1}{2}}} |x|^{-p(s+n)} dx \right)^{\frac{1}{p}} \\
& \lesssim \langle t \rangle^{-\frac{n}{2} \left( 1-\frac{1}{p} \right)-\frac{s}{2}}.
\end{align*}
On the other hand, for $s=0$, Proposition \ref{prop_pointwisem1} with $j=n+1$ implies
\begin{align*}
\| \mathfrak{d} (t) \|_{L^p}
& \lesssim \langle t \rangle^{-\frac{n}{2}} \left( \int_{|x| < \langle t \rangle^{\frac{1}{2}}} dx \right)^{\frac{1}{p}} + \langle t \rangle^{\frac{1}{2}} \left( \int_{|x| \ge \langle t \rangle^{\frac{1}{2}}} |x|^{-p(n+1)} dx \right)^{\frac{1}{p}} \\
& \lesssim \langle t \rangle^{-\frac{n}{2} \left( 1-\frac{1}{p} \right)}.
\end{align*}
This finishes the proof.
\end{proof}

Next, we consider the high frequency part of $\mathcal{D}(t)$.
To estimate the high frequency part, we reduce the high frequency part to the wave propagator
by using Mikhlin's multiplier theorem and apply the $L^p$-estimate of the wave propagator.

\begin{proposition} \label{prop_lpboundhigh}
Let $1<p<\infty$ and
$\beta = (n-1) | \frac{1}{2} - \frac{1}{p} |$.
Then,
there exists $\delta_p > 0$ such that
\[
	\| \mathcal{D}_2(t) g \|_{L^p}
	\lesssim e^{-\frac{t}{2}} \langle t \rangle^{\delta_p} \| g \|_{H^{\beta-1}_p}
\]
for $g \in H^{\beta-1}_p (\mathbb{R}^n)$ and $t \ge 0$.
\end{proposition}

This proposition follows from the $L^p$ bound for linear wave solutions,
which was proved by Sj\"{o}strand \cite{Sj70}
and improved by Miyachi \cite{Mi80} 
and Peral \cite{Pe80}.

\begin{theorem}[$L^p$ estimates for the wave equation] \label{thm_lpboundwave}
Let $1<p<\infty$.
Then, for any
$\beta \ge (n-1) | \frac{1}{2} - \frac{1}{p} |$,
there exists $\delta_p' > 0$ such that
\[
\| e^{it |\nabla|} g \|_{L^p} \lesssim \langle t \rangle^{\delta_p'} \| g \|_{H^{\beta}_p}
\]
for $g \in H^{\beta}_p (\mathbb{R}^n)$ and $t \in \mathbb{R}$.
\end{theorem}

\begin{proof}[Proof of Proposition \ref{prop_lpboundhigh}]
Since we focus on the high frequency and $\sin (t \sqrt{|\xi|^2-1/4}) = \frac{1}{2i} ( e^{it \sqrt{|\xi|^2-1/4}}-e^{-it\sqrt{|\xi|^2-1/4}})$, it suffices to show that
\begin{equation} \label{lpboundmodwave}
\| e^{it \sqrt{-\Delta-1/4}} \mathcal{F}^{-1} [ \chi_{\ge 1} \mathcal{F} g] \|_{L^p} \lesssim \langle t \rangle^{\delta_p} \| g \|_{H^{\beta}_p}
\end{equation}
for $g \in H^{\beta}_p (\mathbb{R}^n)$.
We note that
\[
e^{it \sqrt{-\Delta-1/4}} = e^{it|\nabla|} e^{it \left( \sqrt{-\Delta-1/4}-|\nabla| \right)}.
\]
Because
\[
\left| \sqrt{|\xi|^2-1/4}-|\xi| \right| = \frac{1}{4(\sqrt{|\xi|^2-1/4}+|\xi|)} \sim |\xi|^{-1},
\]
a simple calculation shows
\[
\left| \partial^{\alpha} e^{it \left( \sqrt{|\xi|^2-1/4}-|\xi| \right)} \right|
\lesssim \langle t \rangle^{|\alpha|} |\xi|^{-|\alpha|}
\]
for $\xi \neq 0$ and $\alpha \in \mathbb{Z}_{\ge 0}^n$.
Hence, Mikhlin's multiplier theorem (see \cite[Theorem 6.2.7]{Gra}) shows that
\[
\| e^{it (\sqrt{-\Delta-1/4}-|\nabla|)} \mathcal{F}^{-1} [ \chi_{\ge 1} \mathcal{F} g] \|_{L^p}
\lesssim \langle t \rangle^{[\frac{n}{2}]+1} \| \mathcal{F}^{-1} [ \chi_{\ge 1} \mathcal{F} g] \|_{L^p}
\lesssim \langle t \rangle^{[\frac{n}{2}]+1} \| g \|_{L^p}.
\]
Owing to Theorem \ref{thm_lpboundwave}, this estimate yields \eqref{lpboundmodwave}.
\end{proof}

The estimate \eqref{lplq} in Theorem \ref{thm_lplq} follows from
Propositions \ref{prop_lpboundlow} and \ref{prop_lpboundhigh}.

\subsection{$L^p$-$L^q$ estimates for the derivative of the solution}

Next, we prove the second statement \eqref{lplqdt} of Theorem \ref{thm_lplq}.
We use the same notation as in Remark \ref{remarkd112}.
We define the Fourier multipliers $\mathcal{D}_1^{(j)}$ by $\mathcal{D}_1^{(j)} g = \mathfrak{d}_1^{(j)} \ast g$ for $j=1,2$.
From \eqref{rationalize}, Mikhlin's multiplier theorem shows that
\begin{align*}
& \| \partial_t \mathcal{D}_1^{(1)} g \|_{L^p}
= \bigg\| \Big( \frac{1}{2} + \sqrt{\frac{1}{4}+\Delta} \Big)^{-1} \Delta \mathcal{D}_1^{(1)} g \bigg\|_{L^p}
\lesssim \| \Delta \mathcal{D}_1^{(1)} g \|_{L^p}, \\
& \| \partial_t \mathcal{D}_1^{(2)} g \|_{L^p}
= \bigg\| \Big( \frac{1}{2} + \sqrt{\frac{1}{4}+\Delta} \Big) \mathcal{D}_1^{(1)} g \bigg\|_{L^p}
\lesssim \| \mathcal{D}_1^{(2)} g \|_{L^p}.
\end{align*}
Therefore, Remark \ref{remarkd112} and Proposition \ref{prop_lpboundlow} imply the desired estimate for the low frequency part $\partial_t \mathcal{D}_1 (t) g$.
Moreover, the same argument as in the proof of Proposition \ref{prop_lpboundhigh} shows the estimate for the high frequency part $\partial_t \mathcal{D}_2 (t) g$.

\subsection{$L^p$-$L^q$ estimates for the difference}

Now, we prove Theorem \ref{thm_lplqdiff}.
Set
\[
\mathfrak{m}(t,x) := \mathfrak{d}(t,x) - \mathcal{F}^{-1} [\chi_{<1}(|\xi|)e^{-t|\xi|^2}](x).
\]
We recall that $\mathfrak{d}$ is the multiplier of the low frequency part of $\mathcal{D}$.

We show the pointwise decay estimates for $\mathfrak{m}$.

\begin{proposition} \label{prop_pointwisediff}
For $s \ge 0$, we have
\begin{equation*}
| |\nabla|^s \mathfrak{m} (t,x) |
\lesssim \min \left( |x|^{-1}, \langle t \rangle^{-\frac{1}{2}} \right)^{s+n+2}.
\end{equation*}
\end{proposition}

For the proof of Proposition \ref{prop_pointwisediff}, we observe the following two lemmas.

\begin{lemma} \label{lem_derivkg}
For $k \in \mathbb{N}$, there exist some constants $D_{l,m}^{(k)}$ $(k-[\frac{k}{2}] \le l \le k$, $1 \le m \le l)$ satisfying
\begin{equation}
\partial_1^{k} \left( e^{-t|\xi|^2} \right)
= e^{-t |\xi|^2} \sum_{l=k-[\frac{k}{2}]}^k \sum_{m=1}^l D_{l,m}^{(k)} t^m \xi_1^{2l-k}. \label{derivkg}
\end{equation}
In particular, $D^{(k)}_{l,l} = 2^l C^{(k)}_{l,l}$ for $k-[\frac{k}{2}] \le l \le k$, where $C^{(k)}_{k,l}$ are the constants in Lemma \ref{lem_derivk}.
\end{lemma}

\begin{proof}
For $k=1$, we have $D^{(1)}_{1,1}=-2$, because $\partial_1 \left( e^{-t|\xi|^2} \right) = -2t\xi_1 e^{-t|\xi|^2}$.
We assume that \eqref{derivkg} holds for some $k \in \mathbb{N}$.
For simplicity, we define $D^{(k)}_{l,m} =0$ for $(l,m) \notin \{ (l,m) \in \mathbb{N}^2 \colon  k-[\frac{k}{2}] \le l \le k, 1 \le m \le l \}$.
Then, a direct calculation yields
\begin{align*}
& e^{t |\xi|^2} \partial_1^{k+1} \left( e^{-t |\xi|^2} \right) \\
&= e^{t |\xi|^2} \partial_1 \left\{ e^{-t |\xi|^2} \sum_{l=k-[\frac{k}{2}]}^k \sum_{m=1}^l D^{(k)}_{l,m} t^m \xi_1^{2l-k} \right\} \\
& = \sum_{l=k-[\frac{k}{2}]}^k \sum_{m=1}^l D^{(k)}_{l,m} \left\{ -2t^{m+1} \xi_1^{2l-k+1} + (2l-k) t^m \xi_1^{2l-k-1} \right\} \\
& = \sum_{l=k+1-[\frac{k}{2}]}^{k+1} \sum_{m=2}^{l} (-2) D^{(k)}_{l-1,m-1} t^{m} \xi_1^{2l-(k+1)} + \sum_{l=k+1-[\frac{k+1}{2}]}^k \sum_{m=1}^l (2l-k) D^{(k)}_{l,m} t^m \xi_1^{2l-(k+1)} \\
& = \sum_{l=k+1-[\frac{k+1}{2}]}^{k+1} \sum_{m=1}^{l} \left\{ -2 D^{(k)}_{l-1,m-1} + (2l-k) D^{(k)}_{l,m} \right\} t^m \xi_1^{2l-(k+1)}.
\end{align*}
Hence, the constants $D^{(k+1)}_{l,m}$ are defined by
\[
D^{(k+1)}_{l,m} := -2 D^{(k)}_{l-1,m-1} + (2l-k) D^{(k)}_{l,m},
\]
which shows \eqref{derivkg}.
In particular, $2^{-l} D^{(k)}_{l,l}$ satisfies \eqref{recurrenceC} with $m=l$ because $C^{(k)}_{l-1,l}=0$.
Since $C^{(1)}_{1,1}=-1$ and $D^{(1)}_{1,1} = -2$, we obtain $D^{(k)}_{l,l} = 2^l C^{(k)}_{l,l}$ for $k-[\frac{k}{2}] \le l \le k$.
\end{proof}

\begin{lemma} \label{mulpoint}
For $|\xi| \le \frac{1}{4}$, we have
\[
\left| \frac{e^{t \left( -\frac{1}{2}+\sqrt{\frac{1}{4}-|\xi|^2} \right)}}{2 \sqrt{\frac{1}{4}-|\xi|^2}}  - e^{-t|\xi|^2} \right|
\lesssim \left( |\xi|^2 + \langle t \rangle |\xi|^4 \right) e^{-t|\xi|^2} + e^{-t^{\frac{1}{2}}} \bm{1}_{(\langle t \rangle^{-\frac{1}{4}}, \infty)} (|\xi|).
\]
Furthermore, for $k \in \mathbb{N}$ and $|\xi| \le \frac{1}{4}$, we have
\begin{align*}
& \left| \partial_1^{k} \left( \frac{e^{t \left( -\frac{1}{2}+\sqrt{\frac{1}{4}-|\xi|^2} \right)}}{2 \sqrt{\frac{1}{4}-|\xi|^2}} - e^{-t|\xi|^2} \right) \right| \\
& \lesssim \sum_{l=k-[\frac{k}{2}]-1}^{k+1} \langle t \rangle^l |\xi|^{2(l+1)-k} e^{-t|\xi|^2}
+ \langle t \rangle^k e^{-t^{\frac{1}{2}}} \bm{1}_{(\langle t \rangle^{-\frac{1}{4}}, \infty)} (|\xi|).
\end{align*}
\end{lemma}

\begin{proof}
We note that
\begin{align*}
& \left| \frac{1}{\sqrt{1-4|\xi|^2}} -1 \right| = \frac{4|\xi|^2}{\sqrt{1-4|\xi|^2}(1+\sqrt{1-4|\xi|^2})} \lesssim |\xi|^2, \\
&
\begin{aligned}
\left| e^{t \left( -\frac{1}{2}+\sqrt{\frac{1}{4}-|\xi|^2} \right)}-e^{-t|\xi|^2} \right|
& = e^{-t|\xi|^2} \left| \exp \left( -\frac{4t|\xi|^4}{(1+\sqrt{1-4|\xi|^2})^2} \right)-1 \right| \\
& \lesssim t |\xi|^4 e^{-t|\xi|^2} \bm{1}_{[0,\langle t \rangle^{-\frac{1}{4}}]} (|\xi|) + e^{-t \langle t \rangle^{-\frac{1}{2}}} \bm{1}_{(\langle t \rangle^{-\frac{1}{4}}, \infty)} (|\xi|) \\
& \lesssim \langle t \rangle |\xi|^4 e^{-t|\xi|^2} + e^{-t^{\frac{1}{2}}} \bm{1}_{(\langle t \rangle^{-\frac{1}{4}}, \infty)} (|\xi|)
\end{aligned}
\end{align*}
for
$|\xi| \le \frac{1}{4}$, where  in the second inequality we used 
\[
	-\frac12+\sqrt{\frac{1}{4}-|\xi|^2}+|\xi|^2
		= -\frac{1}{4}\left(1-\sqrt{1-4|\xi|^2}\right)^2
		=-\frac{4|\xi|^4}{\left(1+\sqrt{1-4|\xi|^2}\right)^2}.
\]
The triangle inequality with \eqref{rationalize} implies
\begin{align*}
	\left| \frac{e^{t \left( -\frac{1}{2}+\sqrt{\frac{1}{4}-|\xi|^2} \right)}}{2 \sqrt{\frac{1}{4}-|\xi|^2}}
		- e^{-t|\xi|^2} \right|
	& \le e^{-t|\xi|^2} \left| \frac{1}{\sqrt{1-4|\xi|^2}} -1 \right|
		+ \left| \frac{e^{t \left( -\frac{1}{2}+\sqrt{\frac{1}{4}-|\xi|^2} \right)}
		-e^{-t|\xi|^2} }{\sqrt{1-4|\xi|^2}} \right| \\
& \lesssim e^{-t|\xi|^2} \left( |\xi|^2 + \langle t \rangle |\xi|^4 \right) + e^{-t^{\frac{1}{2}}} \bm{1}_{(\langle t \rangle^{-\frac{1}{4}}, \infty)} (|\xi|).
\end{align*}
Similarly, Lemmas \ref{lem_derivk} and \ref{lem_derivkg} yield that, for $k \in \mathbb{N}$, 
\begin{align*}
& \left| \partial_1^{k} \left( \frac{e^{t \left( -\frac{1}{2}+\sqrt{\frac{1}{4}-|\xi|^2} \right)}}{2 \sqrt{\frac{1}{4}-|\xi|^2}} - e^{-t|\xi|^2} \right) \right| \\
& \le e^{-t|\xi|^2} \sum_{l=k-[\frac{k}{2}]}^k |C_{l,l}^{(k)}| t^l |\xi|^{2l-k} \Bigg\{ \left| \frac{1}{2} (\tfrac{1}{4}-|\xi|^2)^{-\frac{l+1}{2}} -2^l \right| \\
& \hspace*{150pt} + \left| e^{t(-\frac{1}{2}+\sqrt{\frac{1}{4}-|\xi|^2}+|\xi|^2)}-1 \right| 2^l \Bigg\} \\
& \quad + \frac{e^{-t|\xi|^2}}{2} \sum_{l=k-[\frac{k}{2}]}^k \sum_{m=0}^{l-1} |C_{l,m}^{(k)}| t^m |\xi|^{2l-k} (\tfrac{1}{4}-|\xi|^2)^{-l+\frac{m-1}{2}} \\
& \quad + e^{-t |\xi|^2} \sum_{l=k-[\frac{k}{2}]}^k \sum_{m=1}^{l-1} |D_{l,m}^{(k)}| t^m |\xi|^{2l-k} \\
& \lesssim e^{-t|\xi|^2} \sum_{l=k-[\frac{k}{2}]}^k \langle t \rangle^l |\xi|^{2l-k} \left( |\xi|^2 + t|\xi|^4 \right) + \langle t \rangle^k e^{-t^{\frac{1}{2}}} \bm{1}_{(\langle t \rangle^{-\frac{1}{4}}, \infty)} (|\xi|) \\
& \quad + e^{-t|\xi|^2} \sum_{l=k-[\frac{k}{2}]}^k \langle t \rangle^{l-1} |\xi|^{2l-k} \\
& \lesssim \sum_{l=k-[\frac{k}{2}]-1}^{k+1} \langle t \rangle^l |\xi|^{2(l+1)-k} e^{-t|\xi|^2}
+ \langle t \rangle^k e^{-t^{\frac{1}{2}}} \bm{1}_{(\langle t \rangle^{-\frac{1}{4}}, \infty)} (|\xi|).
\end{align*}
\end{proof}

\begin{proof}[Proof of Proposition \ref{prop_pointwisediff}]
Lemmas \ref{lem_lowexpint} and \ref{mulpoint} give
\begin{equation} \label{m(t)est0}
\begin{aligned}
||\nabla|^s \mathfrak{m} (t,x) |
&= (2\pi)^{-\frac{n}{2}} \left| \int_{\mathbb{R}^n} e^{ix\xi} \chi_{<1}(|\xi|) |\xi|^s \left\{ e^{-\frac{t}{2}} L(t,\xi) - e^{-t|\xi|^2} \right\} d\xi \right| \\
& \lesssim \int_{|\xi|<\frac{1}{4}} \left\{ (|\xi|^{s+2} + \langle t \rangle |\xi|^{s+4})
	e^{-t|\xi|^2} + e^{-t^{\frac{1}{2}}} \right\} d\xi \\
& \hspace*{50pt} + \int_{\frac{1}{4} \le |\xi| \le 2} \left( \langle t \rangle e^{-\frac{t}{2}} + e^{-t|\xi|^2} \right) d\xi \\
& \lesssim \langle t \rangle^{-\frac{s+n}{2}-1}.
\end{aligned}
\end{equation}

To obtain the decay with respect to $|x|$, we assume $|x| \ge \langle t \rangle^{\frac{1}{2}}$, otherwise the desired bound follows from \eqref{m(t)est0}.
We divide $\mathfrak{m}$ into four parts $\mathfrak{m} = \mathfrak{m}_1 + \mathfrak{m}_2 + \mathfrak{m}_3 + \mathfrak{m}_4$:
\begin{align*}
& \mathfrak{m}_1 (t,x) := \mathcal{F}^{-1} \left[ \chi_{< \frac{1}{8|x|}} (|\xi|) \left( \tfrac{e^{t \left( -\frac{1}{2}+\sqrt{\frac{1}{4}-|\xi|^2} \right)}}{2 \sqrt{\frac{1}{4}-|\xi|^2}} - e^{-t|\xi|^2} \right) \right] (x), \\
& \mathfrak{m}_2 (t,x) := \mathcal{F}^{-1} \left[ \chi_{\frac{1}{8|x|} \le \cdot <\frac{1}{8}} (|\xi|) \left( \tfrac{e^{t \left( -\frac{1}{2}+\sqrt{\frac{1}{4}-|\xi|^2} \right)}}{2 \sqrt{\frac{1}{4}-|\xi|^2}} - e^{-t|\xi|^2} \right) \right] (x), \\
& \mathfrak{m}_3 (t,x) := -\mathcal{F}^{-1} \left[ \chi_{<\frac{1}{8}} (|\xi|) \tfrac{e^{-t \left( \frac{1}{2}+\sqrt{\frac{1}{4}-|\xi|^2} \right)}}{2 \sqrt{\frac{1}{4}-|\xi|^2}} \right] (x), \\
& \mathfrak{m}_4 (t,x) := \mathcal{F}^{-1} \left[ \chi_{\frac{1}{8} \le \cdot < 1} (|\xi|) \left( e^{-\frac{t}{2}} L(t,\xi) - e^{-t|\xi|^2} \right) \right] (x).
\end{align*}
Without loss of generality, we may assume that $|x| \sim |x_1|$.

Since $\frac{1}{4|x|} \le \frac{1}{4} \langle t \rangle^{-\frac{1}{2}} \le \langle t \rangle^{-\frac{1}{4}}$, Lemma  \ref{mulpoint} yields
\begin{equation} \label{m(t)est1}
\begin{aligned}
| |\nabla|^s \mathfrak{m}_1 (t,x)|
& = (2\pi)^{-\frac{n}{2}} \left| \int_{\mathbb{R}^n} e^{ix\xi} \chi_{<\frac{1}{8|x|}} (|\xi|) |\xi|^s \left( \tfrac{e^{t \left( -\frac{1}{2}+\sqrt{\frac{1}{4}-|\xi|^2} \right)}}{2 \sqrt{\frac{1}{4}-|\xi|^2}} - e^{-t|\xi|^2} \right) d\xi \right| \\
& \lesssim \int_{|\xi| \le \frac{1}{4|x|}} \left( |\xi|^{s+2} + \langle t \rangle |\xi|^{s+4} \right) d\xi \\
& \lesssim |x|^{-s-n-2} + t |x|^{-s-n-4}
\sim |x|^{-s-n-2}.
\end{aligned}
\end{equation}

We set $j := [s]+n+3$ for simplicity.
Integration by parts $j$-times yields
\begin{align*}
	& | |\nabla|^s \mathfrak{m}_2 (t,x)| \\
	&= (2\pi)^{-\frac{n}{2}}
		\left| \int_{\mathbb{R}^n} e^{ix\xi}
		\chi_{\frac{1}{8|x|} \le \cdot <\frac{1}{8}} (|\xi|) |\xi|^s
		\left( \tfrac{
			e^{t \left( -\frac{1}{2}+\sqrt{\frac{1}{4}-|\xi|^2} \right)}}{
			2 \sqrt{\frac{1}{4}-|\xi|^2}} - e^{-t|\xi|^2}
		\right) d\xi \right| \\
	&= (2\pi)^{-\frac{n}{2}} \frac{1}{|x_1|^j}
	\left| \sum_{k=0}^j \binom{j}{k}
	\int_{\mathbb{R}^n} e^{ix\xi} \partial_1^{j-k}
	\left( \chi_{\frac{1}{8|x|} \le \cdot <\frac{1}{8}} (|\xi|) |\xi|^s \right) \right.\\
	&\quad \left.\times\partial_1^k
		\left(
			\tfrac{
				e^{t \left( -\frac{1}{2}+\sqrt{\frac{1}{4}-|\xi|^2} \right)}}{
				2 \sqrt{\frac{1}{4}-|\xi|^2}} - e^{-t|\xi|^2}
		\right) d\xi \right| 
\end{align*}
Since $| \partial_1^k ( \chi_{\frac{1}{8|x|}
\le \cdot <\frac{1}{8}} (|\xi|) |\xi|^s ) | \lesssim |\xi|^{s-k}$ for $k \in \mathbb{Z}_{\ge 0}$,
we proceed the calculation with Lemma \ref{mulpoint} to obtain
\begin{align*}
	& | |\nabla|^s \mathfrak{m}_2 (t,x)| \\
	& \lesssim |x|^{-j} \Bigg(
		\int_{\frac{1}{8|x|} \le |\xi| \le \frac{1}{4}}
		(|\xi|^{s-j+2}+ \langle t \rangle |\xi|^{s-j+4}) e^{-t |\xi|^2} d\xi \\
	& \hspace*{100pt} + \sum_{k=1}^j
		\sum_{l=k-[\frac{k}{2}]-1}^{k+1}
			\langle t \rangle^l
			\int_{|\xi| \le \frac{1}{4}} |\xi|^{s-j+2(l+1)} e^{-t |\xi|^2} d\xi \Bigg) \\
	& \quad + |x|^{-j} \langle t \rangle^j e^{-t^{\frac{1}{2}}}.
\end{align*}
Furthermore, Lemma \ref{lem_lowexpint} gives
\begin{align*}
	& | |\nabla|^s \mathfrak{m}_2 (t,x)| \\
	& \lesssim |x|^{-j} \Bigg\{
		|x|^{j-s-n-2}
		+ \sum_{l=0}^{j+1} \langle t \rangle^l
			\Big( \langle t \rangle^{-(l+1)+\frac{j-s-n}{2}} \bm{1}_{(\frac{j-s-n}{2}-1,j+1]} (l) \\
	& \hspace*{40pt}  + |\log (\langle t \rangle^{\frac{1}{2}} |x|^{-1})|
		\bm{1}_{\{ l=\frac{j-s-n}{2}-1 \}} (l) + |x|^{-2(l+1)+j-s-n}
		\bm{1}_{[0,\frac{j-s-n}{2}-1)}(l) \Big) \Bigg\}\\
	& \quad + |x|^{-j} \langle t \rangle^j e^{-t^{\frac{1}{2}}} \\
	& \lesssim |x|^{-s-n-2} + |x|^{-j}
		\langle t \rangle^{\frac{j-s-n}{2}-1}
		|\log (\langle t \rangle^{\frac{1}{2}} |x|^{-1})| 
			+ |x|^{-j} \langle t \rangle^j e^{-t^{\frac{1}{2}}} \\
	& \lesssim |x|^{-s-n-2},
	\label{m(t)est2}
\end{align*}
provided that $|x| \ge \langle t \rangle^{\frac{1}{2}}$.

The calculation used in \eqref{d(t)est2} yields
\begin{equation} \label{m(t)est3}
| |\nabla|^s \mathfrak{m}_3 (t,x)|
\lesssim |x|^{-j} \langle t \rangle^j e^{-\frac{t}{2}}, \quad 
| |\nabla|^s \mathfrak{m}_4 (t,x)|
\lesssim |x|^{-j} \langle t \rangle^{j+1} e^{\frac{-2+\sqrt{3}}{4}t}
\end{equation}
for any $j \in \mathbb{N}$.
Combining \eqref{m(t)est0}--\eqref{m(t)est3} we obtain the desired bound.
\end{proof}

Proposition \ref{prop_pointwisediff} leads to the $L^p$-$L^q$ estimate for the difference.

\begin{proposition} \label{prop_lpbounddiff}
Let $1\le q \le p \le \infty$ and $s_1 \ge s_2 \ge 0$.
Then, 
\[
\| |\nabla|^{s_1} (\mathcal{D}_1(t) - \mathcal{G}(t)) g \|_{L^p}
\lesssim \langle t \rangle^{-\frac{n}{2}\left( \frac{1}{q} - \frac{1}{p} \right) -\frac{s_1-s_2}{2}-1} \| |\nabla|^{s_2} g \|_{L^q}
\]
for $|\nabla|^{s_2} g \in L^q(\mathbb{R}^n)$ and $t \ge 1$.
\end{proposition}

\begin{proof}
We divide $\mathcal{G}(t)$ into low and high frequency parts $\mathcal{G}(t) = \mathcal{G}_1(t) + \mathcal{G}_2(t)$, where
\[
\mathcal{G}_1(t) = \mathcal{F}^{-1}[ \chi_{<1} (|\xi|) e^{-t|\xi|^2} \mathcal{F}], \quad
\mathcal{G}_2(t) = \mathcal{F}^{-1}[ \chi_{\ge 1} (|\xi|) e^{-t|\xi|^2} \mathcal{F}].
\]
Since Proposition \ref{prop_pointwisediff} gives
\[
\| |\nabla|^s \mathfrak{m} (t) \|_{L^p}
\lesssim \langle t \rangle^{-\frac{n}{2} \left( 1-\frac{1}{p} \right)-\frac{s}{2}-1}
\]
for $s \ge 0$ and $1 \le p \le \infty$, the same argument as in the proof of Proposition \ref{prop_lpboundlow} yields
\[
\| |\nabla|^{s_1} (\mathcal{D}_1(t) - \mathcal{G}_1(t)) g \|_{L^p}
\lesssim \langle t \rangle^{-\frac{n}{2}\left( \frac{1}{q} - \frac{1}{p} \right) -\frac{s_1-s_2}{2}-1} \| |\nabla|^{s_2} g \|_{L^q}.
\]

Set $r := -(s_1-s_2)+2([s_1-s_2]+n+1)$.
Since $r>n$, integration by parts gives
\[
|\mathcal{F}^{-1} [|\cdot|^{-r} \chi_{\ge 1}]|
= (2\pi)^{-\frac{n}{2}} \left| \int_{\mathbb{R}^n} e^{ix\xi} |\xi|^{-r} \chi_{\ge 1} (|\xi|) d\xi \right|
\lesssim \min (1, |x|^{-n-1}),
\]
which yields that $\mathcal{F}^{-1} [|\cdot|^{-r} \chi_{\ge 1}] \in L^1(\mathbb{R}^d)$.
Hence, Young's inequality and the well-known $L^p$-$L^q$ estimates for the heat equation imply
\begin{align*}
\| |\nabla|^{s_1} \mathcal{G}_2(t) g \|_{L^p}
& = \| \mathcal{F}^{-1} [|\cdot|^{-2N-(s_1-s_2)} \chi_{\ge 1}] \ast (|\nabla|^{2N} \mathcal{G}(t) |\nabla|^{s_2} g) \|_{L^p} \\
& \lesssim \| \Delta^{N} \mathcal{G}(t) |\nabla|^{s_2} g \|_{L^p} \\
& \lesssim \langle t \rangle^{-\frac{n}{2}\left( \frac{1}{q} - \frac{1}{p} \right) -2N} \| |\nabla|^{s_2} g \|_{L^q} \\
\end{align*}
for $t \ge 1$ and for any large $N \in \mathbb{N}$.
\end{proof}

Theorem \ref{thm_lplqdiff} follows from Propositions \ref{prop_lpboundhigh} and \ref{prop_lpbounddiff}.


\section{Local and global existence}
Based on the linear estimates, we define the following function spaces.
For $T \in (0,\infty]$, $s \ge 0$, and $r \in (1,2]$,
we define
\begin{align*}%
	X(T) :=
	\{
	 \phi \in L^{\infty}(0,T ; H^s(\mathbb{R}^n) \cap L^r(\mathbb{R}^n));
	 \| \phi \|_{X(T)} < \infty
	 \}
\end{align*}%
with the norm
\begin{align}%
\label{X}
	\| \phi \|_{X(T)}
	&:= \sup_{t \in (0,T)}
		\left\{
		\langle t \rangle^{\frac{n}{2}\left( \frac{1}{r} - \frac{1}{2}\right) + \frac{s}{2}}
				\| |\nabla|^s \phi(t) \|_{L^2}
		+ \langle t \rangle^{\frac{n}{2}\left( \frac{1}{r} - \frac{1}{2} \right)}
				\| \phi(t) \|_{L^2}
			+ \| \phi(t) \|_{L^r}
		\right\}.
\end{align}%
It is obvious that
$X(T)$ is a Banach space.
Let $M>0$. We consider the closed ball
\begin{align*}%
	X(T, M) := \{ u \in X(T) ; \| u \|_{X(T)} \le M \}.
\end{align*}%
We also consider a wider function space
\begin{align*}%
	Z(T) := L^{\infty}(0,T ; L^2(\mathbb{R}^n))
\end{align*}%
with the norm
\begin{align*}%
	\| u \|_{Z(T)} := \| u \|_{L^{\infty}(0,T;L^2(\mathbb{R}^n))}.
\end{align*}%
Then, we can see that
$X(T, M)$ is a closed subset of
$Z(T)$ for $T \in (0,\infty)$ (see Lemma \ref{lem_close}).
We shall find a local solution
by constructing an approximate sequence
in the ball $X(T, M)$
and prove its convergence
with respect to the metric
\begin{align}%
\label{d}
	d(u,v) := \| u - v \|_{Z(T)}.
\end{align}%

To this end, for the estimate of the nonlinear term
satisfying \eqref{N} and $|\mathcal{N}(u)| \lesssim |u|^{p}$,
we define an auxiliary space.
For $T \in (0,\infty]$, $s \ge 0$, $r \in (1,2]$,
and $1<p \le 2n/(n-2s)$ if $n>2s$ and $1<p<\infty$ if $n\le 2s$ we define the space $Y(T)$ as follows.
When $s>1$, we define
\begin{align*}%
	Y(T) :=
	\{
	 \psi \in L^{\infty}(0,T ; \dot{H}^{s-1}(\mathbb{R}^n) \cap L^{\sigma_1}(\mathbb{R}^n)
	 											\cap L^{\sigma_2}(\mathbb{R}^n));
	 \| \psi \|_{Y(T)} < \infty
	 \},
\end{align*}%
where
\begin{align}%
\label{Y}
	\| \psi \|_{Y(T)}
	&:= \sup_{t\in (0,T)}
		\left\{ \langle t \rangle^{\eta} \| |\nabla|^{s-1} \psi(t)\|_{L^2}
			+ \sup_{\gamma\in [\sigma_1, \sigma_2]}
				\langle t \rangle^{\frac{n}{2}\left( \frac{p}{r}-\frac{1}{\gamma}\right)}
				\| \psi(t) \|_{L^{\gamma}}
		\right\}
\end{align}%
and when $0<s<1$ we define
\begin{align*}%
	Y(T) :=
	\{
	 \psi \in L^{\infty}(0,T; L^{\sigma_1}(\mathbb{R}^n)
	 		\cap L^{\sigma_2}(\mathbb{R}^n));
	 \| \psi \|_{Y(T)} < \infty
	 \},
\end{align*}%
where
\begin{align}%
\label{Y1}
	\| \psi \|_{Y(T)}
	&:= \sup_{t\in (0,T)}
		\left\{ \langle t \rangle^{\eta} \| |\nabla|^{s-1} \chi_{>1}(\nabla) \psi(t)\|_{L^2}
			+ \sup_{\gamma\in [\sigma_1, \sigma_2]}
				\langle t \rangle^{\frac{n}{2}\left( \frac{p}{r}-\frac{1}{\gamma}\right)}
				\| \psi(t) \|_{L^{\gamma}}
		\right\}
\end{align}%
where
\begin{align}%
\nonumber
	\eta &= -\frac12 + \frac{s}{2} + \frac{n}{2}\left( \frac{p}{r}- \frac12 \right),\\
\label{def_sigma_1}
	\sigma_1 &= \max\left\{ 1, \frac{r}{p} \right\},\\
\label{def_sigma_2}
	\sigma_2 &=
	\begin{cases}
	2 
	& \text{if } 2s\geq n,
	\\
	\min \left\{ 2, \frac{2n}{p(n-2s)} \right\} 
	&\text{if } 2s<n.
	\end{cases}
\end{align}%
We remark that the condition
$p \le 1+ \frac{\min\{n,2\}}{n-2s}$ if $2s < n$
implies
$\sigma_1 \le \sigma_2$.
We also note that
the choices of the parameters
$\eta, \sigma_1, \sigma_2$
are quite natural.
Indeed, in the proof of Theorems \ref{thm_lwp} and \ref{thm_gwp},
we use the norm of $Y(T)$ with $\psi = \mathcal{N}(u)$,
that is, the nonlinear term of the equation \eqref{dw}.
Roughly speaking, if $u$ belongs to $X(\infty)$,
then $\| |u(t)|^p \|_{L^{2}} = \| u(t) \|_{L^{2p}}^p$ decays like
$\langle t \rangle^{-\frac{n}{2}\left( \frac{p}{r} - \frac12 \right)}$,
and hence, we expect
$\| |\nabla|^{s-1} |u(t)|^p \|_{L^2} \lesssim \langle t \rangle^{-\eta}$.
Also, roughly speaking,
if $u \in X(\infty)$, then
the Sobolev embedding implies $u(t) \in L^r(\mathbb{R}^n) \cap L^{2n/(n-2s)}(\mathbb{R}^n)$,
and hence, we expect
$|u(t)|^p \in L^{r/p}(\mathbb{R}^n) \cap L^{2n/(p(n-2s))}(\mathbb{R}^n)$
and
$\| |u(t)|^p \|_{L^{\gamma}}$
decays like
$\langle t \rangle^{-\frac{n}{2}\left( \frac{p}{r} - \frac{1}{\gamma} \right)}$
for $\gamma \in [\sigma_1, \sigma_2]$.

\subsection{Local and global existence}
Hereafter, we assume the condition in Theorem  \ref{thm_lwp}.

\begin{lemma}\label{lem_duh}
Under the assumption of Theorem \ref{thm_lwp}, we have
\begin{align*}%
	\left\| \int_0^t \mathcal{D}(t-\tau) \psi(\tau)\,d\tau \right\|_{X(T)}
	&\lesssim \int_0^T \| \psi \|_{Y(\tau)} \,d\tau
\end{align*}%
for any $0<T \le 1$ and $\psi \in Y(T)$, and
\begin{align*}%
	\left\| \int_0^t \mathcal{D}(t-\tau) \psi(\tau)\,d\tau \right\|_{X(T)}
	\lesssim \| \psi \|_{Y(T)}
		\begin{cases}
			1 &\big(
				p\ge 1+\frac{2r}{n}
				\big),\\
			\langle T \rangle^{1-\frac{n}{2r}(p-1)} & \big(p< 1+\frac{2r}{n}\big)
		\end{cases}
\end{align*}%
for any $T > 0$ and $\psi \in Y(T)$.
\end{lemma}
\begin{proof}
The first assertion is easily derived
by modifying the following argument, and we omit its proof.

For the second assertion,
first, we estimate
\begin{align*}%
	&\langle t \rangle^{\frac{n}{2}\left( \frac{1}{r} - \frac{1}{2}\right) + \frac{s}{2}}
				\left\|
					|\nabla|^s \int_0^t \mathcal{D}(t-\tau) \psi(\tau)\,d\tau
				\right\|_{L^2}
		+ \langle t \rangle^{\frac{n}{2}\left( \frac{1}{r} - \frac{1}{2} \right)}
				\left\|
					\int_0^t \mathcal{D}(t-\tau) \psi(\tau)\,d\tau
				\right\|_{L^2} \\
	&\quad + \left\|
				\int_0^t \mathcal{D}(t-\tau) \psi(\tau)\,d\tau
			\right\|_{L^r} \\
	&\le I + I\!I + I\!I\!I,
\end{align*}%
where
\begin{align*}%
	I &:=
	\langle t \rangle^{\frac{n}{2}\left( \frac{1}{r} - \frac{1}{2}\right) + \frac{s}{2}}
				\int_0^{t/2} \| |\nabla|^s \mathcal{D}(t-\tau) \psi(\tau)\|_{L^2} \,d\tau
		+ \langle t \rangle^{\frac{n}{2}\left( \frac{1}{r} - \frac{1}{2} \right)}
				\int_0^{t/2} \| \mathcal{D}(t-\tau) \psi(\tau) \|_{L^2} \,d\tau \\
		&\quad + \int_0^{t/2} \| \mathcal{D}(t-\tau) \psi(\tau) \|_{L^r} \, d\tau,
\end{align*}%
\begin{align*}%
	I\!I&:= \langle t \rangle^{\frac{n}{2}\left( \frac{1}{r} - \frac{1}{2}\right) + \frac{s}{2}}
				\int_{t/2}^t \| |\nabla|^s \mathcal{D}(t-\tau) \psi(\tau)\|_{L^2} \,d\tau
		+ \langle t \rangle^{\frac{n}{2}\left( \frac{1}{r} - \frac{1}{2} \right)}
				\int_{t/2}^t \| \mathcal{D}(t-\tau) \psi(\tau) \|_{L^2} \,d\tau,
\end{align*}%
and
\begin{align*}%
	I\!I\!I &:= \int_{t/2}^t \| \mathcal{D}(t-\tau) \psi(\tau) \|_{L^r} \, d\tau.
\end{align*}%
For the term $I$, applying Theorem \ref{thm_lplq} with 
$p=2$ and
$q = \sigma_1$,
we have
\begin{align*}%
	I &\lesssim \langle t \rangle^{\frac{n}{2}\left( \frac{1}{r} - \frac{1}{2}\right) + \frac{s}{2}}
				\int_0^{t/2}
		\left( \langle t -\tau \rangle^{
			-\frac{n}{2}\left( \frac{1}{\sigma_1}-\frac12 \right)-\frac{s}{2}
			}
		\| \psi (\tau) \|_{L^{\sigma_1}} \right. \\
	&\qquad\qquad\qquad\qquad\qquad\qquad\qquad\qquad \left.+ e^{-\frac{t-\tau}{2}} \langle t-\tau \rangle^{\delta_2}
			\| |\nabla|^{s-1} \chi_{>1}(\nabla) \psi(\tau) \|_{L^2} \right)
		\,d\tau \\
	&\quad
	+ \langle t \rangle^{\frac{n}{2}\left( \frac{1}{r} - \frac{1}{2} \right)}
	\int_0^{t/2} \left( \langle t - \tau \rangle^{
			- \frac{n}{2}\left( \frac{1}{\sigma_1} - \frac12 \right)}
				\| \psi (\tau) \|_{L^{\sigma_1}} \right. \\
	&\qquad\qquad\qquad\qquad\qquad\qquad\qquad\qquad \left. + e^{-\frac{t-\tau}{2}} \langle t -\tau \rangle^{\delta_2}
			\| |\nabla|^{-1} \chi_{>1}(\nabla) \psi (\tau) \|_{L^2} \right) \,d\tau \\
	&\quad
	+ \int_0^{t/2} \left( \langle t-\tau \rangle^{
				-\frac{n}{2}\left( \frac{1}{\sigma_1} - \frac1r \right)}
			\| \psi (\tau) \|_{L^{\sigma_1}}
			+ e^{-\frac{t-\tau}{2}} \langle t-\tau \rangle^{\delta_r}
			\| \chi_{>1}(\nabla) \psi(\tau) \|_{H^{\beta -1}_r} \right) \,d\tau.	
\end{align*}%
We note that $\sigma_1 < r$,
\begin{align*}%
	\| \psi (\tau) \|_{L^{\sigma_1}}
	&\le \langle \tau \rangle^{-\frac{n}{2}\left( \frac{p}{r} - \frac{1}{\sigma_1} \right)}
	\| \psi \|_{Y(T)}
\end{align*}%
and
\begin{align}%
\label{chi_nabla_psi}
	\| |\nabla|^{-1} \chi_{>1}(\nabla) \psi (\tau) \|_{L^2}
	&\le \| |\nabla|^{s-1} \chi_{>1}(\nabla) \psi(\tau) \|_{L^2}
	\le \langle \tau \rangle^{-\eta} \| \psi \|_{Y(T)},\\
\label{eq3.5}
	\| \chi_{>1} (\nabla) \psi(\tau) \|_{H^{\beta -1}_r}
	& \lesssim \| \psi \|_{L^{\mu}}
	\lesssim \langle \tau \rangle^{-\frac{n}{2}\left( \frac{p}{r} - \frac{1}{\mu} \right)}
	\| \psi \|_{Y(T)}
\end{align}%
with
$\mu = \max \{ \sigma_1, (\frac{1}{r}+\frac{1-\beta}{n})^{-1} \} \in [\sigma_1, \sigma_2]$,
because
$p \le 1+ \frac{2}{n-2s}$.
Here, we note that $\beta \le 1$ holds (see Remark \ref{rem_r}).
Therefore, a straightforward calculation gives
\begin{align*}%
	I &\lesssim \| \psi \|_{Y(T)} \begin{cases}
			1 & \big(
				p\ge 1+\frac{2r}{n}
				\big),\\
			\langle T \rangle^{1-\frac{n}{2r}(p-1)} & \big(p< 1+\frac{2r}{n}\big).
		\end{cases}
\end{align*}%

Next, we estimate $I\!I$.
We divide the estimate of $I\!I$ into two cases.
First, when $s > 1$,
from Theorem \ref{thm_lplq} with
$(q,s_1, s_2) = (2,s,s-1)$
and
$(q,s_1,s_2) = (\sigma_2, 0,0)$,
we obtain
\begin{align*}%
	I\!I &\lesssim
	\langle t \rangle^{\frac{n}{2}\left( \frac{1}{r} - \frac{1}{2}\right) + \frac{s}{2}}
	\int_{t/2}^t
		\left( \langle t-\tau \rangle^{-\frac12}
		\| |\nabla|^{s-1} \psi (\tau) \|_{L^2} \right.
	\\
	&\qquad\qquad\qquad\qquad\qquad\qquad\qquad \left. + e^{-\frac{t-\tau}{2}} \langle t - \tau \rangle^{\delta_2}
		\| |\nabla|^{s-1} \chi_{>1}(\nabla) \psi (\tau) \|_{L^2} \right)\, d\tau\\
	&\quad
	+ \langle t \rangle^{\frac{n}{2}\left( \frac{1}{r} - \frac{1}{2} \right)}
	\int_{t/2}^t \left(
		\langle t-\tau \rangle^{-\frac{n}{2}\left( \frac{1}{\sigma_2} - \frac12 \right)}
		\| \psi (\tau) \|_{L^{\sigma_2}}\right.
	\\
	&\qquad\qquad\qquad\qquad\qquad\qquad\qquad \left.+ e^{-\frac{t-\tau}{2}} \langle t-\tau \rangle^{\delta_2}
		\| |\nabla|^{-1} \chi_{>1}(\nabla)\psi(\tau) \|_{L^2} \right)\,d\tau.
\end{align*}%
We compute
\begin{align}%
\nonumber
	\| |\nabla|^{s-1} \psi (\tau) \|_{L^2} 
	&\le \langle \tau \rangle^{-\eta} \| \psi \|_{Y(T)},\\
\label{psi_L_sigma2}
	\| \psi (\tau) \|_{L^{\sigma_2}}
	&\le \langle \tau \rangle^{-\frac{n}{2}\left( \frac{p}{r} - \frac{1}{\sigma_2} \right)}
	\| \psi \|_{Y(T)}
\end{align}%
and
we have \eqref{chi_nabla_psi}.
Then, by a simple calculation, we deduce
\begin{align*}%
	I\!I &\lesssim \| \psi \|_{Y(T)} \begin{cases}
			1 &\big(
				p\ge 1+\frac{2r}{n}
				\big),\\
			\langle T \rangle^{1-\frac{n}{2r}(p-1)} &\big(p< 1+\frac{2r}{n}\big).
		\end{cases}
\end{align*}%

On the other hand, when $0\le s \le 1$,
applying Theorem \ref{thm_lplq} with
$(q,s_1,s_2) = (\sigma_2, s, 0)$ and $(q, s_1, s_2) = (\sigma_2, 0,0)$,
we compute
\begin{align*}%
	I\!I &\le
	\langle t \rangle^{\frac{n}{2}\left( \frac{1}{r} - \frac{1}{2}\right) + \frac{s}{2}}
	\int_{t/2}^t
		\left( \langle t-\tau \rangle^{-\frac{n}{2}\left( \frac{1}{\sigma_2}
		- \frac{1}{2}\right)-\frac{s}{2}}
		\|  \psi (\tau) \|_{L^{\sigma_2}}\right. \\
	&\qquad\qquad\qquad\qquad\qquad\qquad\qquad \left.+ e^{-\frac{t-\tau}{2}} \langle t - \tau \rangle^{\delta_2}
		\| |\nabla|^{s-1} \chi_{>1}(\nabla) \psi (\tau) \|_{L^2} \right)\, d\tau\\
	&\quad
	+ \langle t \rangle^{\frac{n}{2}\left( \frac{1}{r} - \frac{1}{2} \right)}
	\int_{t/2}^t \left(
		\langle t-\tau \rangle^{-\frac{n}{2}\left( \frac{1}{\sigma_2} - \frac12 \right)}
		\| \psi (\tau) \|_{L^{\sigma_2}}\right.
	\\
	&\qquad\qquad\qquad\qquad\qquad\qquad\qquad \left.+ e^{-\frac{t-\tau}{2}} \langle t-\tau \rangle^{\delta_2}
		\| |\nabla|^{-1} \chi_{>1}(\nabla)\psi(\tau) \|_{L^2} \right)\,d\tau.
\end{align*}%
Noting that $-\frac{n}{2}\left( \frac{1}{\sigma_2} - \frac{1}{2}\right)-\frac{s}{2}>-1$
since $0\le s \le1$,
and using \eqref{psi_L_sigma2} and \eqref{chi_nabla_psi}, we conclude
\begin{align*}%
	I\!I
	\lesssim \| \psi \|_{Y(T)}
		\begin{cases}
			1 &\big(  p\ge 1+\frac{2r}{n} \big),\\
			\langle T \rangle^{1-\frac{n}{2r}(p-1)} & \big(p< 1+\frac{2r}{n}\big).
		\end{cases}
\end{align*}%
Thus, for both cases we obtain the desired estimate.

Finally, we estimate $I\!I\!I$.
Theorem \ref{thm_lplq}
with $p=r$, $s_1 = s_2 =0$, and 
$q \in [\sigma_1, \sigma_2]$ determined later
implies
\begin{align*}%
	I\!I\!I
	&\lesssim \int_{t/2}^t \left(
			\langle t-\tau \rangle^{-\frac{n}{2}\left(\frac{1}{q}-\frac{1}{r} \right)}
			\| \psi (\tau) \|_{L^{q}}
			+ e^{-\frac{t-\tau}{2}}\langle t-\tau \rangle^{\delta_r}
			\| \chi_{>1}(\nabla) \psi(\tau) \|_{H^{\beta-1}_r} \right)\,d\tau.
\end{align*}%
First, we have \eqref{eq3.5}.
Next, for the first term, we choose $q$ as
\begin{align*}%
	q = \begin{cases}
		\sigma_1 + \epsilon
		&\left(
			-\frac{n}{2}\left( \frac{1}{\sigma_1}-\frac{1}{r}\right) = -1
			\right),\\
		\sigma_1 &( \mbox{otherwise} )
		\end{cases}
\end{align*}%
with sufficiently small $\epsilon>0$.
Consequently, we have
\begin{align*}%
	I\!I\!I
	\lesssim \| \psi \|_{Y(T)}
		\begin{cases}
			1 &\big(  p\ge 1+\frac{2r}{n} \big),\\
			\langle T \rangle^{1-\frac{n}{2r}(p-1)} & \big(p< 1+\frac{2r}{n}\big),
		\end{cases}
\end{align*}%
which completes the proof.
\end{proof}

\begin{lemma}\label{lem_non}
Under the assumption of Theorem \ref{thm_lwp}, we have
\begin{align*}%
	\left\| \mathcal{N}(u) \right\|_{Y(T)}
	&\lesssim \left\| u \right\|_{X(T)}^p
\end{align*}%
for any $T>0$ and $u\in X(T)$.
\end{lemma}
\begin{proof}
Let $\gamma \in [\sigma_1, \sigma_2]$.
Then, by the assumption \eqref{N},
we have
$|\mathcal{N}(u)| \lesssim |u|^{p}$
and we calculate
\begin{align*}%
	\langle t \rangle^{\frac{n}{2}\left( \frac{p}{r} - \frac{1}{\gamma} \right)}
	\| \mathcal{N}(u) \|_{L^{\gamma}}
	&\lesssim \left( \langle t \rangle^{\frac{n}{2}\left( \frac{1}{r} - \frac{1}{p\gamma} \right)}
			\| u \|_{L^{p \gamma}} \right)^p.
\end{align*}%
By the definition of $\sigma_1$ and $\sigma_2$,
we see that
\begin{align*}%
	r \le p \sigma_1 \quad \mbox{and}\quad
	p \sigma_2 \begin{cases} < \infty &(2s\ge n),\\
						\le \frac{2n}{n-2s} &(2s < n).
				\end{cases}		
\end{align*}%
This and the interpolation
between $L^r$ and $L^2$ in the case that $p\gamma \in [r,2]$
and between $\dot{H}^s$ and $L^2$ in other case imply
$\langle t \rangle^{\frac{n}{2}\left( \frac{1}{r} - \frac{1}{p\gamma} \right)}
\| u \|_{L^{p \gamma}} \le \| u \|_{X(T)}$
 (see also \cite[Lemma 2.3]{HaKaNa04DIE} or \cite[Lemma 2.3]{IkeInWa17}),
which leads to
\begin{align*}%
	\langle t \rangle^{\frac{n}{2}\left( \frac{p}{r} - \frac{1}{\gamma} \right)}
	\| \mathcal{N}(u) \|_{L^{\gamma}}
	&\lesssim \| u \|_{X(T)}^p.
\end{align*}%

Next, we prove the following inequalities.
\begin{align*}%
\begin{array}{ll}
	\langle t \rangle^{\eta} \| |\nabla|^{s-1} \mathcal{N}(u) \|_{L^2}
	\lesssim \| u \|_{X(T)}^p 
	& (s > 1),\quad
	\\
	\langle t \rangle^{\eta} \| |\nabla|^{s-1} \chi_{>1}(\nabla) \mathcal{N}(u) \|_{L^2}
	\lesssim \| u \|_{X(T)}^p 
	& (0\le s \le 1).
\end{array}
\end{align*}%
First, we consider the case of $0 \le s \le 1$. 
Now, we have
\begin{align*}%
	\| |\nabla|^{s-1} \chi_{>1}(\nabla) \mathcal{N}(u(t)) \|_{L^2}
	&\lesssim \| \mathcal{N}(u(t)) \|_{L^{\rho}}
	\lesssim \| u(t) \|_{L^{\rho p}}^p,
\end{align*}%
provided that $\rho$ satisfies
\begin{align}%
\label{rho}
	\max \left\{  1, \frac{2n}{n-2s+2} \right\} \le \rho  \le 2,
\end{align}%
since
$\| |\nabla|^{s-1} \chi_{>1}(\nabla) f \|_{L^2} \sim \|  \left\langle \nabla \right\rangle^{s-1} f \|_{L^2}$
and
$\| f \|_{L^2} \lesssim \| \left\langle \nabla \right\rangle^{1-s} f \|_{L^{\rho}}$
holds for
$\rho$
satisfying \eqref{rho} by the Sobolev embedding.
Since $1< p \le \min \{ 1+ \frac{2}{n-2s}, \frac{2n}{n-2s} \}$,
there exists $\rho$ satisfying \eqref{rho} such that 
\begin{align*}%
	r \le \rho p \le \frac{2n}{n-2s}.
\end{align*}%
In fact, the family of intervals
$\{ [r/\rho, 2n/\rho(n-2s)]\}_{\max \left\{  1, \frac{2n}{n-2s+2} \right\} \le \rho  \le 2}$
covers the interval 
$(1,\min\{1+ 2/(n-2s), 2n/(n-2s)\}]$.
Namely, 
\begin{align*}%
	\left( 1, \min \left\{ 1+ \frac{2}{n-2s}, \frac{2n}{n-2s} \right\}\right]
	\subset \mathop{\bigcup}_{\max \left\{  1, \frac{2n}{n-2s+2} \right\} \le \rho  \le 2} 
	\left[ \frac{r}{\rho}, \frac{2n}{\rho (n-2s)} \right]	.
\end{align*}%
Then, we obtain 
\begin{align*}%
	\| u(t) \|_{L^{\rho p}}^p
	\lesssim \left( \langle t \rangle^{-\frac{n}{2} \left( \frac{1}{\rho p} - \frac{1}{2} \right) }
		\| u \|_{X(T)} \right)^p.
\end{align*}%

Next, we consider the case of $s>1$.
Let $\tilde{s}$ be the fractional part of $s$, namely, $\tilde{s} := s - [s]$.
Using the Fa\'a di Bruno formula
(see the proof of Lemma 2.5 in \cite{IkeInWa17}),
the fractional Leibniz rule, and the fractional chain rule
(see for example \cite[Proposition 3.1, 3.3]{CW}), we see that
\begin{align*}
	& \| |\nabla|^{s-1} \mathcal{N} (u) \| _{L^2}
	= \| |\nabla|^{\tilde{s}} |\nabla|^{[s]-1} \mathcal{N} (u) \| _{L^2} \\
	& \lesssim \| |u|^{p-[s]} \| _{L^{q_0}}
		\sum _{\substack{k = (k_1, \dots , k_{[s]}) \in \mathbb{Z}^{[s]}_{\ge 0} \\ |k| = [s]-1}}
		\| |\nabla|^{k_1+\tilde{s}} u\|_{L^{q_1(k)}}
		\prod _{j=2}^{[s]} \| |\nabla|^{k_j} u \| _{L^{q_j(k)}} \\
	& \lesssim \| u \| _{L^{(p-[s])q_0}}^{p-[s]}
	\sum _{\substack{k = (k_1, \dots , k_{[s]}) \in \mathbb{Z}^{[s]}_{\ge 0} \\ |k| = [s]-1}}
	\| |\nabla|^{k_1+\tilde{s}+n (\frac{1}{2} - \frac{1}{q_1(k)})} u\|_{L^2}
	\prod _{j=2}^{[s]} \| |\nabla|^{k_j+n(\frac{1}{2}-\frac{1}{q_j(k)})} u \| _{L^2},
\end{align*}
where $q_0$ and $q_j(k) \ (j = 1,2, \ldots,[s])$ are given so that 
\begin{align}
\label{cond}
\left\{
\begin{array}{l}
	\frac{1}{2} = \frac{1}{q_0} + \frac{1}{q_1(k)} + \cdots + \frac{1}{q_{[s]}(k)},\\
	2 < q_j (k) <\infty \text{ for } j=1,2,\ldots,[s],\\
	\frac{r}{p-[s]} \le q_0 
		\begin{cases}
			< \infty &\ \text{if }\ 2s \ge n,\\
			\le \frac{2n}{(p-[s])(n-2s)} &\ \text{if }\ 2s<n,
		\end{cases} \\
	k_1+\tilde{s}+n \left( \frac{1}{2} - \frac{1}{q_1(k)} \right)
	\le s,\\
	k_j+n \left( \frac{1}{2}-\frac{1}{q_j(k)} \right)
	\le s \ \text{ for } \ j =2, 3, \ldots, [s].
\end{array}
\right.
\end{align}
We postpone the proof of the existence of such exponents to Appendix \ref{appendixb}.
We also note that when $[s]=1$, the above inequality is interpreted as
\begin{align*}
	\| |\nabla|^{s-1} \mathcal{N} (u) \| _{L^2}
	\lesssim 
	\| u \| _{L^{(p-[s])q_0}}^{p-[s]} 
	\| |\nabla|^{\tilde{s}+n (\frac{1}{2} - \frac{1}{q_1(k)})} u\|_{L^2}.
\end{align*}
Finally, by the interpolation, we obtain 
\begin{align*}%
	\langle t \rangle^{\eta} \| |\nabla|^{s-1} \mathcal{N}(u) \|_{L^2}
	&\lesssim \| u \|_{X(T)}^p.
\end{align*}%
The proof is complete.
\end{proof}

Now we prove the local existence of the solution.
\begin{proof}[Proof of Theorem \ref{thm_lwp}]
Let $T \in (0,1]$.
First, we note that
Theorem \ref{thm_lplq} implies
\begin{align}%
\label{est_lin_pt}
	\| (\partial_t + 1) \mathcal{D}(t) \varepsilon u_0
		+ \mathcal{D}(t) \varepsilon u_1 \|_{X(T)}
	&\le C_0 \varepsilon ( \| u_0 \|_{H^s \cap H^{\beta}_r} + \|u_1 \|_{H^{s-1}\cap L^r} )
\end{align}%
with some constant
$C_ 0 = C_0(n,s,r) > 0$,
which is independent of $T$.
We put
\begin{align}%
\label{mep}
	M(\varepsilon)
	:= 2  C_0 \varepsilon ( \| u_0 \|_{H^s \cap H^{\beta}_r} + \|u_1 \|_{H^{s-1}\cap L^r} )
\end{align}%
and consider $X(T,M(\varepsilon))$.
For
$u \in X(T,M(\varepsilon))$,
we define a mapping
\begin{align}%
\label{psi}
	\Psi (u) (t) &:=
	(\partial_t + 1) \mathcal{D}(t) \varepsilon u_0
		+ \mathcal{D}(t) \varepsilon u_1
		+ \int_0^t \mathcal{D}(t-\tau) \mathcal{N}(u(\tau)) d\tau.
\end{align}%
For $u,v \in X(T,M(\varepsilon))$, by Lemmas \ref{lem_duh} and \ref{lem_non},
we have
\begin{align}%
\label{psi_bdd}
	\| \Psi (u) \|_{X(T)}
	&\le \frac{M(\varepsilon)}{2} + C_1 T M(\varepsilon)^{p}
\end{align}%
with some constant $C_1=C_1(n,s,r,p)>0$.
Moreover, we have
\begin{align}%
\label{psi_dif}
	\| \Psi(u) - \Psi(v) \|_{Z(T)}
	&\le C_2 T (2M(\varepsilon))^{p-1} \| u - v \|_{Z(T)}
\end{align}%
with some constant
$C_2=C_2(n,s,r,p)>0$.
Indeed, we put
\begin{align}%
\label{eq3.15}
	\frac{1}{q} 
	&= \begin{cases}
		\displaystyle \delta &(2s \ge n),\\
		\displaystyle \left( \frac{2n}{(p-1)(n-2s)} \right)^{-1} &(2s < n),
	\end{cases} \\
\label{eq3.14}
	\frac{1}{\gamma} 
	&= \frac{1}{q} + \frac12,
\end{align}%
with sufficiently small $\delta > 0$.
We note that the condition
$p \le 1 + \frac{n}{n-2s}$ if $2s < n$ implies
$\gamma \ge 1$.
Then, we calculate
\begin{align}%
\label{eq3.16}
	&\left\| \int_0^t \mathcal{D}(t-\tau) \left( \mathcal{N}(u(\tau)) - \mathcal{N}(v(\tau)) \right)
			\,d\tau \right\|_{L^2} \\
\notag
	&\lesssim \int_0^t \left(
			\langle t -\tau \rangle^{-\frac{n}{2}\left(\frac{1}{\gamma}-\frac12\right)}
			\left\| \mathcal{N}(u(\tau)) - \mathcal{N}(v(\tau)) \right\|_{L^{\gamma}}
			\right. \\
\notag
	&\qquad\qquad  \left.
			+ e^{-\frac{t-\tau}{2}}\langle t-\tau \rangle^{\delta_2}
			\| \langle \nabla \rangle^{-1}
				(\mathcal{N}(u(\tau)) - \mathcal{N}(v(\tau))) \|_{L^2}\right)\,d\tau \\
\notag
	&\lesssim \int_0^t
		\left\| \mathcal{N}(u(\tau)) - \mathcal{N}(v(\tau)) \right\|_{L^{\gamma}}\, d\tau \\
\notag
	&\lesssim \int_0^t
		\| u(\tau) - v(\tau) \|_{L^2}
		( \| u(\tau) \|_{L^{q(p-1)}} + \| v(\tau) \|_{L^{q(p-1)}} )^{p-1} \,d\tau\\
\notag
	&\lesssim T \| u - v\|_{Z(T)} ( \| u \|_{X(T)} + \|v \|_{X(T)} )^{p-1}.
\end{align}%
Here, we have used the embedding
\begin{align*}%
	\| \phi \|_{L^2} &\lesssim \| \langle \nabla \rangle \phi \|_{L^{\gamma}}
\end{align*}%
with noting that
$\frac12 \ge \frac{1}{\gamma} - \frac{1}{n}$,
for the second inequality.
Moreover, we have also used the H\"{o}lder inequality with the relation
$\frac{1}{\gamma} = \frac12 + \frac{1}{q}$.

Therefore, from \eqref{psi_bdd} and \eqref{psi_dif},
taking $T>0$ sufficiently small depending on
$C_1, C_2$, and $M(\varepsilon)$,
we see that
$\Psi$ is a contraction mapping in $X(T,M(\varepsilon))$ with the metric of $Z(T)$,
and it has a unique fixed point $u$,
which is a $H^s\cap L^r$-mild solution of \eqref{dw}.

We show that $ u \in C([0,T) ; H^{s}(\mathbb{R}^n) \cap L^r(\mathbb{R}^n))$. 
The solution $u$ satisfies the integral equation
\begin{align*}%
	u (t) &=
	(\partial_t + 1) \mathcal{D}(t) \varepsilon u_0
		+ \mathcal{D}(t) \varepsilon u_1
		+ \int_0^t \mathcal{D}(t-\tau) \mathcal{N}(u(\tau)) d\tau.
\end{align*}%
Since the linear part of the solution obviously satisfies this property,
it suffices to show that
\begin{align*}
	\int_0^t \mathcal{D}(t-\tau) \mathcal{N}(u(\tau)) \,d\tau
	\in C([0,T) ; H^{s}(\mathbb{R}^n) \cap L^r(\mathbb{R}^n)).
\end{align*}
By Theorem \ref{thm_lplq}, we have
\begin{align*}
	\| |\nabla|^s \mathcal{D}(t-\tau) \mathcal{N}(u(\tau)) \|_{L^2}
	&\lesssim \| \chi_{\le 1}(\nabla) \mathcal{N}(u(\tau)) \|_{L^{\sigma_2}}
		+ \| |\nabla|^{s-1} \chi_{> 1}(\nabla) \mathcal{N}(u(\tau)) \|_{L^2},\\
	\| \mathcal{D}(t-\tau) \mathcal{N}(u(\tau)) \|_{L^2}
	&\lesssim \| \chi_{\le 1}(\nabla) \mathcal{N}(u(\tau)) \|_{L^{\sigma_2}}
		+ \| |\nabla|^{-1} \chi_{> 1}(\nabla) \mathcal{N}(u(\tau)) \|_{L^2},\\
	\| \mathcal{D}(t-\tau) \mathcal{N}(u(\tau)) \|_{L^r}
	&\lesssim \| \chi_{\le 1}(\nabla) \mathcal{N}(u(\tau)) \|_{L^{\sigma_1}}
		+ \| \chi_{> 1}(\nabla) \mathcal{N}(u(\tau)) \|_{H_{r}^{\beta}}
\end{align*}
for $t \in [0,T)$ and $\tau \in [0,t]$ and  the right-hand sides are bounded independently of $t$.
Therefore, we can apply the dominated convergence theorem in the Bochner integral
and thus the continuity holds.

We also prove that
$\partial_t u \in C([0,T) ; H^{s-1}(\mathbb{R}^n))$.
Since the linear part of the solution obviously satisfies this property,
it suffices to show that
\begin{align}
\label{pt_int_dt_N}
	\partial_t \int_0^t \mathcal{D}(t-\tau) \mathcal{N}(u(\tau)) \,d\tau
	\in C([0,T) ; H^{s-1}(\mathbb{R}^n)).
\end{align}
We note that
$\| \mathcal{N}(u) \|_{Y(T)}$ is bounded.
This and Theorem \ref{thm_lplq} implies
\begin{align*}
	\| |\nabla|^s \mathcal{D}(t-\tau) \mathcal{N}(u(\tau)) \|_{L^2}
	&\lesssim \| \chi_{\le 1}(\nabla) \mathcal{N}(u(\tau)) \|_{L^{\sigma_2}}
		+ \| |\nabla|^{s-1} \chi_{> 1}(\nabla) \mathcal{N}(u(\tau)) \|_{L^2},\\
	\| \mathcal{D}(t-\tau) \mathcal{N}(u(\tau)) \|_{L^2}
	&\lesssim \| \chi_{\le 1}(\nabla) \mathcal{N}(u(\tau)) \|_{L^{\sigma_2}}
		+ \| |\nabla|^{-1} \chi_{> 1}(\nabla) \mathcal{N}(u(\tau)) \|_{L^2}
\end{align*}
for $t \in [0,T)$ and $\tau \in [0,t]$,
and the right-hand sides are bounded.
Therefore, for any fixed $t \in [0, T)$,
$\mathcal{D}(t-\cdot) \mathcal{N}(u(\cdot)) \in L^{\infty} (0,t ; H^s(\mathbb{R}^n))$
holds.
Moreover, this also leads to
$\partial_t \mathcal{D}(t-\cdot) \mathcal{N}(u(\cdot))
\in L^{\infty}(0,t ; H^{s-1}(\mathbb{R}^n))$.
Indeed, by \eqref{lplqdt} in Theorem \ref{thm_lplq}, we have
\begin{align*}
	\| |\nabla|^{s-1} \partial_t \mathcal{D}(t-\tau) \mathcal{N}(u(\tau)) \|_{L^2}
	&\lesssim \| \chi_{\le 1}(\nabla) \mathcal{N}(u(\tau)) \|_{L^{\sigma_2}}
		+ \| |\nabla|^{s-1} \chi_{> 1}(\nabla) \mathcal{N}(u(\tau)) \|_{L^2},\\
	\|\partial_t \mathcal{D}(t-\tau) \mathcal{N}(u(\tau)) \|_{L^2}
	&\lesssim \| \chi_{\le 1}(\nabla) \mathcal{N}(u(\tau)) \|_{L^{\sigma_2}}
		+ \| |\nabla|^{-1} \chi_{> 1}(\nabla) \mathcal{N}(u(\tau)) \|_{L^2}
\end{align*}
for $t \in [0,T)$ and $\tau \in [0,t]$,
and the right-hand sides are bounded independently of $t$.
Therefore, by the Lebesgue convergence theorem in the Bochner integral 
and $\mathcal{D}(0)=0$,
we see that
\begin{align*}
	\partial_t \int_0^t \mathcal{D}(t-\tau) \mathcal{N}(u(\tau)) \,d\tau
	= \int_0^t \partial_t \mathcal{D}(t-\tau) \mathcal{N}(u(\tau)) \,d\tau,
\end{align*}
which implies \eqref{pt_int_dt_N}.

Next, we show that under the assumptions of Theorem \ref{thm_lwp},
for any fixed $T_0>0$,
the $H^s$-mild solution on the interval $[0,T_0)$ is unique.
This also implies the uniqueness of $H^s\cap L^r$-mild solution,
because an $H^s\cap L^r$-mild solution is also an $H^s$-mild solution.
Let $T_0 >0$ and fix it, and let
$u, v$
be $H^s$-mild solutions of \eqref{dw} with same initial data
$\varepsilon(u_0, u_1) \in H^s(\mathbb{R}^n) \times H^{s-1}(\mathbb{R}^n)$.
Let $T_1 \in (0,T_0)$ be an arbitrary number.
We define
\begin{align*}%
	\| \phi \|_{X_2(T_1)}
	:= \sup_{0<t<T_1}\left\{
	\langle t \rangle^{\frac{s}{2}} \| |\nabla|^s \phi (t) \|_{L^2}
		+ \| \phi(t) \|_{L^2} \right\}.
\end{align*}%
Then, there exists a constant $M>0$ such that
$\| u \|_{X_2(T_1)}+\|v\|_{X_2(T_1)} \le M$.
From this and the same argument as deriving \eqref{psi_dif} with $r=2$,
we can see that
\begin{align*}%
	\| u - v \|_{Z(T)}
	\le C_3 M^{p-1} \int_0^{T} \| u - v \|_{Z(\tau)}\,d\tau
\end{align*}%
for any $T \in [0,T_1]$
with some constant $C_3 = C_3(n,s,p)>0$.
Indeed, for $t \in [0,T]$, we can obtain
\begin{align*}%
	&\left\| \int_0^t \mathcal{D}(t-\tau) \left( \mathcal{N}(u(\tau)) - \mathcal{N}(v(\tau)) \right)
			\,d\tau \right\|_{L^2} \\
	&\lesssim \int_0^t
		\| u(\tau) - v(\tau) \|_{L^2}
		( \| u(\tau) \|_{L^{q(p-1)}} + \| v(\tau) \|_{L^{q(p-1)}} )^{p-1} \,d\tau\\
	&\lesssim  ( \| u \|_{X_2(T)} + \|v \|_{X_2(T)} )^{p-1} \int_0^t \| u - v\|_{Z(\tau)} d\tau \\
	&\lesssim  M^{p-1} \int_0^t \| u - v\|_{Z(\tau)} d\tau
\end{align*}%
in the same manner as \eqref{eq3.16},
where $q$ and $\gamma$ are defined in \eqref{eq3.15} and \eqref{eq3.14}.
Thus, the Gronwall inequality implies
$u \equiv v$
on
$[0,T_1]$.
Since $T_1 \in (0,T_0)$ is arbitrary,
we have $u \equiv v$ on $[0,T_0)$.

We next prove the locally Lipschitz continuity of the solution map
\begin{align*}%
	(H^s(\mathbb{R}^n)\cap H_{r}^{\beta}(\mathbb{R}^n))
	\times (H^{s-1}(\mathbb{R}^n) \cap L^r(\mathbb{R}^n))
		&\to C([0,T) ; L^2(\mathbb{R}^n)),\\
	\varepsilon(u_0, u_1) &\mapsto u.
\end{align*}%
Let $M > 0$ and we consider the ball
\begin{align*}%
	B(M) := \{ &\varepsilon (u_0, u_1) \in (H^s(\mathbb{R}^n)\cap H_{r}^{\beta}(\mathbb{R}^n))
	\times (H^{s-1}(\mathbb{R}^n) \cap L^r(\mathbb{R}^n)) ; \\
	& \| \varepsilon (u_0, u_1) \|_{(H^s\cap H_r^{\beta})\times (H^{s-1}\cap L^r)}
		\le M \},
\end{align*}%
for the initial data.
Then, by the proof of the existence part above, we find $T>0$ depending only on $M$
such that for each $\varepsilon (u_0, u_1)$, there exists a unique solution
$u \in X(T, 2M)$.
Let $\varepsilon (u_0, u_1)$, $\varepsilon (v_0, v_1) \in B(M)$ and
let $u$, $v$ be the associated solutions, respectively.
From this and the same argument before, we see that
\begin{align*}%
	\| u - v \|_{Z(T_1)}
	&\lesssim \| \varepsilon(u_0- v_0)\|_{H^s \cap H_r^{\beta}}
		+ \| \varepsilon(u_1-v_1) \|_{H^{s-1}\cap L^r}\\
	&\quad + (2M)^{p-1} \int_0^{T_1} \| u - v \|_{Z(\tau)}\,d\tau.
\end{align*}%
Therefore, the Gronwall inequality implies
\begin{align*}%
	\| u - v \|_{Z(T_1)}
	&\lesssim  \| \varepsilon(u_0- v_0) \|_{H^s \cap H_r^{\beta}}
		+ \| \varepsilon(u_1-v_1) \|_{H^{s-1}\cap L^r},
\end{align*}%
which shows the locally Lipschitz continuity of the solution map.

Finally, we prove the blow-up alternative for $H^s$-mild solution, namely,
$T_2(\varepsilon) < \infty$ implies \eqref{bu}.
Let us suppose $T_2(\varepsilon) < \infty$ and
\begin{align*}
	\liminf_{t\to T_2(\varepsilon)} \| (u, \partial_t u) (t) \|_{H^s\times H^{s-1}} < \infty.
\end{align*}
Then, there exist a constant $M>0$ and a sequence
$\{ t_m \}_{m=1}^{\infty} \subset [0,T_2(\varepsilon))$
such that
$t_m \to T_2(\varepsilon) \ (m\to \infty)$
and
\begin{align*}
	\| (u, \partial_t u) (t_m) \|_{H^s \times H^{s-1}} \le M \quad ( m \ge 1).
\end{align*}
We note that, from the above proof of the local existence of the $H^s$-mild solution
in the case $r=2$,
we deduce that there exists $T_1>0$ independent of $\{ t_m \}_{m=1}^{\infty}$
such that we can construct the solution
\begin{align*}
	u \in C([t_m, t_m+T_1) ; H^s(\mathbb{R}^n)),\quad
	\partial_t u \in C([t_m, t_m+ T_1) ; H^{s-1}(\mathbb{R}^n)).
\end{align*}
However, letting $m \to \infty$,
this contradicts the definition of
the lifespan $T_2(\varepsilon)$.
This completes the proof.
\end{proof}

\begin{proof}[Proof of Theorem \ref{thm_gwp}]
Let $p \ge 1+ \frac{2r}{n}$. 
Let $T>0$ be an arbitrary finite number.
We define $M(\varepsilon)$ by \eqref{mep} and consider the mapping \eqref{psi}
on
$X(T, M(\varepsilon))$.
Then, applying Lemmas \ref{lem_duh} and \ref{lem_non}, we have
for $u, v \in X(T, M(\varepsilon))$,
\begin{align}
\label{eq3.18}
	\| \Psi (u) \|_{X(T)} &\le \frac{M(\varepsilon)}{2} + C_1 M(\varepsilon)^p,\\
\notag
	\| \Psi (u) - \Psi (v) \|_{Z(T)} &\le
		C_2 (2M(\varepsilon))^{p-1} \| u - v \|_{Z(T)}
\end{align}
with some constants $C_1, C_2 > 0$ independent of $T$,
instead of \eqref{psi_bdd} and \eqref{psi_dif}, respectively.
Indeed, the first estimate is a direct consequence of Lemmas 3.1 and 3.2.
The second estimate is obtained by a similar way to the proof of (3.13).
More precisely, we have
\begin{align*}%
	&\left\| \int_0^t \mathcal{D}(t-\tau)
			\left( \mathcal{N}(u(\tau)) - \mathcal{N}(v(\tau)) \right) \, d\tau \right\|_{L^2}\\
	&\lesssim
		\int_0^t \left(
			\langle t - \tau \rangle^{-\frac{n}{2}\left( \frac{1}{\gamma}-\frac{1}{2} \right) }
				\| \mathcal{N}(u(\tau)) - \mathcal{N}(v(\tau)) \|_{L^{\gamma}} \right. \\
		&\left. \qquad \qquad
			+ e^{-\frac{t-\tau}{2}}\langle t-\tau \rangle^{\delta_2}
				\| \langle \nabla \rangle^{-1}
					\left( \mathcal{N}(u(\tau)) - \mathcal{N}(v(\tau)) \right) \|_{L^2} \right)\,d\tau \\
	&\lesssim
		\int_0^t \langle t - \tau \rangle^{-\frac{n}{2}\left( \frac{1}{\gamma}-\frac{1}{2} \right) }
			\| \mathcal{N}(u(\tau)) - \mathcal{N}(v(\tau)) \|_{L^{\gamma}} \, d\tau \\
	&\lesssim
		\int_0^t \langle t - \tau \rangle^{-\frac{n}{2}\left( \frac{1}{\gamma}-\frac{1}{2} \right) }
			\| u(\tau) - v(\tau) \|_{L^2}
				\left( \| u(\tau) \|_{L^{q(p-1)}} + \| v(\tau) \|_{L^{q(p-1)}} \right)^{p-1}
				 \, d\tau \\
	&\lesssim
		\int_0^t \langle t - \tau \rangle^{-\frac{n}{2}\left( \frac{1}{\gamma}-\frac{1}{2} \right) }
			\langle \tau \rangle^{-\frac{n}{2} \left( \frac{1}{r} - \frac{1}{q(p-1)} \right)(p-1)}
			\, d\tau
			\cdot \| u - v \|_{Z(T)} \left( \| u \|_{X(T)} + \| v \|_{X(T)} \right)^{p-1},
\end{align*}%
where
$\gamma$ and $q$ are defined by \eqref{eq3.14} and \eqref{eq3.15}
and in the last inequality we have used the Sobolev inequality with
$s' = n \left( \frac{1}{2} - \frac{1}{q(p-1)} \right)$
and the interpolation inequality
to obtain
\begin{align*}%
	&\langle \tau \rangle^{\frac{n}{2} \left( \frac{1}{r} - \frac{1}{q(p-1)} \right)}
		\| u (\tau) \|_{L^{q(p-1)}} \\
	&\le \langle \tau \rangle^{\frac{n}{2} \left( \frac{1}{r} - \frac{1}{q(p-1)} \right)}
		\| |\nabla|^{s'} u(\tau) \|_{L^2} \\
	&= \langle \tau \rangle^{\frac{n}{2} \left( \frac{1}{r} - \frac{1}{2} \right) + \frac{s'}{2}}
		\| |\nabla|^{s'} u(\tau) \|_{L^2} \\
	&= \left( \langle \tau \rangle^{\frac{n}{2} \left( \frac{1}{r} - \frac{1}{2} \right)}
		\| u(\tau) \|_{L^2} \right)^{1-\frac{s'}{s}}
		\left( \langle \tau \rangle^{\frac{n}{2} \left( \frac{1}{r} - \frac{1}{2} \right)+\frac{s}{2}}
		\| |\nabla|^s u(\tau) \|_{L^2} \right)^{\frac{s'}{s}} \\
	&\le \| u(\tau) \|_{X(T)}.
\end{align*}%
Here we also note that the definition of $q$ implies
$0\le s' \le s$.
Therefore, it suffices to show
\begin{align}%
\label{thm14diff}
	\int_0^t \langle t - \tau \rangle^{-\frac{n}{2}\left( \frac{1}{\gamma}-\frac{1}{2} \right) }
			\langle \tau \rangle^{-\frac{n}{2} \left( \frac{1}{r} - \frac{1}{q(p-1)} \right)(p-1)}
			\, d\tau
	\lesssim 1
\end{align}%
holds under the condition
$p \ge 1+ \frac{2r}{n}$.
To prove \eqref{thm14diff}, we divide the integral into
\begin{align*}%
	&\int_0^t \langle t - \tau \rangle^{-\frac{n}{2}\left( \frac{1}{\gamma}-\frac{1}{2} \right) }
			\langle \tau \rangle^{-\frac{n}{2} \left( \frac{1}{r} - \frac{1}{q(p-1)} \right)(p-1)}
			\, d\tau \\
	&= \int_0^{t/2} \langle t - \tau \rangle^{-\frac{n}{2}\left( \frac{1}{\gamma}-\frac{1}{2} \right) }
			\langle \tau \rangle^{-\frac{n}{2} \left( \frac{1}{r} - \frac{1}{q(p-1)} \right)(p-1)}
			\, d\tau \\
	&\quad + \int_{t/2}^t \langle t - \tau \rangle^{-\frac{n}{2}\left( \frac{1}{\gamma}-\frac{1}{2} \right) }
			\langle \tau \rangle^{-\frac{n}{2} \left( \frac{1}{r} - \frac{1}{q(p-1)} \right)(p-1)}
			\, d\tau \\
	&=: A + B.
\end{align*}%
The term $A$ is estimated as
\begin{align*}%
	A&\lesssim \langle t \rangle^{-\frac{n}{2}\left( \frac{1}{\gamma}-\frac{1}{2} \right) }
		\int_0^{t/2} \langle \tau \rangle^{-\frac{n}{2} \left( \frac{1}{r} - \frac{1}{q(p-1)} \right)(p-1)}
			\, d\tau.
\end{align*}%
By noting that
$-\frac{n}{2}\left( \frac{1}{\gamma}-\frac{1}{2} \right) <0$,
if
$-\frac{n}{2} \left( \frac{1}{r} - \frac{1}{q(p-1)} \right)(p-1) \le -1$,
then we immediately have $A \lesssim 1$.
Otherwise, we also easily compute
\begin{align*}%
	A &\lesssim \langle t \rangle^{-\frac{n}{2}\left( \frac{1}{\gamma}-\frac{1}{2} \right) }
		\langle t \rangle^{-\frac{n}{2} \left( \frac{1}{r} - \frac{1}{q(p-1)} \right)(p-1) + 1} \\
	&= \langle t \rangle^{-\frac{n}{2r}(p-1)+1} \\
	&\lesssim 1,
\end{align*}%
since $\frac{1}{\gamma} = \frac{1}{q} + \frac{1}{2}$ and $p \ge 1+\frac{2r}{n}$.
Next, the term $B$ is estimated as
\begin{align*}%
	B&\lesssim
		\langle t \rangle^{-\frac{n}{2} \left( \frac{1}{r} - \frac{1}{q(p-1)} \right)(p-1)}
		\int_{t/2}^t \langle t - \tau \rangle^{-\frac{n}{2}\left( \frac{1}{\gamma}-\frac{1}{2} \right) }
			\, d\tau.
\end{align*}%
When $2s\ge n$,
$\frac{1}{\gamma} = \frac{1}{2} + \delta$
with sufficiently small
$\delta$,
and hence,
$-\frac{n}{2}\left( \frac{1}{\gamma}-\frac{1}{2} \right) > -1$
and we have
\begin{align*}%
	B &\lesssim
		\langle t \rangle^{-\frac{n}{2} \left( \frac{1}{r} - \frac{1}{q(p-1)} \right)(p-1)}
		\langle t \rangle^{-\frac{n}{2}\left( \frac{1}{\gamma}-\frac{1}{2} \right) +1 } \\
		&= \langle t \rangle^{-\frac{n}{2r}(p-1)+1} \\
	&\lesssim 1,
\end{align*}%
since $\frac{1}{\gamma} = \frac{1}{q} + \frac{1}{2}$ and $p \ge 1+\frac{2r}{n}$.
When $2s<n$, the definition of $\gamma$ is
$\frac{1}{\gamma} = \left( \frac{2n}{(p-1)(n-2s)} \right)^{-1} + \frac{1}{2}$,
which leads to
$-\frac{n}{2}\left( \frac{1}{\gamma}-\frac{1}{2} \right)  = -\frac{n-2s}{4}(p-1) \ge - \frac{1}{2}$,
since
$p \le 1+ \frac{2}{n-2s}$.
Hence, we have
\begin{align*}%
	B&\lesssim
		\langle t \rangle^{-\frac{n}{2} \left( \frac{1}{r} - \frac{1}{q(p-1)} \right)(p-1)}
		\langle t  \rangle^{-\frac{n}{2}\left( \frac{1}{\gamma}-\frac{1}{2} \right) + 1} \\
	&= \langle t \rangle^{-\frac{n}{2r}(p-1)+1} \\
	&\lesssim 1,
\end{align*}%
since $\frac{1}{\gamma} = \frac{1}{q} + \frac{1}{2}$ and $p \ge 1+\frac{2r}{n}$.
Thus, we have \eqref{thm14diff}.
Therefore, by taking $\varepsilon_0>0$ so that
\begin{align}
\label{eq3.20}
	C_1 M(\varepsilon)^p \le \frac{M(\varepsilon)}{2},\quad
	C_2 (2M(\varepsilon))^{p-1} \le \frac{1}{2}
\end{align}
holds, 
the mapping $\Psi$ becomes a contraction mapping on $X(T, M(\varepsilon))$
with respect to the metric of $Z(T)$,
and thus we have the solution $u \in X(T, M(\varepsilon))$. 
Moreover, the uniqueness has been already proved in the proof of Theorem \ref{thm_lwp}.
Since $T$ is arbitrary, the solution is global. Moreover, $u \in X(\infty, M(\varepsilon))$. 
Indeed, since we have
$\| u \|_{X(T)} \leq M(\varepsilon)$
for arbitrary $T \in (0,\infty)$ from \eqref{eq3.18} and \eqref{eq3.20}, 
we get $\| u \|_{X(\infty)} \leq M(\varepsilon)$.
We also have 
\begin{align}
\label{eq3.21}
	\| \mathcal{N}(u) \|_{Y(\infty)} 
	\lesssim \| u\|_{X(\infty)}^p 
	\lesssim M(\varepsilon)
	=C \varepsilon ( \| u_0 \|_{H^s \cap H^{\beta}_r} + \|u_1 \|_{H^{s-1}\cap L^r} )
\end{align} 
since Lemma \ref{lem_non} holds for any $T \in (0,\infty)$.
\end{proof}

\section{Global existence of an $H^s$-mild solution for small data}
In this section, we give a sketch of the proof of Theorem \ref{thm_gwp2}.
The argument is the same as that of Theorem \ref{thm_gwp} and
we give only the difference.
For $T \in (0,\infty]$, $s \ge 0$, and $r \in (1,2]$, we define
\begin{align*}%
	\tilde{X}(T)
	:=\left\{ \phi \in L^{\infty}(0,T ; H^s(\mathbb{R}^n) ) ; \| \phi \|_{\tilde{X}(T)} < \infty \right\}
\end{align*}%
with the norm
\begin{align*}%
	\| \phi \|_{\tilde{X}(T)} := \sup_{t \in [0,T)}
		\left\{ \langle t \rangle^{\frac{n}{2}\left(\frac{1}{r}-\frac12\right) + \frac{s}{2}}
			\| |\nabla|^s \phi(t) \|_{L^2}
		+ \langle t \rangle^{\frac{n}{2}\left(\frac{1}{r}-\frac12\right)}
			\| \phi(t) \|_{L^2}
		\right\}.
\end{align*}%
For $T \in (0,\infty]$, $s \ge 0$, $r \in (1,2]$, and $1<p \le 2n/(n-2s)$ if $n>2s$ and $1<p<\infty$ if $n\le 2s$, 
we also define the function space
\begin{align*}%
	\tilde{Y}(T)
	:=\left\{ \psi \in L^{\infty}(0,T ; \dot{H}^{s-1}(\mathbb{R}^n)
	\cap L^{\sigma_1(\mathbb{R}^n) \cap L^{\sigma_2}(\mathbb{R}^n) )}
	;\| \psi \|_{\tilde{Y}(T)}<\infty \right\}
\end{align*}
with the norms
\begin{align*}%
	\| \psi \|_{\tilde{Y}(T)} := \sup_{t \in [0,T)} \left\{
		\langle t \rangle^{\eta} \| |\nabla|^{s-1} \psi(t) \|_{L^2}
			+ \sup_{\gamma \in [\sigma_1, \sigma_2]}
			\langle t \rangle^{\frac{n}{2}\left(\frac{p}{r}-\frac{1}{\gamma}\right)}
			\| \psi(t) \|_{L^{\gamma}}
		\right\}
\end{align*}%
for $s \ge 1$ and
\begin{align*}%
	\| \psi \|_{\tilde{Y}(T)} := \sup_{t \in [0,T)} \left\{
		\langle t \rangle^{\eta} \| |\nabla|^{s-1} \chi_{>1}(\nabla) \psi(t) \|_{L^2}
			+ \sup_{\gamma \in [\sigma_1, \sigma_2]}
			\langle t \rangle^{\frac{n}{2}\left(\frac{p}{r}-\frac{1}{\gamma}\right)}
			\| \psi(t) \|_{L^{\gamma}}
		\right\}
\end{align*}%
for $0\le s <1$,
where
$\eta = -\frac{1}{2} + \frac{s}{2} + \frac{n}{2}\left( \frac{p}{r} - \frac12 \right)$,
$\sigma_1 = \max\{ 1, \frac{2}{p} \}$,
and
$\sigma_2 =2$ if $2s\geq n$ and $\sigma_2 = \min \{ 2, \frac{2n}{p(n-2s)} \}$ if $2s<n$.
We remark that the assumption
$r > \frac{\sqrt{n^2+16n}-n}{4}$
implies that
\begin{align*}%
	1+ \frac{2r}{n} > \frac{2}{r}.
\end{align*}%
From this and the assumption $p \ge 1+ \frac{2r}{n}$,
it follows that
$r > \frac{2}{p}$,
which also implies
$\sigma_1 < r$.
Therefore, we can apply the same argument as
Lemmas \ref{lem_duh} and \ref{lem_non}.
We note that the norm of $\tilde{X}(T)$ does not involve $L^r$-norm,
and hence, we do not need to estimate
the $L^r$-norm of the solution.
From this and repeating the same procedure as the proof of Theorem \ref{thm_gwp},
we can prove the existence of a global solution.

\section{Asymptotic behavior of the global solution}
In this section, we give a proof of Theorem \ref{thm_ab}.

\begin{proof}[Proof of Theorem \ref{thm_ab}]
Let $u$ be the global solution constructed in Theorem \ref{thm_gwp} and let $ p > 1+2r/n$.
We see that
\begin{align*}
	u(t) - \varepsilon \mathcal{G}(t) (u_0 + u_1) = w_L(t) + w_{NL}(t),
\end{align*}
where
\begin{align*}
		w_L(t) &= \varepsilon  \left( \mathcal{D}(t) - \mathcal{G}(t) \right) (u_0 + u_1)
		+\varepsilon \partial_t \mathcal{D}(t) u_0,\\
		w_{NL}(t) &= \int_0^t \mathcal{D}(t-\tau)
			\mathcal{N}(u(\tau))\,d\tau.
\end{align*}
From Theorems \ref{thm_lplq} and \ref{thm_lplqdiff}, we obtain
\begin{align*}
	\| |\nabla|^s w_L(t) \|_{L^2}
	&\lesssim \varepsilon \langle t \rangle^{%
			-\frac{n}{2}\left(\frac{1}{r}-\frac12\right)-\frac{s}{2}-1}%
		( \| u_0 \|_{H^s \cap L^r} + \| u_1 \|_{H^{s-1} \cap L^r}) ,\\
	\|  w_L(t) \|_{L^2} 
	&\lesssim \varepsilon \langle t \rangle^{-\frac{n}{2}\left(\frac{1}{r}-\frac12\right)-1}
		(\| u_0 \|_{L^2 \cap L^r} + \|u_1\|_{H^{-1} \cap L^r}),\\
	\| w_L(t) \|_{L^r}
	&\lesssim \varepsilon \langle t \rangle^{-1}
		( \| u_0 \|_{H^{\beta}_r} + \|u_1\|_{L^r}).
\end{align*}
Therefore, it suffices to estimate $w_{NL}(t)$.
We note that
Lemma \ref{lem_non} gives
\begin{align*}
	\| \mathcal{N}(u) \|_{Y(\infty)}
\le C\varepsilon ( \| u_0 \|_{H^s \cap H^{\beta}_r} + \| u_1 \|_{H^{s-1} \cap L^r})
\end{align*}
since $u$ belongs to $X(\infty,M(\varepsilon))$ (see Theorem \ref{thm_gwp} and \eqref{eq3.21}).
In the same manner as in the proof of Lemma \ref{lem_duh}, we see that
\begin{align*}
	&\langle t \rangle^{\frac{n}{2}\left(\frac{1}{r}-\frac12\right)+\frac{s}{2}}
		\| |\nabla|^s w_{NL}(t) \|_{L^2} \\
	&\lesssim \| \mathcal{N}(u) \|_{Y(\infty)}
		\langle t \rangle^{\frac{n}{2}\left(\frac{1}{r}-\frac12\right)+\frac{s}{2}}
		\Big[ \int_0^{\frac{t}{2}}
		\langle t-\tau \rangle^{-\frac{n}{2}\left( \frac{1}{\sigma_1}-\frac12\right)-\frac{s}{2}}
		\langle \tau \rangle^{-\frac{n}{2}\left( \frac{p}{r} - \frac{1}{\sigma_1} \right)}
		\,d\tau \\
	&\quad
		+ \int_{\frac{t}{2}}^t
			\langle t-\tau \rangle^{-\frac12}
			\langle \tau \rangle^{-\eta} \,d\tau
		+ \int_0^t
			e^{-\frac{t-\tau}{2}} \langle t-\tau \rangle^{\delta_2}
			\langle \tau \rangle^{-\eta} \,d\tau
		\Big]\\
	&\lesssim \varepsilon ( \| u_0 \|_{H^s \cap H^{\beta}_r} + \| u_1 \|_{H^{s-1} \cap L^r})
		\langle t \rangle^{%
			- \min\{ \frac{n}{2r}(p-1)-1,
			\frac{n}{2}\left(\frac{1}{\sigma_1}-\frac{1}{r}\right)\}
			+\delta}.
\end{align*}
Hereafter, $\delta$ denotes an arbitrary small positive number,
and the implicit constants are dependent on $\delta$.
We also have
\begin{align*}
	&\langle t \rangle^{\frac{n}{2}\left(\frac{1}{r}-\frac12\right)}
		\| w_{NL}(t) \|_{L^2} \\
	&\lesssim \| \mathcal{N}(u) \|_{Y(\infty)}
	\langle t \rangle^{\frac{n}{2}\left(\frac{1}{r}-\frac12\right)}
	\Big[ \int_0^{\frac{t}{2}}
		\langle t - \tau \rangle^{-\frac{n}{2}\left(\frac{1}{\sigma_1}-\frac{1}{2}\right)}
		\langle \tau \rangle^{-\frac{n}{2}\left(\frac{p}{r}-\frac{1}{\sigma_1}\right)}
		\,d\tau \\
	&\quad + \int_{\frac{t}{2}}^t
		\langle t -\tau \rangle^{-\frac{n}{2}\left(\frac{1}{\sigma_2}-\frac{1}{2}\right)}
		\langle \tau \rangle^{-\frac{n}{2}\left(\frac{p}{r}-\frac{1}{\sigma_2}\right)}
		\,d\tau
		+ \int_0^t
			e^{-\frac{t-\tau}{2}} \langle t-\tau \rangle^{\delta_2}
			\langle \tau \rangle^{-\eta} \,d\tau
		\Big]\\
	&\lesssim \varepsilon ( \| u_0 \|_{H^s \cap H^{\beta}_r} + \| u_1 \|_{H^{s-1} \cap L^r})
		\langle t \rangle^{%
			- \min\{ \frac{n}{2r}(p-1)-1,
			\frac{n}{2}\left(\frac{1}{\sigma_1}-\frac{1}{r}\right)\}
			+\delta
			}
\end{align*}
and
\begin{align*}
	&\| w_{NL}(t) \|_{L^r}\\
	&\lesssim  \| \mathcal{N}(u) \|_{Y(\infty)}
		\Big[ \int_0^{t/2} \langle t-\tau \rangle^{
				-\frac{n}{2}\left( \frac{1}{\sigma_1} - \frac1r \right)}
			\langle \tau \rangle^{-\frac{n}{2}\left(\frac{p}{r}-\frac{1}{\sigma_1}\right)}
			\,d\tau\\
	&\quad
		+\int_{t/2}^t
			\langle t-\tau \rangle^{-\frac{n}{2}\left(\frac{1}{q}-\frac{1}{r} \right)}
			\langle \tau \rangle^{-\frac{n}{2}\left(\frac{p}{r}-\frac{1}{q}\right)}
			\,d\tau
			+\int_0^t e^{-\frac{t-\tau}{2}} \langle t-\tau \rangle^{\delta_r}
			\langle \tau \rangle^{-\frac{n}{2}\left( \frac{p}{r} - \frac{1}{\mu} \right)}
			\,d\tau
		\Big]\\
	&\lesssim  \varepsilon ( \| u_0 \|_{H^s \cap H^{\beta}_r} + \| u_1 \|_{H^{s-1} \cap L^r})
		\langle t \rangle^{%
			- \min\{
				\frac{n}{2r}(p-1)-1,
				\frac{n}{2}\left(\frac{1}{\sigma_1}-\frac{1}{r}\right),
				\frac{n}{2}\left(\frac{p}{r}-\frac{1}{q} \right)
				\}
			+\delta
			},
\end{align*}
where
$q = \min\{ r, \sigma_2 \} = \min\{r, \frac{2n}{p(n-2s)_+} \} \in [\sigma_1, \sigma_2]$
and $\sigma_1, \sigma_2$ are defined in \eqref{def_sigma_1}, \eqref{def_sigma_2}.
Here we note that $q > r/p$ holds under the assumption of Theorem \ref{thm_ab}.
Consequently, we have
\begin{align*}
	\| |\nabla|^s (u(t) - \varepsilon \mathcal{G}(u_0+u_1)) \|_{L^2}
	&\lesssim  \langle t \rangle^{%
		- \frac{n}{2}\left(\frac{1}{r}-\frac12\right) - \frac{s}{2}
		- \min\{1, \frac{n}{2r}(p-1)-1,
			\frac{n}{2}\left(\frac{1}{\sigma_1}-\frac{1}{r}\right)\}
			+\delta
		}\\ %
	&\quad \times
		\varepsilon ( \| u_0 \|_{H^s \cap H^{\beta}_r} + \| u_1 \|_{H^{s-1} \cap L^r}),\\
		\| u(t) - \varepsilon \mathcal{G}(u_0+u_1) \|_{L^2}
	&\lesssim  \langle t \rangle^{%
		- \frac{n}{2}\left(\frac{1}{r}-\frac12\right)
		- \min\{1, \frac{n}{2r}(p-1)-1,
			\frac{n}{2}\left(\frac{1}{\sigma_1}-\frac{1}{r}\right)\}
			+ \delta
		}\\ %
	&\quad \times
		\varepsilon ( \| u_0 \|_{H^s \cap H^{\beta}_r} + \| u_1 \|_{H^{s-1} \cap L^r}),\\
		\| u(t) - \varepsilon \mathcal{G}(u_0+u_1) \|_{L^r}
	&\lesssim \langle t \rangle^{%
			- \min\{
				1, \frac{n}{2r}(p-1)-1,
				\frac{n}{2}\left(\frac{1}{\sigma_1}-\frac{1}{r}\right),
				\frac{n}{2}\left(\frac{p}{r}-\frac{1}{q} \right)
			\}
			+ \delta
		}\\ %
	&\quad \times
		\varepsilon ( \| u_0 \|_{H^s \cap H^{\beta}_r} + \| u_1 \|_{H^{s-1} \cap L^r}).
\end{align*}
This completes the proof.
\end{proof}

\section{Lower bound of the lifespan}
In this section, we prove Theorem \ref{thm_lb}.
\begin{proof}[Proof of Theorem \ref{thm_lb}]
By Lemma \ref{lem_duh}, when $1<p<1+2r/n$, we see that
there exists a constant $C_1>0$ such that
\begin{align*}%
	\left\| \int_0^t \mathcal{D}(t-\tau) \psi(\tau)\,d\tau \right\|_{X(T)}
	\le C_1 \| \psi \|_{Y(T)}
			\langle T \rangle^{1-\frac{n}{2r}(p-1)}
\end{align*}%
for any $T > 0$ and $\psi \in Y(T)$.
Also, it follows from Lemma \ref{lem_non} that
$\| \mathcal{N}(u) \|_{Y(T)} \le C_1 \| u \|_{X(T)}^p$
for any $T>0$ and $u \in X(T)$.
Based on these estimates, we repeat the argument in the proof of Theorem \ref{thm_lwp}.
First, we have \eqref{est_lin_pt} for any $T>0$.
We define a constant $M(\varepsilon)$ by \eqref{mep}.
Let $\Psi$ be a mapping on $X(T,M(\varepsilon))$ given by \eqref{psi}.
Then, instead of \eqref{psi_bdd} and \eqref{psi_dif}, in this case we obtain
\begin{align*}%
	\| \Psi(u) \|_{X(T)}
	&\le \frac{M(\varepsilon)}{2} + C_1\langle T \rangle^{1-\frac{n}{2r}(p-1)} M(\varepsilon)^{p},\\
	\| \Psi(u) - \Psi(v) \|_{Z(T)}
	&\le C_2 \langle T \rangle^{1-\frac{n}{2r}(p-1)} M(\varepsilon)^{p-1}
	\| u - v \|_{Z(T)}
\end{align*}%
with some constant $C_2>0$
for any $u, v \in X(T,M(\varepsilon))$ and $T>0$.
Therefore, as long as
\begin{align}%
\label{t_ep}
	\max\{ C_1, C_2 \} \langle T \rangle^{1-\frac{n}{2r}(p-1)}
	(2C_0( \| u_0 \|_{H^s \cap H^{\beta}_r} + \|u_1 \|_{H^{s-1}\cap L^r} ))^{p-1}
		\varepsilon^{p-1} \le \frac12
\end{align}%
holds, the mapping $\Psi$ is contractive on $X(T,M(\varepsilon))$
with respect to the metric of $Z(T)$,
and we can construct a unique local solution.
We take $\varepsilon_1>0$ sufficiently small so that
\begin{align*}%
	\max\{ C_1, C_2 \}
	(2C_0( \| u_0 \|_{H^s \cap H^{\beta}_r} + \|u_1 \|_{H^{s-1}\cap L^r} ))^{p-1}
	\varepsilon_1^{p-1} < \frac12,
\end{align*}%
namely, \eqref{t_ep} formally holds for $T=0$ and $\varepsilon = \varepsilon_1$.
Let
$\varepsilon \in (0,\varepsilon_1]$
and let
$\tilde{T}(\varepsilon)$ be the first time which gives
the identity in the condition \eqref{t_ep}, that is,
\begin{align}%
\label{t_ep2}
	\langle \tilde{T}(\varepsilon) \rangle
	&= C_3 \varepsilon^{-1/\omega},
\end{align}%
where $\omega = \frac{1}{p-1} - \frac{n}{2r}$ and
$C_3 = C_3(n,s,r,p, \| u_0 \|_{H^s \cap H^{\beta}_r}, \|u_1 \|_{H^{s-1}\cap L^r})>0$
is a constant independent of
$\varepsilon$.
Since we can construct a solution until the time $\tilde{T}(\varepsilon)$,
we have $\tilde{T}(\varepsilon) \le T_r(\varepsilon)$.
Finally, retaking $\varepsilon_1>0$ smaller if needed so that
$C_3\varepsilon_1^{-1/\omega} \ge \sqrt{4/3}$,
we have for any $\varepsilon \in (0,\varepsilon_1 ]$,
\begin{align*}%
	T_r(\varepsilon) \ge \tilde{T}(\varepsilon) = \sqrt{ (C_3\varepsilon^{-1/\omega})^2-1}
	\ge \frac{C_3}{2}\varepsilon^{-1/\omega},
\end{align*}%
which gives the desired estimate \eqref{lb}.
\end{proof}

\section{Upper bound of the lifespan}
In this section, we prove Theorem \ref{thm_ub}.
More precisely, we will give more general blow-up result
(see Proposition \ref{prop_bu}), and Theorem \ref{thm_ub} will follow as a corollary of it.
A similar result in $L^1(\mathbb{R}^n)$-data setting and in
the Fujita-subcritical case, i.e. $p<1+\frac{2}{n}$ was obtained in
\cite{FuIkeWa_pre}.


We define a smooth compactly supported function
$\widetilde \psi  \in C^\infty([0,\infty);[0,1])$ as
	\[
	\widetilde \psi (y) =
	\begin{cases}
	1 &\mbox{if}\quad y \leq 1,\\
	\searrow &\mbox{if}\quad 1 < y < 2,\\
	0 &\mbox{if}\quad y \geq 2.
	\end{cases}
	\]
Set $\psi (x) := \widetilde{\psi} (|x|)$ for $x\in \mathbb{R}^n$,
and for $R>0$ let $\psi _R(x) = \psi (x/R)$.
For $p>1$, and $A>0$,
we define $\mu=\mu (p,A)$ by
	\begin{align*}
	\mu=\mu (p ,A) 
	:= \min \left\{ 1, \frac{p-1}{2}A
	\right\}.
	\end{align*}

For $n \in \mathbb{N}$, $p>1$, $l \in \mathbb{N}$ satisfying $l > 2p'$, where $p':=p/(p-1)$,
and for $\phi \in C_0^{\infty}(\mathbb{R}^n)$ with $\phi \ge 0$,
we also define $A(n, p,l,\phi )$ as
	\begin{equation}
	\label{1-2-1}
	A=A(n, p,l,\phi )
	:= 2^{p'-1} p'^{-\frac{1}{p}} p^{\frac{1-p'}{p}}
	\| \Phi^{p'} \phi ^{l-2p'} \|_{L^{1}(\mathbb R^n)}^{\frac{1}{p}}
	\| \phi ^{l} \|_{L^1(\mathbb R^n)}^{\frac{1}{p'}},
	\end{equation}
where
	\[
	\Phi
	:= \phi ^{2-l} \Delta(\phi^l)
	= l(l-1) \nabla \phi \cdot \nabla \phi + l \phi  \Delta \phi .
	\]
We derive an ordinary differential inequality
for the weighted average of the solution, i.e.
	\[
	I_\phi (t)=I_{\phi}[u](t)
	:= \int_{\mathbb R^n} u (t,x) \phi ^l(x) dx,
	\]
up to the constant $A(n,p,l,\phi)$
via the method in \cite{FuIkeWa_pre}.

At first, we note that for the local $H^s$-mild solution
$u(t) \in H^s(\mathbb{R}^n)$
constructed in Theorem \ref{thm_lwp},
we see that the linear part
$\varepsilon \mathcal{D}(t)(u_0+u_1) + \varepsilon \partial_t \mathcal{D}(t)u_0$
satisfies the equation
$u_{tt}-\Delta u + u_t = 0$
in $H^{s-2}(\mathbb{R}^n)$.
Also, from the proof of Theorem \ref{thm_lwp},
we know that
$\mathcal{N}(u)
\in C([0,T_2(\varepsilon)); \dot{H}^{s-1} \cap L^{\sigma_1} \cap L^{\sigma_2})$
and hence,
$\mathcal{N}(u) \in C([0,T_2(\varepsilon)); H^{-n(1/\sigma_2-1/2)}(\mathbb{R}^n))$.
Therefore, the nonlinear part
$\int_0^t \mathcal{D}(t-\tau) \mathcal{N}(u(\tau))d\tau$
satisfies the equation
$u_{tt}-\Delta u + u_t = \mathcal{N}(u)$
in $H^{-n(1/\sigma_2-1/2)}(\mathbb{R}^n)$.
Consequently, the local $H^s$-mild solution $u$
satisfies the equation
\eqref{dw}
in $H^{\min\{s-2, -n(1/\sigma_2-1/2)\}}(\mathbb{R}^n)$
for each
$t \in (0,T_2(\varepsilon))$.
This enables us to consider the coupling of the equation \eqref{dw}
with a test function, and we can derive an ODI
by the argument in \cite[Section 3]{FuIkeWa_pre}.

We recall a blow-up result obtained in \cite{FuIkeWa_pre},
which gives an upper estimate of the lifespan of solutions to \eqref{dw}
in a general setting.
We note that in the following theorem,
we do not need any condition on $p$ such as
$p < 1+\frac{2r}{n}$,
but we impose certain condition on the test function $\phi$.
\begin{proposition}[Proposition 1.3 in \cite{FuIkeWa_pre}]
\label{prop_odi}
Let $s \ge 0$, $p \in (1, \infty)$
and $(u_0,u_1) \in H^s(\mathbb R^n) \times H^{s-1}(\mathbb R^n)$,
and let $u$ be the associated $H^s$-mild solution to \eqref{dw}
constructed in Theorem \ref{thm_lwp}
(the solution obtained by applying Theorem \ref{thm_lwp} with $r=2$).
Assume that there exists $\phi \in \mathcal S (\mathbb R^n;[0,\infty))$
such that the inequalities
	\begin{align}
	0
	< I_\phi (0) - A(n,p,l,\phi )
	< 2^{\frac{1}{p-1}} \|\phi ^l\|_{L^1(\mathbb R^n)},
	\quad
	I_\phi '(0)
	> 0
	\label{eq:2}
	\end{align}
holds, where $l\in \mathbb{N}$ with $l>2p'$. Let
	\begin{align*}
	J_\phi (t)
	&:= I_\phi (t) - A(n,p,l,\phi ),\ \ \text{for}\ t\in [0,T_2(\varepsilon)),\\
	\widetilde J_\phi (0)
	&:= 2^{-\frac{1}{p-1}} \|\phi ^l\|_{L^1(\mathbb R^n)}^{-1} J_\phi(0),\\
	A_1
	&:= \frac{J_\phi '(0)}{J_\phi (0)}
	= \frac{I_\phi '(0)}{I_\phi (0) - A(n,p,l,\phi )}.
	\end{align*}
Then, the estimate
	\begin{align}
	J_\phi (t)
	\geq
	J_\phi (0)
	\bigg( 1 - \mu (p,A_1)
	\widetilde J_\phi (0)^{p-1}
	t
	\bigg)^{- \frac{2}{p-1}}
	\label{eq:rate}
	\end{align}
is valid for $t \in [0,T_2(\varepsilon))$.
Moreover,
the lifespan $T_2(\varepsilon)$ of the $H^s$-mild solution $u$ is estimated as
	\[
	T_2(\varepsilon) \leq
	 \mu (p,A_1)^{-1} \widetilde J_\phi(0)^{1-p} .
	\]
\end{proposition}

\begin{remark}
From the estimate \eqref{eq:rate},
we can expect that the blow-up rate of the solution
$u$ is similar to that of the
second order ordinary differential equation
$y''(t) = y(t)^p$,
which indicates the wave-like behavior of the solution near the blow-up time.
However, we remark that the estimate \eqref{eq:rate}
does not directly imply the blow-up rate of the solution,
because there is a possibility that
the blow-up time of
$\| (u, \partial_t u)(t) \|_{H^1 \times L^2}$
is earlier than that of $J_{\phi}(t)$.
\end{remark}

Proposition \ref{prop_odi} means that the condition \eqref{eq:2} is a sufficient condition for
the blow-up of a solution.
Indeed, we prove that if $p \in (1, 1+\frac{2r}{n})$, we take the test function
$\phi = \psi_{R(\varepsilon)}$
with an appropriate scaling parameter $R(\varepsilon)$, which will be defined later,
we ensure the condition \eqref{eq:2} for any $\varepsilon>0$ and
show an upper estimate of the lifespan of solutions to \eqref{dw}
for sufficiently small $\varepsilon>0$.

To state our main result, we introduce several nontation.
We denote by $|S_{n-1}|$ the surface area of the unit sphere $S_{n-1}$ in $\mathbb{R}^n$.
For $\varepsilon >0$, $r\in (1,2]$, $p\in (1,1+\frac{2r}{n})$,
$k \in (\frac{n}{r}, \min\{n, \frac{2}{p-1} \})$,
and $c_0,C_0>0$, set
\begin{align}%
\label{r}
	R(\varepsilon):=
	\max \left(2^{\frac{1}{n-k}},
	\left\{
		\frac{C_0|S_{n-1}|2^{n-k}\varepsilon}{
			(n-k)2^{\frac{1}{p-1}}\|\psi^l\|_{L^1}}\right\}^{\frac{1}{k}},
	\left\{\frac{4(n-k)A(n,p,l,\psi)}{c_0|S_{n-1}|\varepsilon}
		\right\}^{\frac{1}{\frac{2}{p-1}-k}}\right).
\end{align}%
where $A=A(n,p,l,\psi)$ is defined by (\ref{1-2-1}) with $\phi=\psi$.
\begin{proposition}
\label{prop_bu}
Let $n\in \mathbb{N}$, $r\in (1,2]$, $p \in (1,1+\frac{2r}{n})$,
$l\in \mathbb{N}$ with $l>2p'$,
$\frac{n}{r}<k<\min\{ n, \frac{2}{p-1} \}$
and $\varepsilon>0$.
We assume that
$(u_0, u_1) \in H^s(\mathbb{R}^n)  \times H^{s-1}(\mathbb{R}^n)$ satisfies
\begin{equation}
\label{1-6-2}
          C_0(1+|x|)^{-k}\ge u_0(x)\ge
          \begin{cases}
          c_0|x|^{-k},&\text{if}\ |x|\ge 1,\\
          0,&\text{if}\ |x|\le 1,
          \end{cases}
\ \ u_1(x)\ge
          \begin{cases}
          c_1|x|^{-k},&\text{if}\ |x|\ge 1,\\
          0,&\text{if}\ |x|\le 1,
          \end{cases}
\end{equation}
with some positive constants $c_0$, $c_1$ and $C_0$.
Then there exists $\varepsilon_2>0$ depending only on
$n,k,p,l,c_0, C_0$ such that for any $\varepsilon\in (0,\varepsilon_2]$ 
the associated $H^s$-mild solution $u$ of \eqref{dw} satisfies
\begin{align}%
\label{bl-rt}
	\int_{\mathbb{R}^n} u(t,x) \psi_{R(\varepsilon)}^{l}(x) dx
	\ge 
	\frac{c_0|S_{n-1}|}{4(n-k)}R(\varepsilon)^{n-k}\varepsilon
		\left( 1- \mu_0
			\varepsilon^{\frac{1}{\frac{1}{p-1}-\frac{k}{2}}} t
		\right)^{-\frac{2}{p-1}}
\end{align}%
for $t\in [0,T_2(\varepsilon))$,
with some constant $\mu_0 = \mu_0(n,p,k,c_1,C_0) > 0$
independent of $\varepsilon$.
Moreover, the lifespan $T_2(\varepsilon)$
is estimated as
\begin{align}
	\label{lf_upp2}
	T_2(\varepsilon) \le
	 \mu_0^{-1} 
	 \varepsilon^{-\frac{1}{\frac{1}{p-1}-\frac{k}{2}}}.
\end{align}
\end{proposition}

\begin{remark}
{\rm (i)}
The condition $\frac{n}{r}<k<n$ implies
$\langle x \rangle^{-k} \in L^r(\mathbb{R}^n)\setminus L^1(\mathbb{R}^n)$.

{\rm (ii)}
We can also consider the case $k=n$ if $p<1+\frac{2}{n}$ holds.
In this case, with
\begin{align*}%
	R(\varepsilon) :=
	\max\left(
		2,\left\{\frac{C_0\varepsilon|S_{n-1}|}{2^{\frac{1}{p-1}}
		\|\psi^l\|_{L^1}}\right\}^{\frac{1}{n}},
		\left\{\frac{2A(n,p,l,\psi)}{c_0\varepsilon |S_{n-1}|\log 2}
			\right\}^{\frac{1}{\frac{2}{p-1}-n}}\right),
\end{align*}%
we have the estimates
\begin{align*}%
	\int_{\mathbb{R}^n} u(t,x) \psi_{R(\varepsilon)}^{l}(x) dx
	\ge
	\frac{c_0 |S_{n-1}|}{2} \varepsilon \log R(\varepsilon)
		\left( 1- \mu_0
			\varepsilon^{\frac{1}{\frac{1}{p-1}-\frac{n}{2}}} t
		\right)^{-\frac{2}{p-1}}
\end{align*}%
and
\begin{align}%
\label{t2ep_l1}
	T_2(\varepsilon) \le \mu_0^{-1}
	 \varepsilon^{-\frac{1}{\frac{1}{p-1}-\frac{n}{2}}}.
\end{align}%

{\rm (iii)}
The case $k>n$, i.e. $u_0, u_1 \in L^1(\mathbb{R}^n)$,
has already been studied in \cite{FuIkeWa_pre},
and it follows that
\begin{align*}%
	\int_{\mathbb{R}^n} u(t,x) \psi_{R(\varepsilon)}^{l}(x) dx
	\ge \frac{c_0 |S_{n-1}|}{4(k-n)} \varepsilon
		\left( 1- \mu_0
			\varepsilon^{\frac{1}{\frac{1}{p-1}-\frac{n}{2}}} t
		\right)^{-\frac{2}{p-1}}
\end{align*}%
and \eqref{t2ep_l1}.
\end{remark}

Our proof of Proposition \ref{prop_bu} is based on
the combination of the argument of the proof of Corollary 1.4 in
\cite{FuIkeWa_pre}
and the proof of Corollary 1.8 in
\cite{FuOz2}.
In Corollary 1.4 in
\cite{FuIkeWa_pre}, blow-up mechanism
(upper estimate of lifespan and blow-up rate) to (\ref{dw}) in the Fujita-subcritical case, i.e.
$1<p<1+\frac{2}{n}$ and
$L^1(\mathbb{R}^n)$-setting is studied.
In Corollary 1.8 in \cite{FuOz2},
a similar result to the nonlinear Schr\"odinger equation is proved.

\begin{proof}[Proof of Proposition \ref{prop_bu}]
We apply Proposition \ref{prop_odi} with
$\phi=\psi_{R(\varepsilon)}$ for sufficiently small $\varepsilon>0$.
To do so, we first prove that the condition (\ref{eq:2}) with
$\phi=\psi_{R(\varepsilon)}$ is valid for any $\varepsilon>0$.
For any $\varepsilon>0$, let $R(\varepsilon)$ be given by \eqref{r}.
By the definition (\ref{1-2-1}) of $A$ and using changing variable, the identity
	\begin{equation}
	\label{2-1-3}
	A(n,p,l,\psi _{R(\varepsilon)})
	= A(n,p,l,\psi ) R(\varepsilon)^{n-2\frac{p'}{p}}
	\end{equation}
holds for any $\varepsilon>0$.
Since $u$ satisfies the initial condition of (\ref{dw}), by this identity, the equality
\begin{align}%
\label{2-1}
	I_{\psi_{R(\varepsilon)}}(0) - A(n,p,l,\psi _{R(\varepsilon)})
	=\varepsilon\int_{\mathbb{R}^n}u_0(x)\psi_{R(\varepsilon)}^l(x)dx
		-A(n,p,l,\psi)R(\varepsilon)^{n - 2 \frac{p'}{p}}
\end{align}%
holds for any $\varepsilon>0$.
By the definition of $R(\varepsilon)$, the estimates 
\begin{equation}
\label{2-2-1}
   R(\varepsilon)\ge 2^{\frac{1}{n-k}}>1
\end{equation}
hold for any $\varepsilon>0$. Since $u_0$ is non-negative and satisfies
$u_0(x)\ge c_0|x|^{-k}$ for $|x|\ge 1$,
by the properties of the function $\psi$,
the estimates
\begin{align}
\label{2-2}
       \int_{\mathbb{R}^n}u_0(x)\psi^l_{R(\varepsilon)}(x)dx
       &=\int_{0\le |x|\le 2R(\varepsilon)}u_0(x)\psi^l\left(\frac{x}{R(\varepsilon)}\right)dx
       \ge c_0\int_{1<|x|<R(\varepsilon)}|x|^{-k}dx\\
\nonumber
       &=\frac{c_0|S_{n-1}|}{n-k}\left(R(\varepsilon)^{n-k}-1\right)
       \ge\frac{c_0|S_{n-1}|}{2(n-k)}R(\varepsilon)^{n-k}
\end{align}
hold for any $\varepsilon>0$, where we have used the estimate (\ref{2-2-1}) to obtain the last inequality. By the definition of $R(\varepsilon)$, the estimate 
\[
  R(\varepsilon)\ge \left\{\frac{4(n-k)A(n,p,l,\psi)}{c_0|S_{n-1}|\varepsilon}\right\}^{\frac{1}{\frac{2}{p-1}-k}}
\] 
is true for any $\varepsilon>0$.
By the assumption $k<\frac{2}{p-1}$ and the estimates (\ref{2-1}) and (\ref{2-2}), the inequalities
\begin{align}
\label{2-4-1}
     I_{\psi_{R(\varepsilon)}}(0)-A(n,p,l,\psi_{R(\varepsilon)})
     &\ge \frac{c_0|S_{n-1}|}{2(n-k)} R(\varepsilon)^{n-k} \varepsilon
     		-A(n,p,l,\psi)R(\varepsilon)^{n - 2 \frac{p'}{p}}\\
     &=R(\varepsilon)^{n-k}\left\{\frac{c_0|S_{n-1}|\varepsilon}{2(n-k)}-A(n,p,l,\psi)
     	R(\varepsilon)^{-\left(\frac{2}{p-1}-k\right)}\right\}\notag\\
     &\ge \frac{c_0|S_{n-1}|}{4(n-k)}R(\varepsilon)^{n-k}\varepsilon>0
     \notag
\end{align}
hold for any $\varepsilon>0$.

Next we prove the inequality
$I'_{\psi_{R(\varepsilon)}}(0)>0$ for any $\varepsilon>0$.
In the same way as \eqref{2-2},
we have for any $\varepsilon>0$,
\begin{align}
\label{2-5}
         I'_{\psi_{R(\varepsilon)}}(0)
         &=\varepsilon\int_{\mathbb{R}^n}u_1(x)\psi^l_{R(\varepsilon)}(x)dx
         	>\frac{\varepsilon c_1|S_{n-1}|}{2(n-k)}R(\varepsilon)^{n-k}>0.
\end{align}

Finally, we prove that for any $\varepsilon>0$, the inequality
\[
    I_{\psi_{R(\varepsilon)}}(0)
    	- A(n,p,l,\psi _{R(\varepsilon)})
	<2^{\frac{1}{p-1}} \|\psi_{R(\varepsilon)}^{l} \|_{L^1(\mathbb R^n)}
\]
holds. Since $A$ is positive,
and $u_0$ satisfies
\[
        u_0(x)\le C_0(1+|x|)^{-k},\ \ \ \text{for}\ x\in \mathbb{R}^n,
\]
due to the assumption of $u_0$,
by the properties of the function $\psi$, the estimates
\begin{align}%
\label{2-5-1}
	&I_{\psi_{R(\varepsilon)}}(0) - A(n,p,l,\psi _{R(\varepsilon)})
	<\varepsilon\int_{\mathbb{R}^n}u_0(x)\psi_{R(\varepsilon)}^l(x)dx\\
	&\le C_0\varepsilon\int_{|x|<2R(\varepsilon)}(1+|x|)^{-k}dx
	\le \frac{C_0|S_{n-1}|2^{n-k}\varepsilon}{n-k}R(\varepsilon)^{n-k} \notag\\
	&\le 2^{\frac{1}{p-1}} \| \psi^{l} \|_{L^1(\mathbb{R}^n)} R(\varepsilon)^n
	= 2^{\frac{1}{p-1}} \|\psi_{R(\varepsilon)}^{l} \|_{L^1(\mathbb R^n)}\notag
\end{align}
hold for any $\varepsilon>0$,
where we have used the conditions $k<n$ and
\[
         R(\varepsilon)\ge \left\{\frac{C_0|S_{n-1}|2^{n-k}\varepsilon}{(n-k)2^{\frac{1}{p-1}}
         \|\psi^l\|_{L^1}}\right\}^{\frac{1}{k}}.
\]
Therefore, we find that
the function $\psi_{R(\varepsilon)}$ satisfies the condition \eqref{eq:2} with
$\phi=\psi_{R(\varepsilon)}$ for any $\varepsilon>0$.
Thus we can apply Proposition \ref{prop_odi} with
$\phi = \psi_{R(\varepsilon)}$, to obtain the estimate
\begin{align}
\label{2-8-4}
	J_{\psi_{R(\varepsilon)}}(t)
	\ge J_{\psi_{R(\varepsilon)}}(0)
		\left( 1 - \mu(p, A_1(\varepsilon)) \widetilde{J}_{\psi_{R(\varepsilon)}}(0)^{p-1}
			t \right)^{-\frac{2}{p-1}},
\end{align}%
for any $t\in (0,T_2(\varepsilon))$, and the lifespan $T_2(\varepsilon)$ is estimated as
\begin{align}%
\label{2-9-4}
	T_2(\varepsilon) \le  \mu(p,A_1(\varepsilon))^{-1}
		\widetilde{J}_{\psi_{R(\varepsilon)}}(0)^{1-p} ,
\end{align}%
where $A_1(\varepsilon)$ satisfies
\[%
	A_1(\varepsilon)
	:= \frac{I'_{\psi_{R(\varepsilon)}}(0)}{I_{\psi_{R(\varepsilon)}}(0)-A(n,p,l,\psi_{R(\varepsilon)})}
	\ge \frac{c_1}{C_02^{n-k+1}}=:\tilde{A_1},
\]%
for any $\varepsilon>0$, where we have used the estimates (\ref{2-5}) and (\ref{2-5-1}).
We note that $\tilde{A_1}$ is independent of $\varepsilon>0$.
Moreover, by changing variable and the estimate (\ref{2-4-1}), the inequalities
\begin{align}%
\label{2-7-1}
	\widetilde{J}_{\psi_{R(\varepsilon)}}(0)^{p-1}
	&= 2^{-1} \| \psi_{R(\varepsilon)}^{\ell} \|_{L^1(\mathbb{R}^n)}^{-(p-1)}
		J_{\psi_{R(\varepsilon)}}(0)^{p-1} \\
	&\ge 2^{-1} \| \psi^{\ell} \|_{L^1(\mathbb{R}^n)}^{-(p-1)} R(\varepsilon)^{-n(p-1)}
		\left\{\frac{c_0|S_{n-1}|}{4(n-k)}R(\varepsilon)^{n-k}\varepsilon\right\}^{p-1}\notag\\
	&=\left\{\frac{c_0|S_{n-1}|}{2^{\frac{1}{p-1}+2}(n-k)\|\psi^l\|_{L^1}}\right\}^{p-1}
	R(\varepsilon)^{-k(p-1)}\varepsilon^{p-1} \notag
\end{align}%
hold for any $\varepsilon>0$. Here we take sufficiently small
$\varepsilon_2=\varepsilon_2(n,k,l,\psi,c_0,C_0) > 0$
so that
\begin{align}
\label{2-8-1}
     R(\varepsilon)=
     \left\{\frac{4(n-k)A(n,p,l,\psi)}{c_0|S_{n-1}|\varepsilon}\right\}^{\frac{1}{\frac{2}{p-1}-k}}
\end{align}
holds for $\varepsilon\in (0,\varepsilon_2]$.
Thus by combining (\ref{2-7-1}) and (\ref{2-8-1}), the estimate
\begin{equation}
\label{2-9-1}
   \widetilde{J}_{\psi_{R(\varepsilon)}}(0)^{p-1}
   \ge 
   2^{-1+\frac{4}{\frac{2}{p-1}-k}}
   \left(\frac{n-k}{|S_{n-1}|}\right)^{\frac{2(p-1)
   	\left(k-\frac{1}{p-1}\right)}{\frac{2}{p-1}-k}}\|\psi^l\|_{L^1}^{-(p-1)}
		A^{\frac{k(p-1)}{\frac{2}{p-1}-k}}(n,p,l,\psi)
		\varepsilon^{\frac{1}{\frac{1}{p-1}-\frac{k}{2}}}
\end{equation}
holds for any $\varepsilon\in (0,\varepsilon_2]$.
Therefore, by combining the estimates (\ref{2-7-1}) and (\ref{2-9-1}),
we obtain (\ref{bl-rt})
with $\mu_0=\mu(p,\tilde{A_1})$, which completes the proof.
\end{proof}


\appendix
\section{}
In this appendix, we prove that $X(T,M)$ is a closed subset of $Z(T)$ if $T >0$ is finite.
Let
$s \ge 0$, $r \in (1,2]$, $T\in (0,\infty)$.
Since $T$ is finite, we note that
the topology of $X(T)$ with respect to the norm \eqref{X} is the same as
the usual topology of
$L^{\infty}(0,T ; H^s(\mathbb{R}^n)\cap L^r(\mathbb{R}^n))$.
\begin{lemma}
\label{lem_close}
$X(T;M)$ is a closed subset of $Z(T)$.
\end{lemma}
\begin{proof}
First, it is obvious that $X(T;M)\subset Z(T)$.
Therefore, it suffices to show that
for any sequence in $X(T;M)$ converging in $Z(T)$,
its limit belongs to $X(T;M)$.
Let $\{ u_j \}_{j=1}^{\infty} \subset X(T;M)$
converges in $Z(T)$ and let $u$ be its limit.
We note that
\begin{align*}
	L^{\infty}(0,T; H^s(\mathbb{R}^n) \cap L^r(\mathbb{R}^n))
		&= \left( L^1(0,T; H^{-s}(\mathbb{R}^n) + L^{r'}(\mathbb{R}^n)) \right)^*,
\end{align*}
where
$r' = r/(r-1)$.
This and the separability of
$L^1(0,T; H^{-s}(\mathbb{R}^n) + L^{r'}(\mathbb{R}^n))$
(in general, the sum the two separable normed space is separable)
enable us to apply
the sequential Banach--Alaoglu theorem \cite[Theorem 3.16]{Br}
From this theorem and 
\begin{align*}
	\| u_j \|_{X(T)} \le M,
\end{align*}
we can take a subsequence
$\{ u_{j(l)} \}_{l=1}^{\infty}$
and
$v \in L^{\infty}(0,T; H^s(\mathbb{R}^n) \cap L^r(\mathbb{R}^n))$
such that
\begin{align*}
	u_{j(l)} &\overset{\ast}{\rightharpoonup} v
	\quad \mbox{in}\quad L^{\infty}(0,T; H^s(\mathbb{R}^n) \cap L^r(\mathbb{R}^n))
\end{align*}
as $l \to \infty$.
Moreover, we have
\begin{align*}
	\| v\|_{X(T)} &\le \liminf_{l\to \infty} \| u_{j(l)} \|_{X(T)} \le M
\end{align*}
and hence, we see $v \in X(T;M)$.
On the other hand, both
$\{ u_j\}_{j=1}^{\infty}$ and $\{ u_{j(l)} \}_{l=1}^{\infty}$
converge in the space of the distribution
$\mathcal{D}'((0,T)\times \mathbb{R}^n)$
and hence, we obtain
\begin{align*}
	u_{j(l)} &\to v \quad \mbox{in} \quad \mathcal{D}'((0,T)\times \mathbb{R}^n) \quad (l\to \infty),\\
	u_j &\to u \quad \mbox{in} \quad \mathcal{D}'((0,T)\times \mathbb{R}^n) \quad (j\to \infty).
\end{align*}
Thus, by the uniqueness of the limit of distribution implies
$u = v$,
which shows $u \in X(T;M)$.
\end{proof}

\section{}
\label{appendixb}
In the appendix, we prove the existence of the exponents
$q_0$ and $q_j(k)$, $j\in \{1, \ldots, [s]\}$ satisfying \eqref{cond}.
\begin{lemma}\label{lem_a1}
Let $n$ be a positive integer.
Let $s_1,\ldots, s_n > 0$ and $A > 0$ satisfy
$A < \sum_{j=1}^{n} s_j$ (resp. $A \le \sum_{j=1}^{n} s_j$).
Then, there exist $a_1,\ldots, a_n > 0$ such that 
\begin{align*}
	&a_j < s_j \, (\text{resp.}\, a_j \le s_j) \, \text{ for any } j \in\{1,\ldots,n\},
	\\
	&\sum_{j=1}^{n} a_j =A.
\end{align*}
\end{lemma}

\begin{proof}
Let $\theta \in [0,1]$ satisfy
$\theta \sum_{j=1}^{n} s_j = A$,
and we define $a_j := \theta s_j$ for $j=1,\ldots,n$.
Then, we easily see that $\{ a_j \}_{j=1}^n$ has the desired property.
\end{proof}

We apply this lemma to prove the existence of
$q_0$ and $q_j(k) \ (j =1,2, \ldots, [s])$ satisfying \eqref{cond}.

First, we consider the case of $n>2s$. Then, take $q_0$ satisfying 
\begin{align*}
& \frac{r}{p-[s]} \le q_0 
\le \frac{2n}{(p-[s])(n-2s)},
\\
&\sum_{j=1}^{[s]} k_j + \tilde{s} + n\left(\frac{[s]}{2}- \frac{1}{2} + \frac{1}{q_0} \right) \le s[s].
\end{align*}
Namely, in this case, it suffices to take $q_0 = \frac{2n}{(p-[s])(n-2s)}$,
where we note that $p \leq 1+2/(n-2s)$.
Then, noting $\sum_{j=1}^{[s]} k_j= [s]-1$, we have
\begin{align*}
	\frac{[s]}{2}- \frac{1}{2} + \frac{1}{q_0} \le \frac{s - \tilde{s}  -k_1}{n}
	+  \sum_{j=2}^{[s]}  \frac{s -k_j }{n}.
\end{align*}
Applying Lemma \ref{lem_a1} with
$A=[s]/2 -1/2 +1/q_0$, $s_1=(s-\tilde{s}-k_1)/n$, and $s_j = (s-k_j)/n$,
we find $a_1, a_2, \ldots, a_{[s]} > 0$ such that 
\begin{align}
	\label{eq0.1}
	&a_1 < \frac{s-\tilde{s}-k_1}{n},
	\\ \label{eq0.2}
	&a_j < \frac{s-k_j}{n} \text{ for any } j \in\{2,3,\ldots,[s]\},
	\\ \label{eq0.3}
	&\sum_{j=1}^{[s]} a_j =\frac{[s]}{2}- \frac{1}{2} + \frac{1}{q_0}.
\end{align}
Since $n>2s$, we have 
\begin{align*}
	&\frac{s-\tilde{s}-k_1}{n} <\frac{1}{2},
	\\
	&\frac{s-k_j}{n} < \frac{1}{2} \text{ for any } j \in\{1,2,\ldots,[s]\}.
\end{align*}
Therefore, $a_j < 1/2$ for all $j$. We define $q_j(k)$ such that
\begin{align*}
	\frac{1}{q_j(k)} = \frac{1}{2} - a_j.
\end{align*}
Then, $2 < q_j(k) < \infty$. Moreover, we have
\begin{align*}
	&\sum_{j=1}^{[s]} \frac{1}{q_j(k)}= \frac{1}{2} - \frac{1}{q_0},
	\\
	& k_1+\tilde{s}+n \left( \frac{1}{2} - \frac{1}{q_1(k)} \right)
	\le s,
	\\
	&k_j+n \left( \frac{1}{2}-\frac{1}{q_j(k)} \right)
	\le s \text{ for } j =2, 3, \ldots, [s],
\end{align*}
where these come from \eqref{eq0.1}, \eqref{eq0.2}, and \eqref{eq0.3}.

Next, we consider the case of $n \le 2s$. Take $q_0$ satisfying 
\begin{align*}
	&\frac{r}{p-[s]} \le q_0 
	< \infty,
	\\
	&\frac{[s]}{2}- \frac{1}{2} + \frac{1}{q_0} 
	< \min \left\{ \frac{1}{2}, \frac{s - \tilde{s}  -k_1}{n} \right\}
		+  \sum_{j=2}^{[s]}  \min \left\{ \frac{1}{2},\frac{s -k_j }{n}\right\}.
\end{align*}
We show that there exists $q_0$ satisfying the above inequality. The second inequality is equivalent to
\begin{align}
\label{eq0.4}
	\frac{1}{q_0} 
	< \frac{1}{2}+ \min \left\{ 0, \frac{2(s - \tilde{s}  -k_1)-n}{2n} \right\}
		+  \sum_{j=2}^{[s]}  \min \left\{ 0,\frac{2(s -k_j)-n}{2n}\right\}.
\end{align}
Define the set $J$ by
\begin{align*}
	J:=\left\{ j \in \{ 2,3,\ldots,[s]\} ; 2(s -k_j)-n<0 \right\}.
\end{align*}
We denote the number of the elements of $J$ by $\#J$ and $\{ 2,3,\cdots, [s]\} \setminus J$ by $J^c$. 

{\bf Case 1.} We consider the case of $2(s - \tilde{s}  -k_1)-n<0$. Then, we have
\begin{align*}
	\text{(R.H.S of {\eqref{eq0.4}})}
	&= \frac{1}{2}+ \frac{2(s - \tilde{s}  -k_1)-n}{2n} 
		+ \sum_{j \in J} \frac{2(s -k_j)-n}{2n}
	\\
	& = \frac{1}{2} + \frac{2[s] -2k_1-n}{2n}  + \frac{\#J(2s-n) -2\sum_{j \in J} k_j}{2n}
	\\
	& = \frac{2[s] + \#J(2s-n) -2( k_1+\sum_{j \in J} k_j )}{2n}
	\\
	& \ge  \frac{2[s] + \#J(2s-n) -2([s]-1)}{2n}
	\\
	& = \frac{\#J(2s-n) +2}{2n}>0.
\end{align*}
{\bf Case 2-I.} We consider the case of $2(s - \tilde{s}  -k_1)-n \ge 0$ and $\#J \ge 1$. Then, we have
\begin{align*}
	\text{(R.H.S of {\eqref{eq0.4}})}
	&\ge \frac{1}{2}
		+ \sum_{j \in J} \frac{2(s -k_j)-n}{2n}
	\\
	&=\frac{n+\#J(2s-n) - 2 \sum_{j \in J} k_j}{2n}
	\\
	& \ge \frac{n+\#J(2s-n) - 2 ([s]-1)}{2n}
	\\
	& \ge \frac{\#J(2s-n) +n-2s+2}{2n}
	\\
	&= \frac{(\#J-1)(2s-n)+2}{2n} >0 .
\end{align*}
{\bf Case 2-II.} We consider the case of $2(s - \tilde{s}  -k_1)-n \ge 0$ and $\#J =0$.
Then,
\begin{align*}
	\text{(R.H.S of {\eqref{eq0.4}})}
	=\frac{1}{2}.
\end{align*}
Therefore,  (R.H.S of {\eqref{eq0.4}}) is positive and
thus it is enough to take sufficiently large $q_0$. 
Applying Lemma \ref{lem_a1} with
$A=[s]/2 -1/2 +1/q_0$, $s_1=\min\{ 1/2, (s-\tilde{s}-k_1)/n\}$,
and $s_j = \min\{1/2, (s-k_j)/n\}$, we find
$a_1, a_2, \ldots, a_{[s]} > 0$ such that 
\begin{align}
	\label{eq0.5}
	&a_1 < \min \left\{ \frac{1}{2}, \frac{s - \tilde{s}  -k_1}{n} \right\},
	\\ \label{eq0.6}
	&a_j < \min \left\{ \frac{1}{2}, \frac{s  -k_j}{n} \right\} \text{ for any } j \in\{2,3,\cdots,[s]\},
	\\ \label{eq0.7}
	&\sum_{j=1}^{[s]} a_j =\frac{[s]}{2}- \frac{1}{2} + \frac{1}{q_0}.
\end{align}
We define $q_j(k)$ such that
\begin{align*}
	\frac{1}{q_j(k)} = \frac{1}{2} - a_j.
\end{align*}
Then, $2 < q_j(k) < \infty$. Moreover, we have the desired properties
\begin{align*}
	&\sum_{j=1}^{[s]} \frac{1}{q_j(k)}= \frac{1}{2} - \frac{1}{q_0},
	\\
	& k_1+\tilde{s}+n \left( \frac{1}{2} - \frac{1}{q_1(k)} \right)
	\le s,
	\\
	&k_j+n \left( \frac{1}{2}-\frac{1}{q_j(k)} \right)
	\le s \text{ for } j =2, 3, \ldots, [s],
\end{align*}
where these come from \eqref{eq0.5}, \eqref{eq0.6}, and \eqref{eq0.7}.


\section*{Acknowledgments}
This work was supported by JSPS KAKENHI Grant Numbers
JP15K17571, JP16K17624, JP16K17625, JP17J01263. 



\end{document}